\documentclass{article}
\usepackage[utf8]{inputenc}
\usepackage{geometry}
\pdfoutput=1
\usepackage{amsmath}
\usepackage{amssymb}
\usepackage{amsfonts}
\usepackage{caption}
\usepackage{subcaption}
\usepackage{float}
\usepackage{parskip}
\usepackage{amsthm}
\usepackage{thmtools}
\usepackage{csquotes}
\usepackage{mathdots}
\usepackage[dvipsnames]{xcolor}
\usepackage[shortlabels]{enumitem}
\usepackage{authblk}
\usepackage{graphicx}
\graphicspath{ {./documents/} }

\usepackage{tikz}
\usetikzlibrary{calc}
\usetikzlibrary{decorations.markings}

\newtheorem{theorem}{Theorem}[section]
\newtheorem{proposition}[theorem]{Proposition}
\newtheorem{corollary}[theorem]{Corollary}
\newtheorem{lemma}[theorem]{Lemma}

\newtheorem{innercustomthm}{Theorem}
\newenvironment{customthm}[1]
  {\renewcommand\theinnercustomthm{#1}\innercustomthm}
  {\endinnercustomthm}

\newtheorem{innercustomcor}{Corollary}
\newenvironment{customcor}[1]
  {\renewcommand\theinnercustomcor{#1}\innercustomcor}
  {\endinnercustomcor}

\theoremstyle{definition}

\newcommand{\interior}[1]{#1^{\circ}}

\newcommand{\seqj}[3]{(#1_j)_{j=#2}^{#3}}
\newcommand{\seqgen}[4]{(#1_{#2})_{#2 = {#3}}^{#4}}

\newcommand{\slice}[2]{\mathrm{sl}_{#1}(#2)}
\newcommand{\orbs}[2]{\Omega_{#1}(#2)}

\newcommand{\orbsk}[3]{\Omega_{#1}^{#2}(#3)}

\newcommand{\Jfree}{J_{\mathrm{free}}}
\newcommand{\Jfixedone}{J_{\mathrm{fixed},1}}
\newcommand{\Jfixedzero}{J_{\mathrm{fixed},0}}

\newcommand{\qUk}[1]{(2-q)\pi_q(\mathcal{S}^{#1}) +1}
\newcommand{\Uq}{\mathcal{U}_q}
\newcommand{\Uqk}[1]{\pi_q(\mathcal{S}^{#1})}

\newcommand{\mapseqinf}[5]{({#1}_{{#2}_{#3}})_{{#3} = {#4}}^{#5}}
\newcommand{\mapseqn}[5]{{#1}_{{#2}_{#4} \ldots {#2}_{#5}}}
\newcommand{\mapseqnS}[5]{{#1}_{#2^{#3}_{#4} \ldots #2^{#3}_{#5}}}
\newcommand{\mapstring}[5]{({#1}_{{#2}_{#3}})_{{#3}={#4}}^{#5}}
\newcommand{\mapstringS}[6]{({#1}_{#2^#3_#4})_{#4 = #5}^{#6}}

\newcommand{\mapseqinv}[5]{{#1}^{-1}_{{#2}_{#4} \ldots {#2}_{#5}}}
\newcommand{\mapseqinvinf}[5]{({#1}^{-1}_{{#2}_{#3}})_{{#3}={#4}}^{#5}}
\newcommand{\mapseqgen}[6]{
{#1}^{#2}_{{#3}_{#5} \ldots {#3}_{#6}}}
\newcommand{\mapseqgeninf}[6]{
({#1}^{#2}_{{#3}_{#4}})_{{#4} = {#5}}^{#6}
}

\newcommand{\tSk}[1]{\tilde{\mathcal{S}}^{#1}}

\newcommand{\useq}{(u^{-1}_{i_j})_{j=1}^\infty}

\newcommand{\utseq}{(\tilde{u}^{-1}_{i_j})_{j=1}^\infty}

\newcommand{\fseq}{(f_{i_j})_{j=1}^\infty}

\newcommand{\bq}[1]{\hat{f}_{#1}}

\newcommand{\bqseq}[1]{(\hat{f}_{i_j})_{j={#1}}^\infty}

\newcommand{\bqe}[1]{f^*_{#1}}

\newcommand{\bqeseq}[1]{(f^*_{i_j})_{j={#1}}^\infty}

\newcommand{\bqeinv}[1]{f^{*-1}_{#1}}

\newcommand{\hdim}{\dim_{\mathrm{H}}}

\newcommand{\der}{\mathrm{d}}

\newcommand{\okfun}[2]{
    \draw[line width=0.001mm] (0,0) -- (10,10);
    \ifnum #1 > 1 \relax
    \ifnum #1 > 2
    \begin{scope}[color=white, xscale = 0.3333333, yscale = #2]
        \okfun{\numexpr #1-1}{#2}
    \end{scope}
    \begin{scope}[ color = white, shift = {(3.3333,10*#2)}, xscale = 0.3333333, yscale = -(2*(#2)-1)]
        \okfun{\numexpr #1-1}{#2}
    \end{scope}
    \begin{scope}[color = white, shift = {(6.6666667, 10*(1-#2))}, xscale = 0.3333333, yscale = #2]
        \okfun{\numexpr #1-1}{#2}
    \end{scope}
    \fi
    \ifnum #1 =2
    \begin{scope}[color=black, xscale = 0.3333333, yscale = #2]
        \okfun{\numexpr #1-1}{#2}
    \end{scope}
    \begin{scope}[ color= black, shift = {(3.3333,10*#2)}, xscale = 0.3333333, yscale = -(2*(#2)-1)]
        \okfun{\numexpr #1-1}{#2}
    \end{scope}
    \begin{scope}[color = black, shift = {(6.6666667, 10*(1-#2))}, xscale = 0.3333333, yscale = #2]
        \okfun{\numexpr #1-1}{#2}
    \end{scope}
    \fi
    \fi
}

\title{On the cardinality and dimension of the slices of Okamoto's functions}
\author[1]{Simon Baker}
\author[2]{George Bender}
\affil[1]{Department of Mathematical Sciences,
Loughborough University,
Loughborough,
LE11 3TU,
UK, 
simonbaker412@gmail.com}
\affil[2]{School of Mathematics,
University of Birmingham,
Edgbaston,
Birmingham,
B15 2TT,
UK,
gxb119@bham.ac.uk}

\begin{document}
\setlength{\parskip}{2.3ex}

\maketitle

\begin{abstract}
The graphs of Okamoto's functions, denoted by $K_q$, are self-affine fractal curves contained in $[0,1]^2$, parameterised by $q \in (1,2)$.
In this paper we consider the cardinality and dimension of the intersection of these curves with horizontal lines.

Our first theorem proves that if $q$ is sufficiently close to $2$, then $K_q$ admits a horizontal slice with exactly three elements.
Our second theorem proves that if a horizontal slice of $K_q$ contains an uncountable number of elements then it has positive Hausdorff dimension provided $q$ is in a certain subset of $(1,2)$.
Finally, we prove that if $q$ is a $k$-Bonacci number for some $k \in \mathbb{N}_{\geq 3}$, then the set of $y \in [0,1]$ such that the horizontal slice at height $y$ has $(2m+1)$ elements has positive Hausdorff dimension for any $m \in \mathbb{N}$.
We also show that, under the same assumption on $q$, there is some horizontal slice whose cardinality is countably infinite.
\end{abstract}

\section{Introduction}

John Martstrand's 1954 paper \cite{marstrand1954some} provided some of the cornerstone results for the field known as \textit{Fractal Geometry} \cite{falconer2015sixty}.
Marstrand's paper eventually attracted attention primarily for the projection theorem (Theorem \ref{theorem: Marstrand projection}) and slicing theorem (Theorem \ref{theorem: Marstrand slice}) therein.
The original statement of Marstrand's projection theorem imposes the restriction that the planar set $E$ with $\hdim E = s$ is an $s$-set.
That is, $E$ is required to be both measurable and satisfy $0 < \mathcal{H}^s(E) < \infty$.
Davies \cite{DAVIES1952488} proved that analytic sets $E$ with $\hdim E = s$ and satisfying $\mathcal{H}^s(E) = \infty$ contain subsets of arbitrarily large finite $\mathcal{H}^s$-measure, in particular they contain $s$-sets.
This leads to the following version of Marstrand's theorem.

\begin{customthm}{A}[Marstrand's projection theorem]
\label{theorem: Marstrand projection}
    Let $E \subset \mathbb{R}^2$ be analytic with $\hdim E = s$.
    \begin{enumerate}[(a)]
        \item If $s \leq 1$ then the Hausdorff dimension of almost every orthogonal projection of $E$ is equal to $s$.
        \item If $s > 1$ then almost every orthogonal projection of $E$ has positive length (and hence Hausdorff dimension 1).
    \end{enumerate}
\end{customthm}

If $E$ is as above, with $s > 1$, then Marstrand's projection theorem indicates that the projection of $E$ in almost all directions has positive length.
However, the Hausdorff dimension of the projection to a line cannot exceed $1$ so it is natural to seek a finer description via slices.
An interpretation of Theorem \ref{theorem: Marstrand slice} is as a description of how the `surplus dimension' of $E$ is stored in the fibres of the projection.
Formally, Theorem \ref{theorem: Marstrand slice} provides information on the typical dimension of the intersection of $E$ with lines in the plane which pass through $E$.
The original statement of the slicing theorem in Marstrand \cite{marstrand1954some} imposed conditions on the planar set $E$ being an $s$-set, as with the projection theorem above.
As mentioned, by Davies' result on analytic sets \cite{DAVIES1952488}, we are able to reformulate Marstrand's slicing theorem as follows.

\begin{customthm}{B}[Marstrand's slicing theorem]
\label{theorem: Marstrand slice}
    Let $E\subset \mathbb{R}^2$ be analytic with $\hdim E = s > 1$.
    Then almost every line through $\mathcal{H}^s$-almost every point of $E$ intersects $E$ in a set of Hausdorff dimension $s-1$ with finite Hausdorff $(s-1)$-dimensional measure.
\end{customthm}

Davies \cite{DaviesCounterexamples1979} went on to prove that if the planar set $E$ is not assumed to be analytic, then Theorem \ref{theorem: Marstrand projection} fails completely.
Precisely, a planar set $E^*$ is found which projects to a set of Hausdorff dimension zero in all directions and $E^*$ is \textit{essentially $2$-dimensional}\footnote{
Davies \cite{DaviesCounterexamples1979} defines a set to be \textit{essentially at least $s$-dimensional} if it cannot be expressed as the countable union of sets of dimension strictly less than $s$.}.
Although the results we present in this paper are not linked directly to projections, our work is heavily influenced by Marstrand's closely related slicing theorem.
The improvements that have been found for the projection theorem often have analogous consequences for the slicing theorem.
For example, the extension to higher dimensions that Mattila \cite{Mattila1975HausdorffDO} found for the projection theorem carries immediately over to the slicing theorem, as does Davies' result that the set $E$ in question need only be analytic.

The sets of measure zero where the conclusions of the theorems fail are known as exceptional sets.
Given Theorems \ref{theorem: Marstrand projection} and \ref{theorem: Marstrand slice}, natural questions arise regarding the exceptional sets, for instance, we can ask what the Hausdorff and packing dimension of these sets are in general.
The research already published on this topic motivates our own research on the horizontal slices of Okamoto's functions.
In the context of Marstrand's projection theorem, Kaufman proved in \cite{kaufman_1968} that the Hausdorff dimension of the set of exceptional directions of $E$ is no larger than $\hdim E$, and Kaufman and Mattila \cite{KafmanMattila1975} proved that this bound is sharp.
Since the set of directions is an interval, this result is only significant when $\hdim E \leq 1$.
However, in the case $\hdim E > 1$, Falconer \cite{falconer_1982} proved that the exceptional set, i.e. the set of directions whose projections give a set of zero Lebesgue measure, has Hausdorff dimension at most $2-s$.
This bound was also proved to be sharp in the same paper.
In the context of Theorem \ref{theorem: Marstrand slice}, let $\hdim E = s$.
Let $D \subset S^1$ be the set of exceptional directions, where $d \in D$ if the set of points $e \in E$ for which the line through $e$ in direction $d$ intersects $E$ in a set of Hausdorff dimension $s-1$ has $\mathcal{H}^s$-measure zero.
Orponen \cite{OrponenSlicing} proved that $\hdim D \leq 2-s$.
This result is analogous to the results of Kaufman \cite{kaufman_1968} and Falconer \cite{falconer_1982} mentioned above.

A conjecture of Furstenburg \cite{furstenberg1970intersections} concerns the Hausdorff dimension of the intersection of sets $A,B \subset [0,1]$ which are invariant under multiplication maps $T_p: x \mapsto px \mod{1}$ and $T_q: x \mapsto qx \mod{1}$ respectively.
The conjecture states that for all $u,v \in \mathbb{R}$, if $p$ and $q$ are multiplicatively independent (i.e. $\frac{\log p}{\log q}$ is irrational) then
\begin{equation}
\label{equation: Furstenberg's conjecture}
\hdim ((u A + v) \cap B) \leq \max \{0, \hdim A + \hdim B - 1\}.
\end{equation}
We can see for $u,v \in \mathbb{R}$ that $(u A + v) \cap B$ is the intersection of $A \times B$ and the line $y = ux +v$.
This provides a link to our study of slices.
More precisely, Furstenberg's conjecture can be viewed as an improvement on the `almost all' condition of Marstrand's slicing theorem because it gives a bound for the dimension of \textit{all} slices through $A \times B$, under some constraints on $A$ and $B$.
Recently, in 2019, Shmerkin \cite{shmerkin2019furstenberg} and Wu \cite{WuMengFurstenburg} independently and by different methods proved a stronger version of Furstenberg's conjecture, with upper box dimension replacing Hausdorff dimension on the left hand side of \eqref{equation: Furstenberg's conjecture}.

In this paper we study the horizontal intersections of the family of fractal curves known as \textit{Okamoto's functions} \cite{okamoto}.
Although some special cases of this family of curves had been studied before, namely by Bourbaki \cite{spain2013functions} and Perkins \cite{perkins}, research on the whole family of functions was first carried out in 2005 by Okamoto \cite{okamoto}. 

Let $q \in (1,2)$ and consider the following iterated function system on $[0,1]^2$, which we will refer to as the \textit{Okamoto IFS}:
\begin{align*}
S_0(x,y) &= \left(\frac{x}{3} , \frac{y}{q}\right), \\ 
S_1(x,y) &= \left(\frac{x+1}{3} , \left(\frac{2}{q}-1\right)(1-y) + \left(1-\frac{1}{q}\right)\right), \\ 
S_2(x,y) &= \left( \frac{x+2}{3} , \frac{y}{q} + \left(1-\frac{1}{q}\right)\right).
\end{align*}
Each $S_i$ is a contraction on $[0,1]^2$ so by a theorem of Hutchinson \cite{Hutchinson1981}, there is a unique nonempty compact set $K_q$ with the property that
$$K_q = \bigcup_{i \in \{0,1,2\}} S_i(K_q),$$
which is the self-affine set of the Okamoto IFS.
For $i \in \{0,1,2\}$, define $v_i(x)$ and $u_i(y)$ by the equation $S_i(x,y) = (v_i(x), u_i(y))$, that is, they are the coordinate functions of the Okamoto IFS.
For any $y \in [0,1]$ we define the \textit{horizontal slice at} $y$ of $K_q$ to be the set 
$$\slice{q}{y} = \{x \in [0,1] : (x,y) \in K_q\}.$$
For any $q \in (1,2)$ the set $K_q \subset [0,1]^2$ can be realised as the graph of the limit of a sequence of piecewise linear continuous functions on $[0,1]$.
The family of functions $K_q$ for $q \in (1,2)$ (known as Okamoto's functions) was first studied by Okamoto \cite{okamoto}, where the interest originated in their pathological property of being continuous and nowhere differentiable.
In this paper it is shown that if $q \in (1,3/2)$ then $K_q$ is nowhere differentiable, but if $q \in (3/2,2)$ then $K_q$ is differentiable at infinitely many points.
It is worth emphasising that the family of functions studied by Okamoto in \cite{okamoto} is actually larger than the set we consider.
Essentially, they allow the parameter $q$ to range over $(1,\infty)$, however, the limiting functions that appear for $q \in [2,\infty)$ do not have interesting horizontal intersections, so we ignore them here.

Given the $x$-components of the Okamoto IFS are of the form $\frac{x+i}{3}$ for $i \in \{0,1,2\}$, one can expect information on the ternary expansion of $x$ to tell us about the local behaviour of the point $(x,y) \in K_q$.
Allaart \cite{allaart} proved for any $q \in (1,2)$, that the derivative of $K_q$ at $x$ is infinite if and only if the number of $1$s in the ternary expansion of $x$ is finite and some technical limiting condition on $q$ and the ternary expansion of $x$ holds.
Moreover, we have the satisfying polarity property that if the number of $1$s in the ternary expansion of $x$ is even then the derivative is $+\infty$, while if it is negative, the derivative is $-\infty$.
In this paper, Allaart also made use of base $q$ expansion theory to prove results on the Hausdorff dimension of the set of points $x$ for which Okamoto's function has infinite derivative.
We remark here that both ternary and base $q$ expansions also play a key role in our work in this paper on horizontal intersections of $K_q$.
More recently, building on this work, the derivatives of Okamoto's functions with respect to its defining parameter $q$ were studied by Dalaklis et al. \cite{Dalaklis}.
Again, the ternary expansion of $x$ and in particular the limiting behaviour of the number of $1$s in the ternary expansion of $x$ play a key role in the main theorem of this paper.

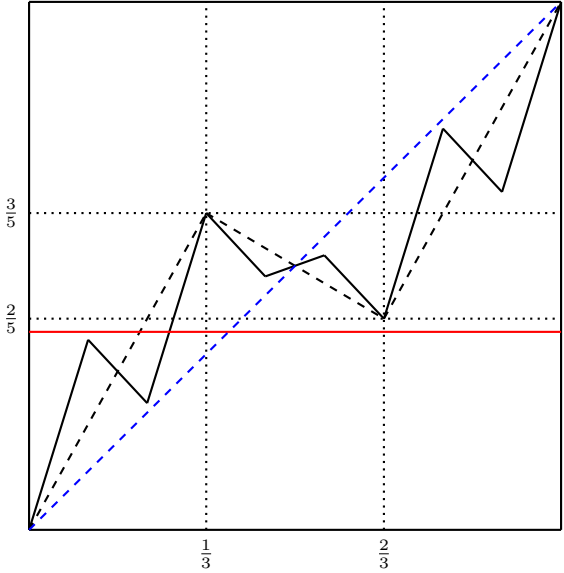
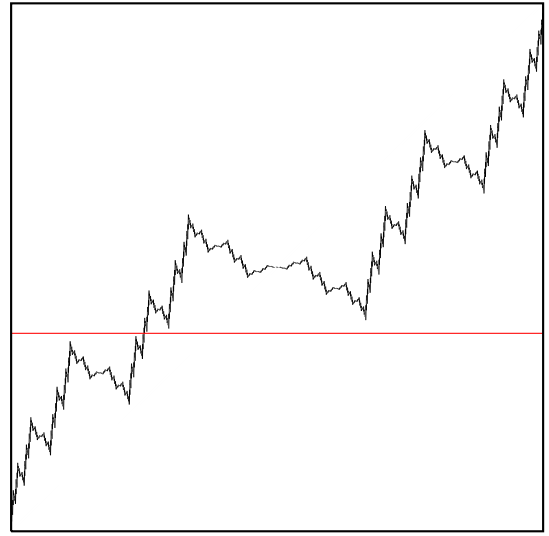
\begin{figure}[t]
    \centering
        \begin{subfigure}[t]{0.45\textwidth}
        \vskip 0pt
            \begin{tikzpicture} [scale = 7]
                \begin{scope}[thick]
                % \begin{scope}[very thick]
                    \draw (0,0) -- (1,0);
                    \draw (0,0) -- (0,1);
                    \draw (1,0) -- (1,1);
                    \draw (0,1) -- (1,1);
                    
                    \draw [dotted] (0.333,0) -- (0.333,1);
                    \draw [dotted] (0.667,0) -- (0.667,1);
                    
                    \draw [dotted] (0,0.6) -- (1,0.6);
                    \draw [dotted] (0,0.4) -- (1,0.4);
                    
                    \draw [dashed](0,0) -- (0.333,0.6);
                    \draw [dashed](0.333,0.6) -- (0.667,0.4);
                    \draw [dashed](0.667,0.4) -- (1,1);
                    
                    \draw (0,0) -- (0.111,0.36);
                    \draw (0.111,0.36) -- (0.222,0.24);
                    \draw (0.222,0.24) -- (0.333,0.6);
                    
                    \draw (0.333,0.6) -- (0.444,0.48);
                    \draw (0.444,0.48) -- (0.555,0.52);
                    \draw (0.555,0.52) -- (0.667,0.4);
                    
                    \draw (0.667,0.4) -- (0.778,0.76);
                    \draw (0.778,0.76) -- (0.889,0.64);
                    \draw (0.889,0.64) -- (1,1);
                    
                    \draw [red] (0,0.375) -- (1,0.375);
                    
                    % \draw [red] (0.21,0.375) -- (0.21,0);
                    
                    \draw [blue, dashed] (0,0) -- (1,1);
                    
                    \node [left] at (0,0.6){\small $\frac{3}{5}$};
                    \node [left] at (0,0.4){\small $\frac{2}{5}$};
                    
                    \node [below] at (0.333,0){\small $\frac{1}{3}$};
                    \node [below] at (0.667,0){\small $\frac{2}{3}$};
                \end{scope}
            \end{tikzpicture}
            \caption{The first two iterations of the IFS construction of the graph of Okamoto's function, $K_q$, where $q = 5/3$.}
        \end{subfigure}
        \hfill
        \begin{subfigure}[t]{0.45\textwidth}
        \vskip 0pt
            \begin{tikzpicture}[scale = 0.7]
                \okfun{8}{0.6};
                \draw[line width = 0.4mm, white] (0,0) -- (4.99,4.99);
                \draw[line width = 0.4mm, white](5.01,5.01) -- (10,10);
                % \draw [line width = 0.001mm, red] (0,4.870856) -- (10,4.870856);
                \draw[red] (0,3.75) -- (10,3.75);
                \draw[thick] (0,0) -- (0,10) -- (10,10) -- (10,0) -- (0,0);
            \end{tikzpicture}
            \caption{The red line $y = 3/8$ has unique intersection with $K_q$, i.e. $\slice{5/3}{3/8}$ has a unique element.}
        \end{subfigure}
        \caption{Constructing $K_q$ for $q = 5/3$.}
        \label{figure: iterations of K_a}
    % \caption{(a) shows the first two steps in the construction of $K_{\frac{3}{5}}$. (b) shows the $F_q$ dynamics for $q = 5/3$. The red line in (a) shows the unique slice at height $y=3/8$ and in (b) it shows the corresponding unique orbit.}
\end{figure}

We will prove in Corollary \ref{corollary: implications of the bijection lemma} that for all $q \in (1,2)$, almost all $y \in [0,1]$ are such that $\slice{q}{y}$ is uncountable.
This result motivates our interest in which parameters $q \in (1,2)$ admit values $y \in [0,1]$ such that $\slice{q}{y}$ is finite or at most countably infinite.
In the context of Theorem \ref{theorem: Marstrand slice}, the set $\slice{q}{y}$ is the intersection of a horizontal line with the planar set $K_q$, so for typical $y$ has Hausdorff dimension $\hdim K_q - 1$.
B\'ar\'any (private communication) has found the following results which link in with the work we present in this paper.
Let $s = 1 + \log_3(\frac{4}{q} -1)$ be the affinity dimension \cite{falconer_1988} of the Okamoto IFS, and let $\dim_{\mathrm{B}}$ and $\dim_{\mathrm{A}}$ denote the box \cite{falconer2004fractal} and Assouad \cite{fraser_2020} dimensions respectively.
Then there is a set $E$ with $\hdim E = 0$ such that for all $ q \in (1,2) \setminus E$, the following are true.
\begin{enumerate}[(a)]
    \item $ \hdim K_q = \dim_{\mathrm{B}} K_q = \dim_{\mathrm{A}} K_q = s$.
    \item For all $y \in [0,1]$, $\hdim (\slice{q}{y}) \leq s-1$.
    \item Lebesgue almost every $y \in [0,1]$ has $\hdim (\slice{q}{y}) = s-1$.
\end{enumerate}
The fact that $\slice{q}{y}$ is typically uncountable tells us that the set of $y$ for which $\slice{q}{y}$ is finite or countable is anomalous.
Despite the questions of whether $K_q$ is an $s$-set for any $s$ being open, it follows easily from the definition of $K_q$ that all projections of $K_q$ have Hausdorff dimension $1$ and indeed are all intervals.

In Subsection \ref{subsection: bijection lemma} we prove the existence of a bijection between $\slice{q}{y}$ and the set of allowable sequences of maps in a well-understood system.
An important property of this bijection is that if $\mapseqinf{f}{i}{j}{1}{\infty}$ is one of the aforementioned sequences of maps, then the sequence $\seqj{i}{1}{\infty}$ is the ternary expansion of some $x \in \slice{q}{y}$.
Moreover, this bijection allows us to restate our theorems on $\slice{q}{y}$ in a form that allows for the introduction of techniques and ideas from base $q$ expansions (also known as $\beta$ expansions).

Let $k \in \mathbb{N}$. A generalisation of the golden ratio, $G \approx 1.618$, which satisfies $G^2 - G - 1 =0$ is the real number $q_k \in (1,2)$ we call the \textit{$k$-Bonacci number} which satisfies
\begin{equation}
\label{equation: k-Bonacci definition}
   q^k - q^{k-1} - \cdots - q - 1 = 0.
\end{equation}
We note that $q_2 = G$ is the golden ratio and $q_1 = 1$ is ignored.

Our main results are stated below.

\begin{customthm}{1}
\label{theorem: interval with slice three}
    If $q \in (q_9,2) = (1.99803..., 2)$ then there is some $y \in [0,1]$ such that $|\slice{q}{y}| = 3$.
\end{customthm}

\begin{customthm}{2}
\label{theorem: positive Haudsorff dimension}
    Let $q \in (1,2)$ be such that for all $y \in [0,1]$, $|\slice{q}{y}| \in \{1,2^{\aleph_0}\}$.
    Then if $\slice{q}{y}$ is uncountable, there exists an $s > 0$ depending only on $q$ such that $\hdim(\slice{q}{y}) \geq s$.
\end{customthm}

Let $\mathcal{T}$ be the set of transcendental numbers and let $q_{\aleph_0} \approx 1.64541$ be the root of $x^6 = x^4 + x^3 + 2x^2 + x + 1$ in $(1,2)$.
It was shown by the first author in \cite{BAKER2015515} that $q_{\aleph_0}$ is the smallest $q \in (G,2)$ with the property that there exists $x \in I_q$ with countably many base $q$ expansions.
Let $q_{\mathrm{KL}}$ be the Komornik-Loreti constant \cite{KomornikLoretiUniqueDevelopments}, defined to be the smallest base $q$ for which $1$ has a unique base $q$ expansion.
We will prove that Corollary \ref{corollary: positive Hausdorff dimension} follows from Theorem \ref{theorem: positive Haudsorff dimension} in Section \ref{section: proof of positive dimension}, before proving an equivalent version of Theorem \ref{theorem: positive Haudsorff dimension}.

\begin{customcor}{2}
\label{corollary: positive Hausdorff dimension}
    Let $q \in (1, q_{\aleph_0}) \setminus \{G\}$ or $q \in \mathcal{T} \cap (q_{\aleph_0}, q_\mathrm{KL})$.
    Then if $\slice{q}{y}$ is uncountable there is some $s > 0$ depending only on $q$ such that $\hdim(\slice{q}{y}) \geq s$.
\end{customcor}

In the context of base $q$ expansions, Sidorov \cite[Theorem 2.1]{sidorov2009expansions} proved the following dichotomy.
If $q \in \mathcal{T} \cap (G, q_{\mathrm{KL}})$ then any $x \in [0,\frac{1}{q-1}]$ has either a unique base $q$ expansion or uncountably many of them.
Corollary \ref{corollary: positive Hausdorff dimension} strengthens this result in the sense that if $q$ is restricted to the same set, then we know that for any $y \in [0,1]$, $\slice{q}{y}$ has either a unique element or it has positive Hausdorff dimension.

\begin{customthm}{3}
\label{theorem: k-Bonacci numbers admit any odd slice}
    Let $\{q_i\}_{i=3}^\infty$ be the set of $k$-Bonacci numbers excluding $G$.
    If $q = q_i$ for some $i \in \{3,4, \ldots\}$ then the following are true.
    \begin{enumerate}
        \item There exists an $\epsilon > 0$ such that if 
        $$Y_m = \{ y \in [0,1] : |\slice{q}{y}| = 2m+1 \},$$
        then for all $m \in \mathbb{N}$, $\hdim(Y_m)>\epsilon$.
        \item There exists $y_{\aleph_0} \in [0,1]$ such that the cardinality of $\slice{q}{y_{\aleph_0}}$ is countably infinite.
    \end{enumerate}
\end{customthm}

In Section \ref{section: interval with slice three} we prove Theorem \ref{theorem: interval with slice three}.
In Section \ref{section: proof of positive dimension} we prove Theorem \ref{theorem: positive Haudsorff dimension} and its Corollary \ref{corollary: positive Hausdorff dimension}.
In Section \ref{section: k-Bonacci results} we prove Theorem \ref{theorem: k-Bonacci numbers admit any odd slice} along with the result that 
there is some $y \in [0,1]$ such that $\slice{q}{y}$ has two elements 
if and only if
$1$ has a unique base $q$ expansion 
(Theorem \ref{theorem: U = C_2}).

\section{Background and preliminaries}
\label{section: background and preliminaries}

In this section we present the necessary background theory required for the proofs of Theorems \ref{theorem: interval with slice three}, \ref{theorem: positive Haudsorff dimension} and \ref{theorem: k-Bonacci numbers admit any odd slice}. In particular we prove Lemma \ref{lemma: slice orbit bijection} which provides an important correspondence allowing us to effectively study both the cardinality and dimension of $\slice{q}{y}$.

\subsection{Background theory}
\label{subsection: background theory}

Let $q \in (1,2)$ and let $K_q$ be the graph of Okamoto's function with parameter $q$ as defined above.
Define $\{0,1,2\}^* = \cup_{k \in \mathbb{N}_{\geq 0}} \{0,1,2\}^k$ and $\{0,1\}^* = \cup_{k \in \mathbb{N}_{\geq 0}} \{0,1\}^k$.
The following proposition is a consequence of \cite[Theorem 9.1]{falconer2004fractal}.

\begin{proposition}
\label{proposition: Okamoto IFS sequences}
A point $(x,y) \in K_q$ if and only if there is a sequence of maps of the Okamoto IFS $(S_{i_j})_{j=1}^\infty \in \{S_0, S_1, S_2\}^\mathbb{N}$ such that for all $k \in \mathbb{N}_{\geq 0},$
\begin{equation}
\label{equation: Okamoto IFS}
(x,y) \in S_{i_1} \circ S_{i_2} \circ \cdots \circ S_{i_k}([0,1]^2),
\end{equation}
which is equivalent to both
\begin{equation}
\label{equation: Okamoto IFS x-coordinate}
x \in v_{i_1} \circ v_{i_2} \circ \cdots \circ v_{i_k}([0,1]),
\end{equation}
and
\begin{equation}
\label{equation: Okamoto IFS y-coordinate}
y \in u_{i_1} \circ u_{i_2} \circ \cdots \circ u_{i_k}([0,1]),
\end{equation}
holding for all $ k \in \mathbb{N}_{\geq 0}$.
\end{proposition}

Note that for $k = 0$, the sequence $\seqj{i}{1}{k} \in \{0,1,2\}^*$ is the empty word and the corresponding sequence of maps $S_{i_1} \circ \cdots \circ S_{i_k}$ is just the identity map.
Therefore, in the case $k=0$, the proposition states that $(x,y) \in [0,1]^2$.

For each $q \in (1,2)$, let $I_q = [0,\frac{1}{q-1}]$ and $J_q = [\frac{1}{q}, \frac{1}{q(q-1)}]$. 
In the literature on base $q$ expansions (e.g. \cite{sidorov2009expansions}) $J_q$ is commonly referred to as the \textit{switch region}.
The map $\pi_3 : \{0,1,2\}^\mathbb{N} \rightarrow [0,1]$ is defined by
$$\pi_3(\seqj{i}{1}{\infty}) = \sum_{j=1}^\infty i_j 3^{-j},$$
and for $q \in (1,2)$, the map $\pi_q: \{0,1\}^\mathbb{N} \rightarrow I_q$ is defined by
$$\pi_q(\seqj{i}{1}{\infty}) = \sum_{j=1}^\infty i_j q^{-j}.$$
These are the \textit{projection maps}.
The domain of $\pi_q$ can be extended to $\{-1,0,1\}^\mathbb{N}$ in the obvious way and in this case the codomain is $I_q^* = [-\frac{1}{q-1}, \frac{1}{q-1}]$. We will need this extension of $\pi_q$ in Section \ref{section: interval with slice three}.

Given a finite sequence $\seqj{a}{1}{k} \in \{0,1\}^*$, we define the associated \textit{cylinder} by
$$ [\seqj{a}{1}{k}] = \{\seqj{i}{1}{\infty} \in \{0,1\}^\mathbb{N} : i_j = a_j \ \mathrm{for} \ 1 \leq j \leq k \}. $$
Similarly, if $\seqj{a}{1}{k} \in \{0,1,2\}^*$ then the associated cylinder is defined to be
$$ [\seqj{a}{1}{k}] = \{\seqj{i}{1}{\infty} \in \{0,1,2\}^\mathbb{N} : i_j = a_j \ \mathrm{for} \ 1 \leq j \leq k \}. $$
If $k=0$ we define the empty cylinder to be the whole sequence space: $ [e] = \{0,1\}^\mathbb{N}$ or $[e] = \{0,1,2\}^\mathbb{N}$ where $e$ represents the empty word.
It will be clear from context which of $\{0,1\}^\mathbb{N}$ and $\{0,1,2\}^\mathbb{N}$ $[e]$ represents.
The following lemma is a simple consequence of the definition of the projection map $\pi_q$ and the existence of at least one base $q$ expansion for any point in $I_q$, which Parry showed in his seminal paper \cite{Parry1960OnTO}.

\begin{lemma}
\label{lemma: projections of cylinders are intervals}
    If $\seqj{a}{1}{k} \in \{0,1\}^*$ is arbitrary then $\pi_q[\seqj{a}{1}{k}]$ is the interval $[\pi_q(\seqj{a}{1}{k}0^\infty), \pi_q(\seqj{a}{1}{k}1^\infty)]
    = [\sum_{j=1}^k a_j q^{-j} , \sum_{j=1}^k a_j q^{-j} + q^{-k}(\frac{1}{q-1})]$.
\end{lemma}

Let $(a_j)_{j=1}^k, \seqj{b}{1}{l} \in \{0,1,2\}^*$ be two finite ternary sequences. 
We say that $(a_j)_{j=1}^k $ is a \textit{prefix} of $\seqj{b}{1}{l}$ if $k \leq l$ and $a_j = b_j$ for all $0 \leq j \leq k$. 
The sequence $(a_j)_{j=1}^k$ is a \textit{strict prefix} of $\seqj{b}{1}{l}$ if it is a prefix and they are not equal as sequences. 
Elements of $\{0,1,2\}^\mathbb{N}$ can be prefixed by elements of $\{0,1,2\}^*$ in the expected way and the definitions hold for elements of $\{0,1\}^*$ too.

The proof of the following lemma is straightforward and omitted.

\begin{lemma}
\label{lemma: subsequences are interval containments}
    Let $\seqj{a}{1}{k} , \seqj{b}{1}{l} \in \{0,1\}^*$ with $\seqj{a}{1}{k}$ a prefix of $ \seqj{b}{1}{l}$.
    Then $\pi_q[\seqj{b}{1}{l}] \subset \pi_q[\seqj{a}{1}{k}]$.
    The same conclusion holds for elements of $\{0,1,2\}^*$.
\end{lemma}

For $q \in (1,2)$, we define the set of maps $ E_q = \{f_0 , f_1 , f_2 \}$ (Figure \ref{figure: Okamoto dynamics}) where $f_0, f_1$ and $f_2$ are given by
\begin{align*}
f_0 &: \left[0,\frac{1}{q(q-1)}\right) \rightarrow \left[0,\frac{1}{q-1}\right);   &f_0(x) &= qx,\\
f_1 &: \left( \frac{1}{q}, \frac{1}{q(q-1)}\right] \rightarrow \left[0, \frac{1}{q-1}\right);   &f_1(x) &= \frac{1}{q-1} - \frac{1}{2-q}(qx-1),\\
f_2 &: \left[\frac{1}{q},\frac{1}{q-1}\right] \rightarrow \left[0, \frac{1}{q-1}\right];   &f_2(x) &= qx - 1.
\end{align*}

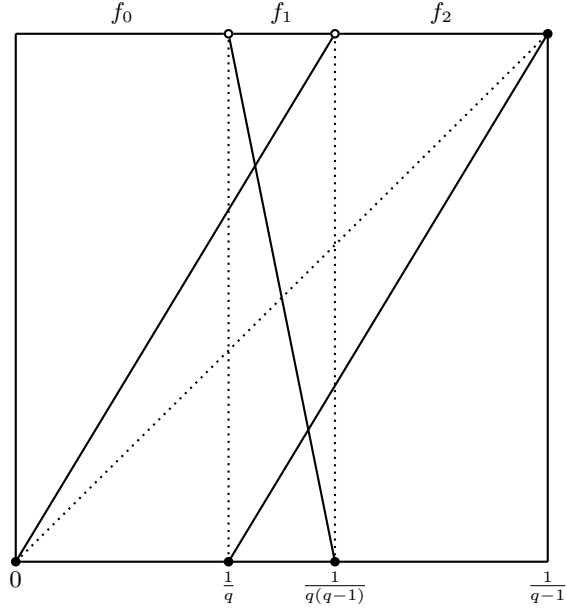
\begin{figure}[t]
    \centering
            \begin{tikzpicture} [scale = 7]
    \begin{scope}[thick]
        \draw (0,0) -- (1,0);
        \draw (0,0) -- (0,1);
        \draw (1,0) -- (1,1);
        \draw (0,1) -- (1,1);
        
        \draw [dotted] (0,0) -- (1,1);
        
        \draw (0,0) -- (0.6,1);
        \draw (0.4,0) -- (1,1);
        \draw (0.6,0) -- (0.4,1);
        
        \draw [dotted] (0.6,1) -- (0.6,0);
        \draw [dotted] (0.4,0) -- (0.4,1);

    \node[above] at (0.2,1){\small $f_0$};
    \node[above] at (0.5,1){\small $f_1$};
    \node[above] at (0.8,1){\small $f_2$};

    \node[below] at (0,0){\small $0$};
    \node[below] at (0.4,0){\small $\frac{1}{q}$};
    \node[below] at (0.6,0){\small $\frac{1}{q(q-1)}$};
    \node[below] at (1,0){\small $\frac{1}{q-1}$};

    \filldraw[white] (0.4,1) circle (0.2pt);
    \filldraw[white] (0.6,1) circle (0.2pt);
    \draw (0.4,1) circle (0.2pt);
    \draw (0.6,1) circle (0.2pt);
    \filldraw[black] (0.4,0) circle (0.2pt);
    \filldraw[black] (0.6,0) circle (0.2pt);
    \filldraw[black] (0,0) circle (0.2pt);
    \filldraw[black] (1,1) circle (0.2pt);
    \end{scope}
    
        \end{tikzpicture}
        \caption{The set of maps $E_q = \{f_0, f_1, f_2\}$ on $I_q$.}
        \label{figure: Okamoto dynamics}
\end{figure}

For any $\seqj{i}{1}{k} \in \{0,1,2\}^*$ let $f_{i_k} \circ \cdots \circ f_{i_1} = \mapseqn{f}{i}{j}{1}{k}$ and write $\fseq$ for an arbitrary element of $\{f_{0}, f_{1}, f_{2} \}^\mathbb{N}$.
In particular if $k = 0$ then $\mapseqn{f}{i}{j}{1}{k}$ is defined to be the identity map.
For any $x \in I_q$, we define the \textit{orbit space} of $x$ in $E_q$ by
$$\Omega_{E_q}(x) = \{\fseq \in \{f_0, f_1, f_2 \}^\mathbb{N} : \mapseqn{f}{i}{j}{1}{k}(x) \in I_q \ \forall k \in \mathbb{N}_{\geq 0} \}.$$
We note here that implicit in the definition of the orbit space is the fact that $\mapseqn{f}{i}{j}{1}{k}(x) \in I_q$ only makes sense when $\mapseqn{f}{i}{j}{1}{l-1}(x)$ is in the domain of $f_{i_l}$ for all $1 \leq l \leq k$.
If this property fails for some $k \in \mathbb{N}$ then $\fseq$ is not an element of $\orbs{E_q}{x}$.
The half open intervals that form the domains of $f_0$ and $f_1$ prevent points other than $\frac{1}{q-1}$ from mapping to $\frac{1}{q-1}$, which restricts the possible sequences of maps in the orbit space as we will see in Subsection \ref{subsection: bijection lemma}.

\begin{lemma}
\label{lemma: unique orbit if avoids switch}
    Let $x \in I_q$, then $\fseq \in \orbs{E_q}{x}$ is unique if and only if $\mapseqn{f}{i}{j}{1}{k}(x) \in I_q \setminus J_q$ for all $k \in \mathbb{N}_{\geq 0}$.
\end{lemma}

\begin{proof}
    Let $x \in I_q$. 
    The sequence $\fseq \in \orbs{E_q}{x}$ is unique if and only if the following condition holds.
    \begin{equation}
    \label{condition: unique orbit}
        \tag{A}
        \parbox{\dimexpr\linewidth-4em}{%
    \strut
   For all $k \in \mathbb{N}_{\geq 0}$, there is a unique choice of $i_{k+1} \in \{0,1,2\}$ such that $\mapseqn{f}{i}{j}{1}{k}(x)$ is in the domain of $f_{i_{k+1}}$.
    \strut
  }
    \end{equation}
    By inspection of the domains of the maps in $E_q$, condition \eqref{condition: unique orbit} is equivalent to $\mapseqn{f}{i}{j}{1}{k}(x) \in I_q \setminus J_q$ for all $k \in \mathbb{N}_{\geq 0}$ and hence the lemma.
\end{proof}

\subsection{Bijection lemma}
\label{subsection: bijection lemma}

Let $x \in [0,1]$. We define the \textit{restricted ternary expansion} (RTE) of $x$ to be the unique sequence $\seqj{i}{1}{\infty} \in \{0,1,2\}^\mathbb{N}$ with $\pi_3(\seqj{i}{1}{\infty}) = x$ such that $\seqj{i}{k}{\infty} \neq 2^\infty$ for any $k \in \mathbb{N}$ unless $x = 1$ in which case $\seqj{i}{1}{\infty} = 2^\infty$.
By observing that $v_i(x) = \frac{x + i}{3}$ for $i \in \{0,1,2\}$, it is clear that given $x \in [0,1]$, $\seqj{i}{1}{\infty} \in \{0,1,2\}^\mathbb{N}$ satisfies \eqref{equation: Okamoto IFS x-coordinate}
for all $k \in \mathbb{N}_{\geq 0}$ if and only if $x = \pi_3(\seqj{i}{1}{\infty})$.
The following lemma allows us to transpose questions about $\slice{q}{y}$ into questions about $\orbs{E_q}{\frac{y}{q-1}}$.
In Lemma \ref{lemma: orbit space inequality}, we prove an inequality between the cardinality of $\orbs{E_q}{x}$ and the number of base $q$ expansions of $x$ for any $x \in I_q$.
Hence, this transposition to working with $\orbs{E_q}{\frac{y}{q-1}}$ allows us to use existing work on base $q$ expansion theory.
The details of this are contained in Subsection \ref{subsection: implications of the bijection lemma}.

\begin{lemma}
\label{lemma: slice orbit bijection}
    Let $q \in (1,2)$ and $y \in [0,1]$, then there is a bijection between $\slice{q}{y}$ and $\orbs{E_q}{\frac{y}{q-1}}$ given by $x \leftrightarrow \fseq$ where $\seqj{i}{1}{\infty}$ is the RTE of $x$.
\end{lemma}

The proof of the lemma concerns the following two closely related sets of maps on $[0,1]$.
Define $F_q = \{u_0^{-1}, u_1^{-1}, u_2^{-1}\}$ by
\begin{align*}
u_0^{-1}&:[0,1/q] \rightarrow [0,1];   &u_0^{-1}(y) &= qy,\\
u_1^{-1}&:[1-1/q,1/q] \rightarrow [0,1];   &u_1^{-1}(y) &= \frac{1-qy}{2-q},\\
u_2^{-1}&:[1-1/q,1] \rightarrow [0,1];   &u_2^{-1}(y) &= qy+(1-q).
\end{align*}
The corresponding orbit space is
$$\orbs{F_q}{y} = \{(u_{i_j}^{-1})_{j=1}^\infty \in \{u_0^{-1}, u_1^{-1}, u_2^{-1}\}^\mathbb{N} : \mapseqinv{u}{i}{j}{1}{k}(y) \in [0,1] \mathrm{\ for \ all \ }k \in \mathbb{N}_{\geq 0}\},$$
and has the property that for a given $y \in [0,1]$, $\seqj{i}{1}{\infty}$ solves \eqref{equation: Okamoto IFS y-coordinate} for all $k \in \mathbb{N}_{\geq 0}$ if and only if
$(u_{i_j}^{-1})_{j=1}^\infty \in \orbs{F_q}{y}$.

Let $\tilde{F}_q = \{\tilde{u}_0^{-1}, \tilde{u}_1^{-1}, \tilde{u}_2^{-1} \}$ where we make the requirement that the maps $\tilde{u}_0^{-1}$ and $\tilde{u}_1^{-1}$ are defined on half-open intervals but are otherwise identical to the corresponding maps in $F_q$:
\begin{align*}
\tilde{u}_0^{-1}&:[0,1/q) \rightarrow [0,1);   &\tilde{u}_0^{-1}(y) &= qy,\\
\tilde{u}_1^{-1}&: (1-1/q,1/q] \rightarrow [0,1); &\tilde{u}_1^{-1}(y) &= \frac{1-qy}{2-q},\\
\tilde{u}_2^{-1}&: [1-1/q,1] \rightarrow [0,1]; &\tilde{u}_2^{-1}(y) &= qy + (1-q).
\end{align*}
The orbit space of $\tilde{F}_q$ is
$$\orbs{\tilde{F}_q}{y} = \{\utseq \in \{\tilde{u}_0^{-1}, \tilde{u}_1^{-1}, \tilde{u}_2^{-1}\}^\mathbb{N} : \mapseqinv{\tilde{u}}{i}{j}{1}{k}(y) \in [0,1] \mathrm{\ for \ all \ }k \in \mathbb{N}_{\geq 0}\}.$$
The purpose of making the restriction from $F_q$ to $\tilde{F}_q$ is to guarantee that if $\seqj{i}{1}{\infty} \in \{0,1,2\}^\mathbb{N}$ is not an RTE then $\mapseqinf{\tilde{u}^{-1}}{i}{j}{1}{\infty}$ is not an element of $\orbs{\tilde{F}_q}{y}$ for any $y \in [0,1]$.

\begin{proof}[Proof of Lemma \ref{lemma: slice orbit bijection}]
We first show there is a bijection $\slice{q}{y} \leftrightarrow \orbs{\tilde{F}_q}{y}$ given by $x \leftrightarrow \utseq$ where $\seqj{i}{1}{\infty}$ is the RTE of $x$, before proving that $\utseq \leftrightarrow \fseq$ is a bijection between $\orbs{\tilde{F}_q}{y}$ and $\orbs{E_q}{\frac{y}{q-1}}$ for all $y \in [0,1]$.

Let $q \in (1,2)$, let $x \in \slice{q}{y}$ for some $y \in [0,1]$ and let $\seqj{i}{1}{\infty}$ be the RTE of $x$. 
By Proposition \ref{proposition: Okamoto IFS sequences}, $x \in \slice{q}{y}$ if and only if $\mapseqinv{u}{i}{j}{1}{k}(y) \in [0,1]$ for all $k \in \mathbb{N}_{\geq 0}$, i.e. $\mapseqinvinf{u}{i}{j}{1}{\infty} \in \orbs{F_q}{y}$.

Suppose $y = 1$, if $\mapseqinvinf{u}{i}{j}{1}{\infty} \in \orbs{F_q}{1}$ then $\seqj{i}{1}{\infty} = 2^\infty$, so $\mapseqinvinf{\tilde{u}}{i}{j}{1}{\infty} \in \orbs{\tilde{F}_q}{1}$.
If $y \in [0,1)$ then since $\seqj{i}{1}{\infty}$ is an RTE, it avoids the tail $2^\infty$.
Therefore $\mapseqinv{u}{i}{j}{1}{k}(y) \neq 1$ for any $k \in \mathbb{N}_{\geq 0}$.
To see what the consequences of this are, we consider the inverse images of $1$.
There are three cases corresponding to the three distinct maps of $F_q$.
\begin{enumerate}
    \item $u^{-1}_0(1/q) = 1$.
    \item $u^{-1}_1(1-1/q) = 1$.
    \item $u^{-1}_2(1) = 1$.
\end{enumerate}
We can ignore the final case since we know $\mapseqinv{u}{i}{j}{1}{k}(y) \neq 1$ for any $k \in \mathbb{N}_{\geq 0}$.
The first case means that if $\mapseqinv{u}{i}{j}{1}{k}(y) = 1/q$ for some $k \in \mathbb{N}_{\geq 0}$ then $i_{k+1} \neq 0$.
If this were the case then $\seqj{i}{1}{\infty}$ would have tail $02^{\infty}$, contradicting the fact that it is an RTE.
Similarly, in the second case, if $\mapseqinv{u}{i}{j}{1}{k}(y) = 1-1/q$ for some $k \in \mathbb{N}_{\geq 0}$ then $i_{k+1} \neq 1$ to avoid $\seqj{i}{1}{\infty}$ having the tail $12^\infty$.
Therefore, this restriction is equivalent to removing $1/q$ from the domain of $u^{-1}_0$ and removing $1-1/q$ from the domain of $u^{-1}_1$.
One can observe that this is exactly the restriction imposed on $F_q$ to transform it into $\tilde{F}_q$.
Hence if $y \in [0,1)$ then $\mapseqinvinf{\tilde{u}}{i}{j}{1}{\infty} \in \orbs{\tilde{F}_q}{y}$.
We conclude that $\mapseqinvinf{\tilde{u}}{i}{j}{1}{\infty} \in \orbs{\tilde{F}_q}{y}$ for all $y \in [0,1]$ and that we have a map $\slice{q}{y} \rightarrow \orbs{\tilde{F}_q}{y}$ given by $x \mapsto \utseq$ where $\seqj{i}{1}{\infty}$ is the RTE of $x$.

Let $\utseq \in \orbs{\tilde{F}_q}{y}$ for some point $y \in [0,1]$.
Because $1$ is not in the image of $\tilde{u}^{-1}_0$ or $\tilde{u}^{-1}_1$, $\utseq$ has the property that $\mapseqinv{\tilde{u}}{i}{j}{1}{k}(y) \neq 1$ for any $k \in \mathbb{N}_{\geq 0}$ unless $y=1$.
If $y = 1$ then $\seqj{i}{1}{\infty} = 2^\infty$ and if $y \in [0,1)$, there is no $k \in \mathbb{N}$ such that $\seqj{i}{k}{\infty} = 2^\infty$.
In either case $\seqj{i}{1}{\infty}$ is an RTE.
If $\utseq \in \orbs{\tilde{F}_q}{y}$ then $\useq \in \orbs{F_q}{y}$ since the domain of $\tilde{u}^{-1}_i$ is contained in the domain of $u^{-1}_i$ for $i \in \{0,1,2\}$.
Then, by Proposition \ref{proposition: Okamoto IFS sequences}, $\seqj{i}{1}{\infty}$ is the RTE of some $x \in \slice{q}{y}$.
This proves that there is a bijection $\slice{q}{y} \leftrightarrow \orbs{\tilde{F}_q}{y}$ given by $x \leftrightarrow \utseq$ where $\seqj{i}{1}{\infty}$ is the RTE of $x$.

It remains to show that the map $\utseq \leftrightarrow \fseq$ is a bijection between $\orbs{\tilde{F}_q}{y}$ and $\orbs{E_q}{\frac{y}{q-1}}$ for all $y \in [0,1]$.
For each $i \in \{0,1,2\}$, and for any $y \in [0,1]$, it can be checked that
$$\frac{1}{q-1}\tilde{u}^{-1}_i(y) = f_i\left(\frac{y}{q-1}\right).$$
Moreover, for all $i \in \{0,1,2\}$, if the domain of $f_i$ is $D_i$ and the domain of $\tilde{u}^{-1}_i$ is $C_i$ then $D_i = \frac{1}{q-1}C_i$.
Hence, for all $k \in \mathbb{N}_{\geq 0}$ and any $y \in [0,1]$,
$$ \mapseqinv{\tilde{u}}{i}{j}{1}{k}(y) \in [0,1] \iff \frac{1}{q-1}\left[\mapseqinv{\tilde{u}}{i}{j}{1}{k}(y)\right] \in I_q \iff \mapseqn{f}{i}{j}{1}{k}\left(\frac{y}{q-1}\right) \in I_q.$$
So $\utseq \in \orbs{\tilde{F}_q}{y} \iff \fseq \in \orbs{E_q}{y}$.
Therefore, given $q \in (1,2)$ and any $y \in [0,1]$, the map $\slice{q}{y} \leftrightarrow \orbs{E_q}{\frac{y}{q-1}}$ given by $x \leftrightarrow \fseq$ where $\seqj{i}{1}{\infty}$ is the RTE of $x$, is a bijection.
\end{proof}

\subsection{Base $q$ expansions}
\label{subsection: base q expansions}
In this subsection we introduce base $q$ expansions \cite{Parry1960OnTO} and state Theorem \ref{theorem: basic results on base-q expansions} which, alongside Lemma \ref{lemma: slice orbit bijection}, puts $\slice{q}{y}$ in the context of base $q$ expansions.

Let $q \in (1,2)$ and $x \in I_q$ then a base $q$ expansion of $x$ is a sequence $\seqj{i}{1}{\infty} \in \{0,1\}^\mathbb{N}$ such that $x = \sum_{j=1}^\infty i_j q^{-j}$, that is, such that $x = \pi_q(\seqj{i}{1}{\infty})$.
We define the \textit{base $q$ dynamics} (Figure \ref{figure: Base-q dynamics}) to be the set of maps $ \hat{E}_q = \{\hat{f}_0 , \hat{f}_1 \}$ where \begin{align*}
\bq{0}&:\left[0,\frac{1}{q(q-1)}\right] \rightarrow I_q; &\bq{0}(x) &= qx, \\
\bq{1}&:\left[\frac{1}{q}, \frac{1}{q-1}\right] \rightarrow I_q; &\bq{1}(x) &= qx - 1.
\end{align*}

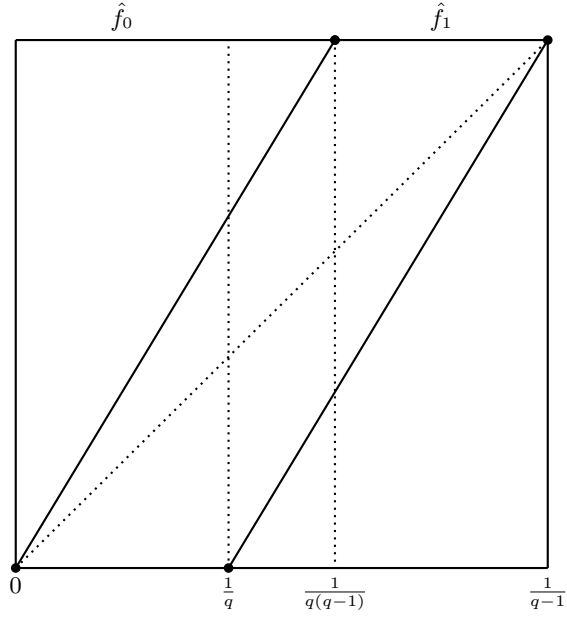
\begin{figure}[t]
    \centering
            \begin{tikzpicture} [scale = 7]
    \begin{scope}[thick]
        \draw (0,0) -- (1,0);
        \draw (0,0) -- (0,1);
        \draw (1,0) -- (1,1);
        \draw (0,1) -- (1,1);
        
        \draw [dotted] (0,0) -- (1,1);
        
        \draw (0,0) -- (0.6,1);
        \draw (0.4,0) -- (1,1);
        % \draw (0.6,0) -- (0.4,1);
        
        \draw [dotted] (0.6,1) -- (0.6,0);
        \draw [dotted] (0.4,0) -- (0.4,1);

    \node[above] at (0.2,1){\small $\hat{f}_0$};
    % \node[above] at (0.5,1){\small $f_1$};
    \node[above] at (0.8,1){\small $\hat{f}_1$};

    \node[below] at (0,0){\small $0$};
    \node[below] at (0.4,0){\small $\frac{1}{q}$};
    \node[below] at (0.6,0){\small $\frac{1}{q(q-1)}$};
    \node[below] at (1,0){\small $\frac{1}{q-1}$};

    % \filldraw[white] (0.4,1) circle (0.2pt);
    \filldraw[black] (0.6,1) circle (0.2pt);
    % \draw (0.4,1) circle (0.2pt);
    \draw (0.6,1) circle (0.2pt);
    \filldraw[black] (0.4,0) circle (0.2pt);
    % \filldraw[black] (0.6,0) circle (0.2pt);
    \filldraw[black] (0,0) circle (0.2pt);
    \filldraw[black] (1,1) circle (0.2pt);
    \end{scope}
    
        \end{tikzpicture}
        \caption{The base $q$ dynamics $\hat{E}_q = \{\hat{f}_0, \hat{f}_1\}$ on $I_q$.}
        \label{figure: Base-q dynamics}
\end{figure}
For any $x \in I_q$, the orbit space of $x$ in $\hat{E}_q$ is given by
$$\orbs{\hat{E}_q}{x} = \{\bqseq{1} \in \{\bq{0}, \bq{1}\}^\mathbb{N} : \mapseqn{\hat{f}}{i}{j}{1}{k}(x) \in I_q \mathrm{\ for \ all \ }k \in \mathbb{N}_{\geq 0}\}.$$
For any $x \in I_q$, it is a simple exercise to show that $\bqseq{1} \in \orbs{\hat{E}_q}{x}$ if and only if $\pi_q(\seqj{i}{1}{\infty}) = x$ (c.f. \cite{baker2012generalised}).
We say that $x \in I_q$ has a unique base $q$ expansion if $\orbs{\hat{E}_q}{x}$ has only one element.

The following table indicates some notation for $\hat{E}_q$ in line with Sidorov \cite{sidorov2009expansions} on the left hand side. On the right hand side is the notation for the corresponding sets in $E_q$, in the context of the self-affine sets $K_q$.
We emphasise here that the notation $\mathcal{V}_q^{(k)}$ used here is \textit{not} in line with \cite{sidorov2009expansions}.

\begin{table}[h]
\centering
\noindent\makebox[\textwidth]{
\begin{tabular}{ | l | l |} 
\hline
$\hat{E}_q$ & 
$E_q$ \\
 \hline
 $\mathcal{B}_k = \{ q \in (1,2) : \exists \ x \in I_q \mathrm{\ such \ that \ } |\orbs{\hat{E}_q}{x}| = k\}$ & 
 $\mathcal{C}_k = \{ q \in (1,2) : \exists \ x \in I_q \mathrm{\ such \ that \ } |\orbs{E_q}{x}| = k\}$ \\ 
 $\mathcal{B}_{\aleph_0} = \{ q \in (1,2) : \exists \ x \in I_q \mathrm{\ such \ that \ } |\orbs{\hat{E}_q}{x}| = \aleph_0 \}$ & 
 $\mathcal{C}_{\aleph_0} = \{ q \in (1,2) : \exists \ x \in I_q \mathrm{\ such \ that \ } |\orbs{E_q}{x}| = \aleph_0 \}$ \\ 
 $\mathcal{U}_q = \{x \in I_q : |\orbs{\hat{E}_q}{x}| = 1\}$ & 
 $\mathcal{V}_q = \{ x  \in I_q : |\orbs{E_q}{x}| = 1\}$ \\ 
 $\mathcal{U}_q^{(k)} = \{x \in I_q : |\orbs{\hat{E}_q}{x}| = k\}$ &
 $\mathcal{V}_q^{(k)} = \{ x  \in I_q : |\orbs{E_q}{x}| = k\}$ \\
 $\mathcal{U}_q^{(\aleph_0)} = \{x \in I_q : |\orbs{\hat{E}_q}{x}| = \aleph_0 \}$ &
$\mathcal{V}_q^{(\aleph_0)} = \{ x  \in I_q : |\orbs{E_q}{x}| = \aleph_0 \}$ \\
$\mathcal{U} = \{ q \in (1,2) : |\orbs{\hat{E}_q}{1}| = 1 \}$ &
$\mathcal{V} = \{ q \in (1,2) : |\orbs{E_q}{1}| = 1\}$ \\
\hline
\end{tabular}
}
% \end{center}
\end{table}

The following result which combines
\cite[Theorem 3]{ErdosJooKomornikCharacter} and
\cite[Theorem 1]{SidorovAlmostEvery2003} 
alongside Lemma \ref{lemma: slice orbit bijection} has immediate implications for the cardinality of $\slice{q}{y}$ as shown in Subsection \ref{subsection: implications of the bijection lemma}.

\begin{theorem}
\label{theorem: basic results on base-q expansions}
    \begin{enumerate}
        \item If $q \in (1,G)$ then $\orbs{\hat{E}_q}{x}$ is uncountably infinite for all $x \in \interior{I_q}$.
        \item If $q = G$ then there is a countably infinite set of points $x \in \interior{I_q}$ such that $\orbs{\hat{E}_q}{x}$ is countably infinite and for all other $x \in \interior{I_q}$, $\orbs{\hat{E}_q}{x}$ is uncountably infinite.
        \item If $q \in (G,2)$ then the set of points $x \in \interior{I_q}$ such that $\orbs{\hat{E}_q}{x}$ has a unique element is at least countably infinite. The set of points $x \in \interior{I_q}$ such that $\orbs{\hat{E}_q}{x}$ is uncountably infinite has full Lebesgue measure.
    \end{enumerate}
\end{theorem}

\subsection{Implications of the bijection lemma}
\label{subsection: implications of the bijection lemma}

In this subsection we use existing results from the theory of base $q$ expansions to prove some elementary results on $\slice{q}{y}$.

If $\bqseq{1} \in \orbs{\hat{E}_q}{x}$ for some $x \in I_q$ then define $D : \{0,1\}^\mathbb{N} \rightarrow \{0,1,2\}^\mathbb{N}$ by $D(\seqj{i}{1}{\infty}) = \seqj{2i}{1}{\infty}$ unless $\seqj{i}{1}{\infty} = (01^\infty)$ in which case $D(01^\infty) = (10^\infty)$.
Then it can be checked that $\bqseq{1} \in \orbs{\hat{E}_q}{x}$ implies that $(f_{d(i_j)})_{j=1}^\infty \in \orbs{E_q}{x}$ where $D(\seqj{i}{1}{\infty}) = (d(i_j))_{j=1}^\infty$.

Lemma \ref{lemma: unique orbit if avoids switch} and the existence of the map $D$ with the aforementioned property prove the following lemma.

\begin{lemma}
\label{lemma: orbit space inequality}
Let $x \in I_q$, then $|\orbs{E_q}{x}| \geq |\orbs{\hat{E}_q}{x}|$ and $|\orbs{\hat{E}_q}{x}| = 1$ if and only if $|\orbs{E_q}{x}| = 1$.
\end{lemma}

An immediate consequence of Lemma \ref{lemma: orbit space inequality} is that $\mathcal{U} = \mathcal{V}$.
Although the set $\mathcal{U}$ has generated significant interest within base $q$ expansion theory (see e.g. \cite{ErdosJooKomornikCharacter, KomornikLoretiUniqueDevelopments}), we do not consider it here until the final result of the paper (Theorem \ref{theorem: U = C_2}).

In conjunction with Lemmas \ref{lemma: slice orbit bijection} and \ref{lemma: orbit space inequality}, Theorem \ref{theorem: basic results on base-q expansions} has the following corollary on slices.

\begin{corollary}
\label{corollary: implications of the bijection lemma}
\begin{enumerate}
    \item If $q \in (1,G)$ then for any $y \in (0,1)$, $\slice{q}{y}$ is uncountably infinite.
    In fact, a consequence of \cite[Theorem 1.4]{SimonBaker2012} is that $\slice{q}{y}$ has positive Hausdorff dimension under these hypotheses.
    \item If $q=G$ then there is a countably infinite set of points $y \in (0,1)$ such that $\slice{q}{y}$ is countably infinite and for all other $y \in (0,1)$ $\slice{q}{y}$ is uncountably infinite.
    \item Let $q \in (G,2)$. 
    The set of points $y \in (0,1)$ such that $\slice{q}{y}$ has a unique element is at least countably infinite. 
    The set of points $y \in (0,1)$ such that $\slice{q}{y}$ is uncountably infinite has full Lebesgue measure. 
\end{enumerate}
\end{corollary}

\section{Proof of Theorem \ref{theorem: interval with slice three}}
\label{section: interval with slice three}

Using Lemma \ref{lemma: slice orbit bijection}, Theorem \ref{theorem: interval with slice three} is equivalent to the claim that if $q \in (q_9,2)$ then there is some $x \in I_q$ such that $|\orbs{E_q}{x}| = 3$. That is, with the above notation, Theorem \ref{theorem: interval with slice three} is equivalent to Theorem \ref{theorem: interval in C_3} below.

\begin{theorem}
\label{theorem: interval in C_3}
    $(q_9,2) \subset \mathcal{C}_3$.
\end{theorem}

We say that a sequence (finite or infinite) $\seqj{i}{1}{k} \in \{0,1\}^*$ \textit{avoids} $\seqj{a}{1}{l} \in \{0,1\}^*$ if there does not exist  $m \in \mathbb{N}$ such that $\seqj{i}{m}{m +l -1} = \seqj{a}{1}{l}$.

Define
$$\mathcal{S}^k = \{\seqj{i}{1}{\infty} \in \{0,1\}^\mathbb{N} : \seqj{i}{1}{\infty} \mathrm{\ avoids \ } (01^k) \mathrm{\ and \ }(10^k)\},$$
and notice that if a finite sequence is a prefix of an element of $\mathcal{S}^k$, then it avoids $(01^k)$ and $(10^k)$.
The following lemma is a consequence of 
\cite[Lemma 4]{glendinning2001unique}.

\begin{lemma}
\label{lemma: projection of S^k is contained in U_q}
    If $q \in (q_k,2)$ then $\pi_q(\mathcal{S}^k) \subset \mathcal{U}_q$.
\end{lemma}

In particular, we will use the fact that if $q \in (q_9,2)$ then $\pi_q(\seqj{i}{1}{\infty}) \in \mathcal{U}_q$ for any sequence $\seqj{i}{1}{\infty} \in \mathcal{S}^9$.
We sketch the proof of Theorem \ref{theorem: interval in C_3}.

\begin{enumerate}
    \item 
    \label{item: images in projection of S^9} Suppose that for each $q \in (q_9,2)$ there is some set $A_q \subset \mathcal{S}^9$ such that $\pi_q(A_q), \pi_q(A_q)-1 \subset \Uqk{9}$ and $A_q$ also satisfies 
    \begin{equation}
    \label{equation: Uq8 intersect pi(A)}
    (\qUk{9}) \cap \pi_q(A_q) \neq \emptyset.
    \end{equation}
    Taking $qy \in (\qUk{9}) \cap \pi_q(A_q)$, we have 
    \begin{align*}
    f_0(y) &= qy \in \pi_q(A_q) \subset \Uqk{9}, \\
    f_1(y) &= \frac{1}{q-1} - \frac{qy-1}{2-q} \in \Uqk{9},\\
    f_2(y) &= qy-1 \in \pi_q(A_q)-1 \subset \Uqk{9}.
    \end{align*}
    In the second line we used
    $$qy \in \qUk{9} \iff \frac{qy-1}{2-q} \in \Uqk{9},$$
    and the fact that $x \in \Uqk{k}$ if and only if $\frac{1}{q-1} - x \in \Uqk{k}$ for all $k \in \mathbb{N}$.
    So under the assumptions on $\pi_q(A_q)$ we have $f_0(y), f_1(y), f_2(y) \in \Uqk{9}$.
    \item By Lemma \ref{lemma: projection of S^k is contained in U_q}, $q \in (q_9,2)$ implies that $\Uqk{9} \subset \Uq$. 
    Let $q \in (q_9,2)$, Item \ref{item: images in projection of S^9} implies that if $A_q$ exists then there exists $y \in \mathcal{V}_q^{(3)}$ and hence $(q_9,2) \subset \mathcal{C}_3$.
    Therefore to prove the theorem, for each $q \in (q_9,2)$ we need to construct $A_q$ such that $\pi_q(A_q), \pi_q(A_q) - 1 \subset \Uqk{9}$
    and \eqref{equation: Uq8 intersect pi(A)} holds.
    \item The claim that each $A_q$ satisfies \eqref{equation: Uq8 intersect pi(A)} is proved using Newhouse's theorem \cite{NewhouseNondensity}.
    \item The property that $\pi_q(A_q), \pi_q(A_q) - 1 \subset \Uqk{9}$ will follow from our construction of $A_q$.
\end{enumerate}
The majority of the work is in proving that the structure of $\pi_q(A_q)$ is amenable to an application of Newhouse's theorem. 
This involves calculating the thickness of $\pi_q(A_q)$ and showing it is interleaved with $(2-q)\pi_q(\mathcal{S}^9)+1$.

\subsection{Results on projections of sequences}

This subsection introduces some useful results on the projections of cylinders and inequalities involving $k$-Bonacci numbers.

\begin{lemma}
\label{lemma: 2-q inequality}
    Let $k \in \mathbb{N}_{\geq 2}$ then 
    \begin{enumerate}
        \item \begin{enumerate}
            \item $0 < (2-q) < q^{-k}$ if $q \in (q_k,2)$,
            \item $q^{-k-1} \leq (2-q) < q^{-k}$ if $q \in (q_k,q_{k+1}]$.
        \end{enumerate}
        \item \begin{enumerate}
            \item $0 < q^k - q^{k-1} - \cdots - q - 1 < 1$ if $q \in (q_k,2)$,
            \item $0 < q^k - q^{k-1} - \cdots - q - 1 \leq \frac{1}{q}$ if $q \in (q_k,q_{k+1}]$.
        \end{enumerate}
    \end{enumerate}
\end{lemma}

\begin{proof}
    It can be checked that $q^{k+1} > k$ whenever $k \in \mathbb{N}_{\geq 2}$ and $q \in (G,2)$, and it follows that $1 > kq^{-k-1}$ under the same conditions.
    Observe that
    $$\left|\frac{\der }{\der q}(2-q)\right| = |-1| > |-k q^{-k-1}| = \left|\frac{\der}{\der q}(q^{-k})\right|,$$
    so $(2-q)$ is monotone decreasing faster than $q^{-k}$.
    Therefore, if $q = q_k$ solves $2-q = q^{-k}$ then both 1(a) and 1(b) would follow.
    Let $q = q_k$, then $q$ satisfies $q^k - q^{k-1} - \cdots - q - 1 = 0$, so 
    \begin{align*}
    0 &=q(q^k - q^{k-1} - \cdots - q - 1) \\
    &=q^{k+1} - 2q^k + (q^k - q^{k-1} - \cdots - q^2 - q) \\
    &=q^{k+1} - 2q^k + 1 \\
    &=q^k(q-2) + 1,
    \end{align*}
    so $(2-q_k) = q_k^{-k}$.

    We prove 2(a) by induction.
    Since $G^2 - G - 1 = 0$, $2^2 - 2 - 1 = 1$ and $q^2 - q -1$ is increasing for $q \in (G,2)$, we know that the claim holds for $q \in (G,2)$.
    For the induction, assume that $q \in (q_{k-1},2)$ implies that $0 < q^{k-1} - q^{k-2} - \cdots -q -1 < 1$ for some $k \in \mathbb{N}_{\geq 3}$.
    Let $q \in (q_k,2)$, and set 
    $$f(q) = q^k - q^{k-1} - \cdots - q - 1,$$
    then
    $$\frac{\der f}{\der q} = kq^{k-1} - (k-1)q^{k-2} - \cdots - 2q - 1 > k(q^{k-1} - q^{k-2} - \cdots - q -1).$$
    Since $q_{k-1} < q_k < 2$, we know by assumption that if $q \in (q_k,2)$ then $0 < q^{k-1} - q^{k-2} - \cdots - q -1 < 1$ so $\frac{\der f}{\der q} > 0$.
    Hence, $f(q)$ is increasing with $q$ on the interval $(q_k,2)$.
    By observing that if $q = q_k$ then $f(q) = 0$ and if $q = 2$ then $f(q) = 1$, we have shown that $q \in (q_k,2)$ implies that $0 < f(q) < 1$ and by induction we are done.

    For 2(b), the above proof of 2(a) shows that $f(q)$ is increasing for all $q \in (q_k,2)$, and $\frac{1}{q}$ is obviously decreasing with $q$.
    It therefore remains to show that $q = q_{k+1}$ solves $q^k - q^{k-1} - \cdots - q - 1 = \frac{1}{q}$.
    Let $q = q_{k+1}$, then
    \begin{align*}
    0 &= q^{k+1} - q^k - \cdots - q - 1\\
    &= q^{k+1} - 2q^k + (q^k - q^{k-1} - \cdots - q - 1),
    \end{align*}
    so
    $$q^k - q^{k-1} - \cdots - q - 1 = q^k(2-q) = \frac{1}{q},$$
    since $(2-q) = q^{-k-1}$ at $q = q_{k+1}$.
\end{proof}

\begin{lemma}
    \label{lemma: sequence inequalities}
    If $q \in (q_k, 2)$ for some $k \in \mathbb{N}_{\geq 2}$ then 
    \begin{enumerate}
        \item $\pi_q(01^k0^\infty) < \pi_q(10^\infty) < \pi_q(01^\infty) < \pi_q(10^k1^\infty)$,
        \item $\pi_q((01^{k-1})^\infty) < \pi_q(10^\infty)$ and symmetrically $\pi_q(01^\infty) < \pi_q((10^{k-1})^\infty)$.
    \end{enumerate} 
\end{lemma}

\begin{proof}
    For the first part, if $q \in (q_k,2)$ Lemma \ref{lemma: 2-q inequality} implies that $q^k - q^{k-1} - \cdots - q - 1 > 0$, so 
    \begin{equation}
    \label{equation: sum of inverses inequality}
        q^{-2} + q^{-3} + \cdots + q^{-k-1} < q^{-1},
    \end{equation}
    and hence $\pi_q(01^k0^\infty) < \pi_q(10^\infty)$. 
    Since $a < b \implies \frac{1}{q-1} - a > \frac{1}{q-1} - b $ we have the symmetric result, $\pi_q(01^\infty) < \pi_q(10^k1^\infty)$. 
    $ q \in (1, 2) $ implies that $\pi_q(10^\infty) < \pi_q(01^\infty)$ which gives the complete inequality.
    
    For the second part, we observe that a consequence of \eqref{equation: sum of inverses inequality} is that for any $m \in \mathbb{N}$, 
    \begin{equation}
    \label{equation: sum of inverses inequality variation}
     q^{-mk-2} + q^{-mk-3} + \cdots + q^{-(m+1)k-1} < q^{-mk-1}.
    \end{equation}
    Using \eqref{equation: sum of inverses inequality} and \eqref{equation: sum of inverses inequality variation},
    \begin{align*}
        \pi_q(10^\infty) = q^{-1} &> (q^{-2} + q^{-3} + \cdots + q^{-k}) + q^{-k-1},\\
        &> (q^{-2} + q^{-3} + \cdots + q^{-k}) + (q^{-k-2} + q^{-k-3} + \cdots + q^{-2k}) + q^{-2k-1},\\
        & \vdots \\
        &> \sum_{m=0}^\infty (q^{-mk-2} + q^{-mk-3} + \cdots + q^{-(m+1)k}),\\
        & = \pi_q((01^{k-1})^\infty),
    \end{align*}
    where at each step we have replaced the final term $q^{-mk-1}$ with the finite sum $(q^{-mk-2} + \cdots + q^{-(m+1)k-1})$ which is smaller by \eqref{equation: sum of inverses inequality variation}.
    By symmetry we also have $\pi_q(01^\infty) < \pi_q((10^{k-1})^\infty)$.
\end{proof}

Let $\seqj{i}{1}{k}, \seqj{i'}{1}{k} \in \{0,1\}^k$ be arbitrary distinct finite sequences. 
We write $\seqj{i}{1}{k} \prec \seqj{i'}{1}{k}$ if $i_p = 0$ and $i'_p = 1$ where $p \in \mathbb{N}$ is the smallest number with $i_p \neq i'_p$. 
This is the natural ordering of the sequence space with the obvious extension to infinite sequences in $\{0,1\}^\mathbb{N}$. 
Let $k \in \mathbb{N}_{\geq 0}$ and let $\seqj{i}{1}{k},\seqj{i'}{1}{k} \in \{0,1\}^k$ be two finite binary strings of length $k$. 
We say that $\seqj{i}{1}{k}$ and $\seqj{i'}{1}{k}$ are \textit{lexicographically consecutive} if they are consecutive as binary numbers. 
That is, if $\sum_{j=1}^k i_j 2^{k-j} + 1 = \sum_{j=1}^k i'_j 2^{k-j}$ or $\sum_{j=1}^k i'_j 2^{k-j} + 1 = \sum_{j=1}^k i_j 2^{k-j}$.

\begin{lemma}
\label{lemma: projections of sequences are well ordered}
    Let $q \in (q_k,2)$ and let $\seqj{i}{1}{l}, \seqj{i'}{1}{l} $ be prefixes of elements of $\mathcal{S}^k$ such that $\seqj{i}{1}{l} \prec \seqj{i'}{1}{l}$.
    Then $\pi_q(\seqj{i}{1}{l}0^\infty) < \pi_q(\seqj{i'}{1}{l}0^\infty)$ and $\pi_q(\seqj{i}{1}{l}1^\infty) < \pi_q(\seqj{i'}{1}{l}1^\infty)$.
\end{lemma}

\begin{proof}
    Let $q \in (q_k, 2)$ and let $\seqj{i}{1}{l} \prec \seqj{i'}{1}{l}$ where both sequences are prefixes of elements of $\mathcal{S}^k$. 
    Since both $\seqj{i}{1}{l}0^\infty$ and $\seqj{i'}{1}{l}0^\infty$ are greedy expansions, as defined in Section 2 of \cite{DEVRIES2009390}, by \cite[Proposition 2.5]{DEVRIES2009390} we know that $\pi_q(\seqj{i}{1}{l}0^\infty) < \pi_q(\seqj{i'}{1}{l}0^\infty)$. 
    This inequality immediately tells us that $\pi_q(\seqj{i}{1}{l}1^\infty) < \pi_q(\seqj{i'}{1}{l}1^\infty)$.
\end{proof}

Recall by Lemma \ref{lemma: projections of cylinders are intervals} that $\pi_q(\seqj{i}{1}{l}0^\infty)$ and $\pi_q(\seqj{i}{1}{l}1^\infty)$ are the left and right endpoints of the interval $\pi_q[\seqj{i}{1}{l}]$.

\begin{lemma}
\label{lemma: cylinders intersect iff lex cons}
    Let $q \in (q_k,2)$, $l \in \mathbb{N}$ and let $\seqj{i}{1}{l}, \seqj{i'}{1}{l}$ be prefixes of elements of $\mathcal{S}^k$. 
    Then $\pi_q[\seqj{i}{1}{l}] \cap \pi_q[\seqj{i'}{1}{l}] \neq \emptyset$ if and only if $\seqj{i}{1}{l}$ and $\seqj{i'}{1}{l}$ are lexicographically consecutive.
\end{lemma}

\begin{figure}[t]
    \centering
    \begin{tikzpicture}
        \begin{scope}[thick]
            \draw (0,0) -- (5,0);
            \draw (4,-0.5) -- (9,-0.5);

            \draw (-4,-0.5) -- (1,-0.5);

            \node[above] at (0,0){$\pi_q(\seqj{i}{1}{l}0^\infty)$};
            \node[above] at (5,0){$\pi_q(\seqj{i}{1}{l}1^\infty)$};

            \node[below] at (4,-0.5){$\pi_q(\seqj{i^+}{1}{l}0^\infty)$};
            \node[below] at (9,-0.5){$\pi_q(\seqj{i^+}{1}{l}1^\infty)$};

            \node[below] at (-4,-0.5){$\pi_q(\seqj{i^-}{1}{l}0^\infty)$};
            \node[below] at (1,-0.5){$\pi_q(\seqj{i^-}{1}{l}1^\infty)$};

            \filldraw[black] (0,0) circle (1pt);
            \filldraw[black] (5,0) circle (1pt);
            \filldraw[black] (4,-0.5) circle (1pt);
            \filldraw[black] (9,-0.5) circle (1pt);

            \filldraw[black] (-4,-0.5) circle (1pt);
            \filldraw[black] (1,-0.5) circle (1pt);
        \end{scope}
    \end{tikzpicture}
    \caption{In the context of Lemma \ref{lemma: cylinders intersect iff lex cons}, $\pi_q[\seqj{i^-}{1}{l}]$ and $\pi_q[\seqj{i^+}{1}{l}]$ intersect if and only if $\seqj{i^-}{1}{l} \prec \seqj{i^+}{1}{l}$ are lexicographically consecutive.}
    \label{figure: intersecting intervals}
\end{figure}
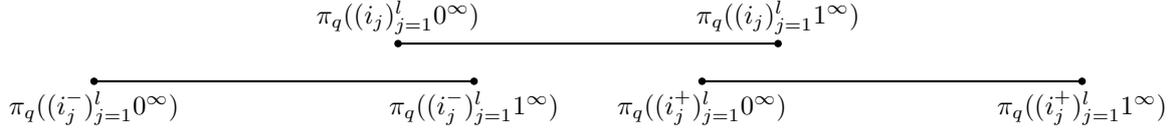

\begin{proof}
    We first show the reverse direction that lexicographically consecutive sequences subject to the hypotheses of the lemma project to intervals which intersect.

    Let $q \in (q_k, 2)$ and let $\seqj{i}{1}{l} , \seqj{i'}{1}{l}$ be prefixes of elements of $\mathcal{S}^k$.
    Suppose $\seqj{i}{1}{l}, \seqj{i'}{1}{l}$ are lexicographically consecutive with $\seqj{i}{1}{l} \prec \seqj{i'}{1}{l}$, then there is some $p \in \mathbb{N}$ such that 
    $\seqj{i}{1}{l-p} = \seqj{i'}{1}{l-p}$,
    $i_{l-p+1} = 0, i_{l-p+2} = \cdots = i_l = 1$, and
    $i'_{l-p+1} = 1, i'_{l-p+2} = \cdots = i'_l = 0$.
    So $\seqj{i}{1}{l}$ and $\seqj{i'}{1}{l}$ are of the form
    $$\seqj{i}{1}{l} =  \seqj{i}{1}{l-p}01^{p-1},$$
    and
    $$\seqj{i'}{1}{l} = \seqj{i}{1}{l-p}10^{p-1}.$$
    Since $\seqj{i}{1}{l}$ and $\seqj{i'}{1}{l}$ avoid $(01^k)$ and $(10^k)$, we know that $1 \leq p \leq k$.
    To prove $\pi_q[\seqj{i}{1}{l}] \cap \pi_q[\seqj{i'}{1}{l}] \neq \emptyset$ it suffices to show that the left endpoint of $\pi_q[\seqj{i}{1}{l}]$ is less than the left endpoint of $\pi_q[\seqj{i'}{1}{l}]$ which in turn is less than the right endpoint of $\pi_q[\seqj{i}{1}{l}]$ (see Figure \ref{figure: intersecting intervals}). 
    This is equivalent to the sequence of inequalities
    $$\pi_q(\seqj{i}{1}{l}0^\infty) < \pi_q(\seqj{i'}{1}{l}0^\infty) < \pi_q(\seqj{i}{1}{l}1^\infty).$$
    The first inequality is immediate from Lemma \ref{lemma: projections of sequences are well ordered} and the second is equivalent to $\pi_q(10^\infty) < \pi_q(01^\infty)$ which is true for all $q \in (1,2)$.
    This proves the reverse direction.

    For the forwards direction of the proof, we show that sequences $\seqj{i^-}{1}{l} \prec \seqj{i^+}{1}{l}$ prefixing elements of $\mathcal{S}^k$ which are not lexicographically consecutive project to disjoint cylinders.
    In this case, by Lemma \ref{lemma: projections of sequences are well ordered}, the left and right endpoints are well ordered as follows:
    $$\pi_q(\seqj{i^-}{1}{l}0^\infty) < \pi_q(\seqj{i^+}{1}{l}0^\infty),$$
    and
    $$\pi_q(\seqj{i^-}{1}{l}1^\infty) < \pi_q(\seqj{i^+}{1}{l}1^\infty).$$
    
    It therefore suffices to show that if
    $\seqj{i^-}{1}{l}, \seqj{i^+}{1}{l}$ prefix elements of $\mathcal{S}^k$ and are such that 
    $\seqj{i^-}{1}{l} \prec \seqj{i}{1}{l}$ are lexicographically consecutive and $\seqj{i}{1}{l} \prec \seqj{i^+}{1}{l}$ are lexicographically consecutive for some $\seqj{i}{1}{l} \in \{0,1\}^l$,
    then the right endpoint of $\seqj{i^-}{1}{l}$ is less than the left endpoint of $\seqj{i^+}{1}{l}$ (see Figure \ref{figure: intersecting intervals}). 
    That is, we aim to prove the inequality
    \begin{equation}
    \label{equation: right endpoint less than left endpoint}
    \pi_q(\seqj{i^-}{1}{l}1^\infty) < \pi_q(\seqj{i^+}{1}{l}0^\infty).
    \end{equation}
    It can be checked that $i_l^- = i_l^+$ and that $\seqj{i^-}{1}{l-1}$ and $ \seqj{i^+}{1}{l-1}$ are lexicographically consecutive with $\seqj{i^-}{1}{l-1} \prec \seqj{i^+}{1}{l-1}$.
    Hence $\seqj{i^-}{1}{l}$ and $\seqj{i^+}{1}{l}$ are of the form
    $$\seqj{i^-}{1}{l} = \seqj{i^-}{1}{l-p-1}01^{p-1}i_l^-,$$
    and
    $$\seqj{i^+}{1}{l} = \seqj{i^-}{1}{l-p-1}10^{p-1}i_l^-.$$
    As in the proof of the reverse direction, we know that $p \leq k-1$ since by assumption $\seqj{i^-}{1}{l}$ and $\seqj{i^+}{1}{l}$ avoid $(01^k)$ and $(10^k)$.
    Substituting the above expressions for $\seqj{i^-}{1}{l}$ and $\seqj{i^+}{1}{l}$ back into \eqref{equation: right endpoint less than left endpoint}, the inequality we now aim to prove is
    $$\pi_q(\seqj{i^-}{1}{l-p-1}01^{p-1}i_l^- 1^\infty) < \pi_q(\seqj{i^-}{1}{l-p-1}10^{p-1}i_l^- 0^\infty),$$
    which we simplify by restricting to the tails to give
    \begin{equation}
    \label{equation: right endpoint less than left endpoint variation}
        \pi_q(01^{p-1}i_l^- 1^\infty) < \pi_q(10^{p-1}i_l^- 0^\infty).
    \end{equation}
    By inspection, the left hand side of \eqref{equation: right endpoint less than left endpoint variation} is maximal at $p = k-1$ and the right hand side is minimal at $p=k-1$.
    It therefore suffices to consider only $p = k-1$.
    Setting $i_l^- = 0$, \eqref{equation: right endpoint less than left endpoint variation} becomes 
    $$\pi_q(01^{k-2}01^\infty) < \pi_q(10^\infty).$$
    We have that 
    $$\pi_q(01^{k-2}01^\infty) < \pi_q(01^{k-2}(10^{k-1})^\infty) < \pi_q((01^{k-1})^\infty) < \pi_q(10^\infty),$$
    where in the first and third inequalities we have used Lemma \ref{lemma: sequence inequalities} and in the second inequality we have used $\pi_q((0^{k-1}1)^\infty) < \pi_q((01^{k-1})^\infty)$.
    Therefore \eqref{equation: right endpoint less than left endpoint} holds when $i_l^- = 0$ and the case for $i_l^- = 1$ is similar.
    This proves the forwards direction and the lemma holds.
    \end{proof}

\subsection{Extension of the base $q$ expansion}
\label{subsection: extension of the base q expansion}

In this subsection we define a set of maps $E_q^* = \{\bqe{-1}, \bqe{0}, \bqe{1}\}$ on $I_q^* = [-\frac{1}{q-1}, \frac{1}{q-1}]$ which extends the standard base $q$ expansion on $I_q$ from the alphabet $\{0,1\}$ to the alphabet $\{-1,0,1\}$.

Define the set of maps $E_q^* = \{\bqe{-1}, \bqe{0}, \bqe{1}\}$ on the interval $I_q^*$ (see Figure \ref{figure: Extended base q expansion}) by
\begin{align*}
    \bqe{-1} &: \left[-\frac{1}{q-1} , \frac{2-q}{q(q-1)}\right] \rightarrow I_q^* ;
&\bqe{-1}(x) &= qx + 1, \\
    \bqe{0} &: \left[ -\frac{1}{q(q-1)}, \frac{1}{q(q-1)}\right] \rightarrow I_q^*;
&\bqe{0}(x) &= qx, \\
    \bqe{1} &: \left[ \frac{-(2-q)}{q(q-1)}, \frac{1}{q-1}\right] \rightarrow I_q^*;
&\bqe{1}(x) &= qx - 1.     
\end{align*}

\begin{figure}[t]
    \centering
        \begin{subfigure}[t]{0.45\textwidth}
        \vskip 0pt
    \begin{tikzpicture}[scale = 3.5]
        \begin{scope}[thick]
        \draw (-1,0) -- (1,0);
        \draw (0,-1) -- (0,1);

        \draw (-1,-1) -- (0.2,1);
        \draw (-0.6, -1) -- (0.6,1);
        \draw (-0.2,-1) -- (1,1);

        \draw[dotted] (0.4,-1) -- (0.4,1);
        \draw[dotted] (0.6,-1) -- (0.6,1);

        \draw [dotted] (-0.4,-1) -- (-0.4,1);
        \draw [dotted] (-0.6,-1) -- (-0.6,1);

        % \draw[dotted] (-0.2,-1) -- (-0.2,1);
        % \draw[dotted] (0.2,-1) -- (0.2,1);

        \draw [dotted] (-1,-1) -- (1,1);

        \draw (-1,-1) -- (-1,1) -- (1,1) -- (1,-1) -- (-1,-1);

        \node[below] at (1,-1) {$\frac{1}{q-1}$};
        \node[below] at (0.6,-1) {$\frac{1}{q(q-1)}$};
        \node[below] at (0.4,-1) {$\frac{1}{q}$};
        \node[below] at (0,-1) {$0$};
        \node[below] at (-0.4,-1) {$-\frac{1}{q}$};
        \node[below] at (-0.64,-1) {$-\frac{1}{q(q-1)}$};
        \node[below] at (-1,-1) {$-\frac{1}{q-1}$};

        \node[left] at (-1,0) {$0$};
        \node[left] at (-1,1) {$\frac{1}{q-1}$};

        \node[above] at (-0.5,1) {$f_{-1}^*$};
        \node[above] at (0,1) {$f_0^*$};
        \node[above] at (0.5,1) {$f_1^*$};
        
        \end{scope}
    \end{tikzpicture}
            \caption{The extension of the base $q$ expansion given by the maps $E_q^*$.
    One can see the top right quadrant resembles the standard base $q$ expansion on $I_q$ if $f_{-1}^*$ is ignored.}
    \label{figure: Extended base q expansion}
        \end{subfigure}
        \hfill
        \begin{subfigure}[t]{0.45\textwidth}
        \vskip 0pt
        \begin{tikzpicture}[scale = 3.5]
        \begin{scope}[thick, opacity = 0.2]
        \draw (-1,0) -- (1,0);
        \draw (0,-1) -- (0,1);

        \draw (-1,-1) -- (0.2,1);
        \draw (-0.6, -1) -- (0.6,1);
        \draw (-0.2,-1) -- (1,1);

        \draw[dotted] (0.4,-1) -- (0.4,1);
        \draw[dotted] (0.6,-1) -- (0.6,1);

        \draw [dotted] (-0.4,-1) -- (-0.4,1);
        \draw [dotted] (-0.6,-1) -- (-0.6,1);

        % \draw[dotted] (-0.2,-1) -- (-0.2,1);
        % \draw[dotted] (0.2,-1) -- (0.2,1);

        \draw [dotted] (-1,-1) -- (1,1);

        \draw (-1,-1) -- (-1,1) -- (1,1) -- (1,-1) -- (-1,-1);

        \node[below] at (1,-1) {$\frac{1}{q-1}$};
        \node[below] at (0.6,-1) {$\frac{1}{q(q-1)}$};
        \node[below] at (0.4,-1) {$\frac{1}{q}$};
        \node[below] at (0,-1) {$0$};
        \node[below] at (-0.4,-1) {$-\frac{1}{q}$};
        \node[below] at (-0.6,-1) {$-\frac{1}{q(q-1)}$};
        \node[below] at (-1,-1) {$-\frac{1}{q-1}$};

        \node[left] at (-1,0) {$0$};
        \node[left] at (-1,1) {$\frac{1}{q-1}$};

        \node[above] at (-0.5,1) {$f_{-1}^*$};
        \node[above] at (0,1) {$f_0^*$};
        \node[above] at (0.5,1) {$f_1^*$};
        
        \end{scope}

        \begin{scope}[thick]
        \draw (-0.8,0) -- (0.8,0);

        \draw (-0.8,-0.1) -- (-0.4,-0.1);
        \draw (-0.5,-0.2) -- (-0.1,-0.2);
        \draw (-0.2,-0.3) -- (0.2,-0.3);
        \draw (0.8,-0.1) -- (0.4,-0.1);
        \draw (0.5,-0.2) -- (0.1,-0.2);

        \draw[dotted] (-0.8,-0.1) -- (-0.8,0);
        \draw[dotted] (-0.4,-0.1) -- (-0.4,0);

        \draw[dotted] (-0.5,-0.2) -- (-0.5,0);
        \draw[dotted] (-0.1, -0.2) -- (-0.1,0);

        \draw [dotted] (-0.2, -0.3) -- (-0.2,0);
        \draw[dotted] (0.2, -0.3) -- (0.2,0);

        \draw[dotted] (0.8,-0.1) -- (0.8,0);
        \draw[dotted] (0.4,-0.1) -- (0.4,0);

        \draw[dotted] (0.5,-0.2) -- (0.5,0);
        \draw[dotted] (0.1, -0.2) -- (0.1,0);

        \node [above] at (0,0) {$H_q$};

        \node [below left] at (-0.6,-0.1) {$S_{-1 0}$};
        \node [below left] at (-0.3,-0.2) {$S_{0 -1}$};
        \node [below] at (0, -0.3) {$S_{00}$};
        \node [below right] at (0.3, -0.2) {$S_{01}$};
        \node [below right] at (0.6, -0.1) {$S_{10}$};
        
        \end{scope}
    \end{tikzpicture}
            \caption{Interval $H_q$ with overlapping subintervals $S_w(H_q)$ for $w \in W_2$.
            Note that the interval $H_q$ and the subintervals are not drawn to scale.}
            \label{figure: H_q}
        \end{subfigure}
        \caption{}
        \label{figure: Extended base q expansion and H_q}
    % \caption{(a) shows the first two steps in the construction of $K_{\frac{3}{5}}$. (b) shows the $F_q$ dynamics for $q = 5/3$. The red line in (a) shows the unique slice at height $y=3/8$ and in (b) it shows the corresponding unique orbit.}
\end{figure}

For any $x \in I_q$, the orbit space of $x$ in $E^*_q$ is given by
$$\orbs{E^*_q}{x} = \{\bqeseq{1} \in \{\bqe{-1}, \bqe{0}, \bqe{1}\}^\mathbb{N} : \mapseqgen{f}{*}{i}{j}{1}{k}(x) \in I_q^* \mathrm{\ for \ all \ }k \in \mathbb{N}_{\geq 0}\}.$$
Since ternary expansions do not index the orbit space of this extension of the base $q$ dynamics, we impose no restrictions on the domains of the maps to be half-open intervals.

As for the standard base $q$ expansion, it is straightforward to show that 
$\bqeseq{1} \in \orbs{E^*_q}{x}$ if and only if $\pi_q(\seqj{i}{1}{\infty}) = x$ for any $x \in I_q^*$. We note that in the case of the extended base $q$ dynamics $E_q^*$ the associated sequence space is $\{-1,0,1\}^\mathbb{N}$ rather than
$\{0,1\}^\mathbb{N}$.

Define $H_q = \left[ -\frac{q}{q^2-1} , \frac{q}{q^2-1} \right] 
\subset I_q^* $ and $W_2 = \{(- \! 10),(0 - \!\!1),(00),(01),(10)\}$. For $x \in I_q^*$, define $S_i(x) = \bqeinv{i}(x)$ for $i \in \{-1,0,1\}$, and for $w = i_1 \ldots i_n \in \{-1,0,1\}^n$, define $S_w(x) = \bqeinv{i_1} \circ \cdots \circ \bqeinv{i_n}(x)$.

\begin{lemma}
\label{lemma: H_q is the union of images of H_q}
    For any $q \in (1,2)$,
    $$H_q = \bigcup_{w \in W_2} S_w(H_q).$$
\end{lemma}

\begin{proof}
The proof is a straightforward calculation and comparison of endpoints of intervals (see Figure \ref{figure: H_q}).
Notice that 
\begin{align*}
S_{00}(H_q) = f_0^{*-1}(f_0^{*-1}(H_q)) &= \left[-\frac{1}{q}\left(\frac{1}{q^2-1}\right), \frac{1}{q}\left(\frac{1}{q^2-1}\right)\right], \\
S_{01}(H_q) =f_0^{*-1}(f_1^{*-1}(H_q)) &= \left[\frac{1}{q^2}\left(\frac{q^2-q-1}{q^2-1}\right) , \frac{1}{q^2}\left(\frac{q^2+q-1}{q^2-1}\right) \right], \\
S_{10}(H_q) =f_1^{*-1}(f_0^{*-1}(H_q)) &= \left[\frac{1}{q}\left(\frac{q^2-2}{q^2-1}\right) , \left(\frac{q}{q^2-1}\right) \right] , \\
S_{0-1}(H_q) =f_0^{*-1}(f_{-1}^{*-1}(H_q)) &= \left[\frac{1}{q^2}\left(\frac{-q^2-q+1}{q^2-1}\right) , \frac{1}{q^2}\left(\frac{-q^2+q+1}{q^2-1}\right) \right] , \\
S_{-10}(H_q) =f_{-1}^{*-1}(f_0^{*-1}(H_q)) &= \left[\left(-\frac{q}{q^2-1}\right) , \frac{1}{q}\left(\frac{2-q^2}{q^2-1}\right) \right] .
\end{align*}
$S_{00}(H_q) \cap S_{01}(H_q) \neq \emptyset$ because
$$ \frac{1}{q}\left(\frac{1}{q^2-1}\right) > \frac{1}{q^2}\left(\frac{q^2-q-1}{q^2-1}\right) ,$$
$$ \iff q^2 - 2q - 1 < 0, $$
which is true for all $q \in (1,2)$, and $S_{01}(H_q) \cap S_{10}(H_q) \neq \emptyset$ because
$$ \frac{1}{q^2}\left(\frac{q^2+q-1}{q^2-1}\right) > \frac{1}{q}\left(\frac{q^2-2}{q^2-1}\right) , $$
$$ \iff q^3 -q^2 - 3q +1 < 0, $$
which is also true for all $q \in (1,2)$. By symmetry $S_{00}(H_q) \cap S_{0-1}(H_q) \neq \emptyset $ and $S_{0-1}(H_q) \cap S_{-10}(H_q) \neq \emptyset $.
The left and right endpoints of $H_q$ are preserved by the maps $S_{-10}(H_q)$ and $S_{10}(H_q)$ respectively, hence the lemma.
\end{proof}

The following lemma follows easily from Lemma \ref{lemma: H_q is the union of images of H_q}.

\begin{lemma}
\label{lemma: elements of H_q have orbits in W_2}
    If $x \in H_q$ then there is some $ \bqeseq{1} \in \orbs{E^*_q}{x}$ such that $\seqj{i}{1}{\infty} \in W_2^\mathbb{N}$ and $\mapseqgen{f}{*}{i}{j}{1}{k}(x) \in H_q$ for all $k \in 2\mathbb{N}$.
\end{lemma}

\begin{proof}
    If $x \in H_q$ then $x \in S_w(H_q)$ for some $w \in W_2$ by Lemma \ref{lemma: H_q is the union of images of H_q}. 
    Let $w_1 = i_1 i_2$ then $\bqe{i_2}(\bqe{i_1}(x)) \in H_q$. Repeating the above argument we generate a sequence of words $\seqj{w}{1}{\infty} = \seqj{i}{1}{\infty}$ where $w_j \in W_2$ for all $j \in \mathbb{N}$ and $\mapseqgen{f}{*}{i}{j}{1}{k} \in H_q$ for all $k \in 2\mathbb{N}$.
    By construction it follows that $(f^*_{i_j})_{j=1}^\infty \in \orbs{E^*_q}{x}$.
\end{proof}

The following lemma guarantees the existence of an orbit of $1$ whose associated sequence has a tail in $W_2^\mathbb{N}$.
This is important for the definition of the \textit{fixed expansion of $1$} in the next subsection to be valid.

\begin{lemma}
\label{lemma: orbit of 1 has tail in W_2}
    If $q \in (1,2)$ then there is some $M \in \mathbb{N}_{\geq 0}$ depending only on $q$ and some sequence $\seqj{w}{1}{\infty} \in W_2^\mathbb{N}$ such that $\seqj{i}{1}{\infty} = 1^M\seqj{w}{1}{\infty}$ satisfies $\mapseqgeninf{f}{*}{i}{j}{1}{\infty} \in \orbs{E_q^*}{1}$.
    Moreover, if $q \in (1,G]$ then $M=0$ and if $q \in (q_k, q_{k+1}]$ for $k \in \mathbb{N}_{\geq 2}$ then $M = k$.
\end{lemma}

\begin{proof}
    Since $q_1 = 1$ and $\lim_{n \rightarrow \infty} q_n = 2$, we know that $q \in (1,2)$ implies that $q \in (q_k, q_{k+1}]$ for some $k 
    \in \mathbb{N}$.
    By Lemma \ref{lemma: elements of H_q have orbits in W_2} it suffices to show, for each $q \in (1,2)$, that there is some $M \in \mathbb{N}_{\geq 0}$ such that $(f^*_1)^M(1) \in H_q$.
    Recall that $q_2 = G$.
    If $q \in (1,G]$ then $0 < 1 \leq \frac{q}{q^2-1}$ so $1 \in H_q$ and by setting $M=0$ we are done.
    Let $q \in (q_k, q_{k+1}]$ for some $k \geq 2$, and observe $(f^*_1)^k(1) = q^k - q^{k-1} - \cdots - q - 1$.
    By Lemma \ref{lemma: 2-q inequality}, we know that if $q \in (q_k, q_{k+1}]$ then $0 < (f^*_1)^k(1) \leq \frac{1}{q}$.
    Finally, $\frac{1}{q} < \frac{q}{q^2-1}$ for all $q \in (G,2)$ so setting $M = k$, $(f^*_1)^M(1) \in H_q$ and we are done.
\end{proof}

\subsection{Construction of $A_q$}
\label{subsection: construction of A_q}

In this subsection we construct a family of sequences $A_q$ for each $q \in (q_9,2)$ with the property that whenever $\seqj{a}{1}{\infty} \in A_q$ we can guarantee that $\pi_q(\seqj{a}{1}{\infty}), \pi_q(\seqj{a}{1}{\infty})-1 \in \mathcal{U}_q$. 
This is the content of Proposition \ref{proposition: A_q and A_q - 1 in U_q} below.

For each $q \in (q_9,2)$, fix $ \seqj{c}{1}{\infty} = 1^M\seqj{c}{M+1}{\infty} \in \{-1,0,1\}^\mathbb{N}$ where $1 = \pi_q(\seqj{c}{1}{\infty})$ and $\seqj{c}{M+1}{\infty} \in W_2^\mathbb{N}$ is as described in Lemma \ref{lemma: orbit of 1 has tail in W_2}.
The sequence $\seqj{c}{1}{\infty}$ is referred to as the \textit{fixed expansion of $1$}.
We emphasise here that the sequence $\seqj{c}{1}{\infty}$ with these properties is not necessarily unique, but by fixing one such sequence for each $q \in (q_9,2)$, the definition of the fixed expansion of $1$ is valid.

Let $J$ be the set of zeros of the sequence $\seqj{c}{1}{\infty}$:
$$ J = \{j \in \mathbb{N} : c_j = 0\}. $$
We can enumerate $J$ in the obvious way:
$J = \{j_0, j_1 , \ldots\}$ where $j_0 < j_1 < \cdots$.
Define 
$$\Jfixedone = \{j_m \in J : m =4k + 1 \ \mathrm{for \ some \ } k \in \mathbb{N}\},$$
$$ \Jfixedzero = \{j_m \in J : m = 4k+3 \ \mathrm{for \ some \ } k \in \mathbb{N}\}, $$
and
$$\Jfree = \{j_m \in J : m \ \mathrm{is \ even} \},$$ 
and call an index in $\Jfree$ a \textit{free zero} of $\seqj{c}{1}{\infty}$. 
If $j \in \Jfixedone \cup \Jfixedzero$ or if $j$ is such that $c_j \in \{-1,1\}$ then $j$ is a \textit{fixed index}, otherwise $j$ is a free zero.

Given $q \in (q_9,2)$ and $1^M\seqj{c}{M+1}{\infty}$ the fixed expansion of $1$, the set $A_q$ consists of the sequences $\seqj{a}{1}{\infty} \in \{0,1\}^\mathbb{N}$ which satisfy the following properties:
\begin{enumerate}
    \item $a_j = 1$ if $c_j = 1$,
    \item $a_j = 0$ if $c_j = -1$,
    \item $a_j = 1$ if $j \in \Jfixedone$ and
    \item $a_j = 0$ if $j \in \Jfixedzero$.
\end{enumerate}

Notice that there are no restrictions on the value of $a_j$ if $j \in \Jfree$ and later we will use this fact along with the bounded distance between zeros of $\seqj{c}{1}{\infty}$ to show that the thickness of $\pi_q(A_q)$ can be bounded below.
We can now state the key proposition of this subsection.

\begin{proposition}
\label{proposition: A_q and A_q - 1 in U_q}
    If $q \in (q_9,2)$ then $\pi_q(A_q),\pi_q(A_q)-1 \subset \mathcal{U}_q$.
\end{proposition}

Proposition \ref{proposition: A_q and A_q - 1 in U_q} is a consequence of the following two lemmas.

\begin{lemma}
    \label{lemma: projection of A_q is in U_q}
    If $q \in (q_9,2)$ then $\pi_q(A_q) \subset \mathcal{U}_q$.
\end{lemma}

\begin{proof}
    By Lemma \ref{lemma: projection of S^k is contained in U_q}, it suffices to show that if $q \in (q_9,2)$ then $A_q \subset \mathcal{S}^9$.
    
    Let $q \in (q_9,2)$ and let $1^M\seqj{c}{M+1}{\infty}$ be the fixed expansion of $1$. Since each $w \in W_2$ has length two and at least one $0$, we know that for any $m \in \mathbb{N}$, $(2m+1)$ consecutive indices of $\seqj{c}{M+1}{\infty}$ must contain at least $m$ $0$s. 
    In particular, each set of nine consecutive indices must contain at least four $0$s.

    By inspection of the sets $\Jfixedzero$ and $\Jfixedone$, we see that in every set of four consecutive $0$s of $\seqj{c}{M+1}{\infty}$, $\{j_m, j_{m+1} , j_{m+2}, j_{m+3}\}$, we have at least one element of $\Jfixedzero$ and at least one element of $\Jfixedone$. 
    Hence, for $j>M$, $\{a_j , \ldots , a_{j+8}\}$ contains at least one $0$ and at least one $1$.
    Since $\seqj{a}{1}{M} = 1^M$ and $\seqj{a}{M+1}{\infty}$ avoids the strings $(1^9)$ and $(0^9)$, we know that $\seqj{a}{1}{\infty}$ avoids $(01^9)$ and $(10^9)$, that is, $A_q \subset \mathcal{S}^9$. 
    % Lemma \ref{lemma: projection of S^k is contained in U_q} then provides the conclusion that $\pi_q(A_q) \subset \mathcal{U}_q$.
\end{proof}

\begin{lemma}
\label{lemma: projection of A_q - 1 is in U_q}
    If $q \in (q_9,2)$ then $\pi_q(A_q)-1 \subset \mathcal{U}_q$.
\end{lemma}

\begin{proof}
    Let $q \in (q_9,2)$ and let $1^M \seqj{c}{1}{\infty}$ be the fixed expansion of $1$. 
    For each $\seqj{a}{1}{\infty} \in A_q$ define a sequence $\seqj{b}{1}{\infty}$ by declaring that $b_j = a_j - c_j$ for all $j \in \mathbb{N}$. Therefore $\seqj{b}{1}{\infty}$ satisfies
    \begin{enumerate}
        \item $b_j = 0$ if $c_j = 1$,
        \item $b_j = 1$ if $c_j = -1$,
        \item $b_j = a_j$ if $c_j = 0$.
    \end{enumerate}
    Since $\seqj{a}{1}{\infty} \in \{0,1\}^\mathbb{N}$, we know that $\seqj{b}{1}{\infty} \in \{0,1\}^\mathbb{N}$. 
    Moreover, by inspection of the proof of Lemma \ref{lemma: projection of A_q is in U_q}, since $b_j = a_j$ whenever $c_j = 0$, we know that $\seqj{b}{M+1}{\infty}$ also avoids $0^9$ and $1^9$. 
    Noting that $\seqj{b}{1}{M} = 0^M$, we can say that $\seqj{b}{1}{\infty}$ avoids $(01^9)$ and $(10^9)$, so $\seqj{b}{1}{\infty} \in \mathcal{S}^9$. 
    Since $q \in (q_9,2)$, by Lemma \ref{lemma: projection of S^k is contained in U_q} we know that $\pi_q(\seqj{b}{1}{\infty}) \in \mathcal{U}_q$.

    By definition of $\seqj{b}{1}{\infty}$ and by the additive property of $\pi_q$, we have that $\seqj{a}{1}{\infty}$, $\seqj{b}{1}{\infty}$ and $\seqj{c}{1}{\infty}$ satisfy 
    $$1 = \pi_q(\seqj{c}{1}{\infty}) = \pi_q(\seqj{a}{1}{\infty}) - \pi_q(\seqj{b}{1}{\infty}).$$
    So for each $\seqj{a}{1}{\infty} \in A_q$ there is some uniquely defined $\seqj{b}{1}{\infty}$ such that $\pi_q(\seqj{b}{1}{\infty}) = \pi_q(\seqj{a}{1}{\infty}) - 1$. 
    Since $q \in (q_9,2)$, $\pi_q(\seqj{b}{1}{\infty}) \in \mathcal{U}_q$ and hence $\pi_q(A_q) - 1  \subset \mathcal{U}_q$.
\end{proof}

Lemmas \ref{lemma: projection of A_q is in U_q} and \ref{lemma: projection of A_q - 1 is in U_q} together prove Proposition \ref{proposition: A_q and A_q - 1 in U_q}.

\subsection{The structure of $\pi_q(A_q)$}
\label{subsection: structure of pi_q(A_q)}
In this subsection we prove Proposition \ref{proposition: thickness bound}.
This thickness estimate forms a key part of the application of Newhouse's theorem \cite{NewhouseNondensity} to prove Theorem \ref{theorem: interval in C_3}.
The following definition of \textit{thickness} is due to Newhouse \cite{Newhouse1979}.
We note that an interleaving condition is also required for Newhouse's theorem and this is addressed in Subsection \ref{subsection: interleaving}.

Let $C \subset \mathbb{R}$ be a compact set such that the complement of $C$ in $\mathbb{R}$ is the union of two disjoint unbounded open intervals with a countable collection of open intervals we refer to as \textit{gaps}. 
Denote this countable collection of gaps by $\{G_n : n \in \mathbb{N} \}$ and let $|G_m| \geq |G_n|$ whenever $m < n $. 
To each gap $G_n$ in the complement of $C$ we associate a left and right \textit{bridge}. 
Given a gap $G_n$ the left bridge, $L_n$, is defined to be the interval to the left of $G_n$ whose right endpoint is the left endpoint of $G_n$ and whose left endpoint is the right endpoint of the gap immediately to the left of $G_n$ and at least as big. The right bridge, $R_n$, is defined similarly.
So for a given gap $G_n$, we know the gaps contained in the bridges $L_n$ and $R_n$ are strictly smaller than $G_n$.
The thickness of the compact set $C$ is given by
$$\tau(C) = \inf \left\{ \min \left\{ \frac{|L_n|}{|G_n|} , \frac{|R_n|}{|G_n|} \right\} : n \in \mathbb{N} \right\}.$$
Note that the unbounded components of the complement of $C$ in $\mathbb{R}$ are not considered when dealing with thickness, so we ignore these in what follows.
We are now able to state the key proposition of this subsection.

\begin{proposition}    
\label{proposition: thickness bound}
    If $q \in (q_9,2)$ then $\tau(\pi_q(A_q)) > q^{-5}$.
\end{proposition}

Let $q \in (q_9,2)$ and let $1^M\seqj{c}{M+1}{\infty} \in \{-1,0,1\}^\mathbb{N}$ be the fixed expansion of $1$, defined in Subsection \ref{subsection: construction of A_q}.
Recall that $M$ and the expansion $1^M\seqj{c}{M+1}{\infty}$ are defined such that $\seqj{c}{M+1}{\infty} \in W_2^\mathbb{N}$ and $\pi_q(1^M\seqj{c}{M+1}{\infty}) = 1$.
For every $k \in \mathbb{N}$ we define the set of length $k$ prefixes of sequences in $A_q$ by $A_q^k$, that is
$$A_q^k = \{\seqj{a}{1}{k} \in \{0,1\}^k : \exists \seqj{a'}{1}{\infty} \in A_q \mathrm{\ such \ that \ } a'_j = a_j \ \forall 1 \leq j \leq k \},$$
and write
$$\pi_q[A_q^k] = \{ \pi_q[\seqj{a}{1}{k}] : \seqj{a}{1}{k} \in A_q^k \},$$ 
so elements of $\pi_q[A_q^k]$ are intervals.
Since $q \in (q_9,2)$ implies that $A_q \subset \mathcal{S}^9$, it is immediate that any $\seqj{a}{1}{k} \in A_q^k$ avoids $(10^9)$ and $(01^9)$.

The lemma below will be used frequently throughout the proofs of the remaining lemmas of this section.

\begin{lemma}
\label{lemma: projected intervals in pi_q(A_q)}
    \begin{enumerate}
        \item $\seqj{a}{1}{k} \in A_q^k$ if and only if $\pi_q[\seqj{a}{1}{k}] \cap \pi_q(A_q) \neq \emptyset$.
        \item If $\seqj{a}{1}{k} \notin A_q^k$ is such that there are sequences $\seqj{a^-}{1}{k}, \seqj{a^+}{1}{k} \in A_q^k$ with the property that $\seqj{a^-}{1}{k} \prec \seqj{a}{1}{k}$ and $\seqj{a}{1}{k} \prec \seqj{a^+}{1}{k}$, then $\pi_q[\seqj{a}{1}{k}] \subset \mathrm{conv}(\pi_q(A_q)) \setminus \pi_q(A_q)$.
        \item If $\seqj{a}{1}{k-2} 0 a_k , \seqj{a}{1}{k-2} 1 a_k \in A_q^k$ then $\pi_q[\seqj{a}{1}{k-2} a_k \overline{a_k}] \cap \pi_q[\seqj{a}{1}{k-2} 0 a_k] \neq \emptyset$ and $\pi_q[\seqj{a}{1}{k-2} a_k \overline{a_k}] \cap \pi_q[\seqj{a}{1}{k-2} 1 a_k] \neq \emptyset$.
    \end{enumerate}
\end{lemma}

\begin{proof}
    The first part follows from the fact that $\seqj{a}{1}{k} \in A_q^k$ if and only if $\seqj{a}{1}{k}$ is a prefix of some element $\seqj{a}{1}{\infty} \in A_q$ which satisfies $\pi_q(\seqj{a}{1}{\infty}) \in \pi_q[\seqj{a}{1}{k}]$.

    For the second part, we assume the hypotheses and use the first part to guarantee that $\pi_q[\seqj{a^-}{1}{k}] \cap \pi_q(A_q) \neq \emptyset$ and $\pi_q[\seqj{a^+}{1}{k}] \cap \pi_q(A_q) \neq \emptyset$.
    Since $\seqj{a}{1}{k} \notin A_q^k$, $\seqj{a}{1}{k}$ is not a prefix of any element of $A_q$ we know that $\pi_q[\seqj{a}{1}{k}] \cap \pi_q(A_q) = \emptyset$.
    Therefore $\seqj{a^-}{1}{k}$ and $\seqj{a^+}{1}{k}$ are prefixes of sequences $\seqj{a^-}{1}{\infty} , \seqj{a^+}{1}{\infty} \in A_q$ which satisfy
    $$\pi_q(\seqj{a^-}{1}{\infty}) < \pi_q(\seqj{a}{1}{k}0^\infty) < \pi_q(\seqj{a}{1}{k}1^\infty) < \pi_q(\seqj{a^+}{1}{\infty}).$$
    This inequality proves that the interval $\pi_q[\seqj{a}{1}{k}]$ is bounded on both sides by elements of $\pi_q(A_q)$, and hence $\pi_q[\seqj{a}{1}{k}] \subset \mathrm{conv}(\pi_q(A_q)) \setminus \pi_q(A_q)$.

    For the final part, notice that if $a_k$ is $0$ or $1$, in either case the inequality $\pi_q(\seqj{a}{1}{k-2}0a_k 1^\infty) > \pi_q(\seqj{a}{1}{k-2}a_k \overline{a_k}0^\infty)$ reduces to $\pi_q(01^\infty) > \pi_q(10^\infty)$, which we know is true for all $ q \in (1,2)$.
    Therefore $\pi_q[\seqj{a}{1}{k-2}a_k \overline{a_k}] \cap \pi_q[\seqj{a}{1}{k-2}0a_k] \neq \emptyset$ and a similar argument shows that $\pi_q[\seqj{a}{1}{k-2}a_k \overline{a_k}] \cap \pi_q[\seqj{a}{1}{k-2}1a_k] \neq \emptyset$ which completes the proof.
\end{proof}

A \textit{level $k$ gap} is a gap $G$ which contains the interval (which we show to exist) between $\pi_q[\seqj{a}{1}{k-2}0a_k]$ and $\pi_q[\seqj{a}{1}{k-2}1 a_k]$ for some $\seqj{a}{1}{k-2}0a_k, \seqj{a}{1}{k-2}1 a_k \in A_q^k$.
Given a gap $G$, we denote the left and right bridges of $G$ by $L_G$ and $R_G$ respectively.

We prove the following claims which allow us to bound $\tau(\pi_q(A_q))$ from below.
\begin{enumerate}
    \item All gaps of $\pi_q(A_q)$ are open.
    \item Level $k$ gaps exist if and only if $k-1$ is a free zero (which requires $k \geq M+1$).
    \item A level $k$ gap cannot be a level $l$ gap for any $l \neq k$.
    \item All gaps of $\pi_q(A_q)$ are level $k$ gaps for some $k$.
    \item If $G$ is a level $k$ gap and $H$ is a level $l$ gap for some $l > k$ then $|H| < |G|$.
    Moreover, $q^{-k} < |G| < q^{-k+1}$.
    \item If $G$ is a level $k$ gap then $|L_G|, |R_G| > q^{-k-4}$.
\end{enumerate}

These steps lead us to conclude that $\tau( \pi_q(A_q)) > q^{-5}$ (Proposition \ref{proposition: thickness bound}).
This bound will then be used in an application of Newhouse's theorem to prove \eqref{equation: Uq8 intersect pi(A)} from which Theorem \ref{theorem: interval in C_3} follows.
The items in the list are proved sequentially in the next six lemmas.

\begin{lemma}
    All gaps of $\pi_q(A_q)$ are open.
\end{lemma}

\begin{proof}
Let $\seqj{i}{1}{\infty}, \seqj{i'}{1}{\infty} \in \{0,1\}^\mathbb{N}$ and consider the metric on $\{0,1\}^\mathbb{N}$ given by
\begin{equation*}
 d(\seqj{i}{1}{\infty}, \seqj{i'}{1}{\infty}) = 
\begin{cases} 
    2^{-n} & \mathrm{where \ } n=\min\{j \in \mathbb{N} : i_j \neq i'_j\}, \\
    0 & \mathrm{if \ } \seqj{i}{1}{\infty} = \seqj{i'}{1}{\infty}. 
\end{cases}
\end{equation*}
$A_q$ is compact with respect to the topology generated by this metric, and $\pi_q$ is continuous.
Therefore $\pi_q(A_q)$ is compact with respect to the normal topology on $\mathbb{R}$, generated by the open intervals, and hence its complement is open.
The gaps of $\pi_q(A_q)$ form a subset of the complement, so all the gaps are necessarily open.
\end{proof}

Let $\overline{x}$ be the reflection of $x$, i.e. if $x = \seqj{i}{1}{k} \in \{0,1\}^*$ then $ \overline{x} = \seqj{\overline{i}}{1}{k} = \seqj{1-i}{1}{k}$.
In particular if $x \in \{0,1\}$ then $\overline{x} = 1-x$ and the reflection of the empty word is just itself.

\begin{lemma}
    Let $q \in (q_9,2)$, then level $k$ gaps exist if and only if $k-1$ is a free zero. 
\end{lemma}

\begin{proof}
    Observe that the existence of sequences $\seqj{a}{1}{k-2}0a_k$ and $\seqj{a}{1}{k-2}1 a_k$ as elements of $A_q^k$ is equivalent to $k-1$ being a free zero.
    Hence it suffices to show that if $\seqj{a}{1}{k-2}0a_k$ and $\seqj{a}{1}{k-2}1 a_k$ are elements of $A_q^k$ then $\pi_q[\seqj{a}{1}{k-2}0a_k]$ and $\pi_q[\seqj{a}{1}{k-2}1 a_k]$ do not intersect and the interval between them is contained in a gap.

    Suppose $\seqj{a}{1}{k-2}0a_k, \seqj{a}{1}{k-2}1 a_k \in A_q^k$ then since $\seqj{a}{1}{k-2}0a_k \prec \seqj{a}{1}{k-2}1 a_k$ are not lexicographically consecutive, Lemma \ref{lemma: cylinders intersect iff lex cons} implies that $\pi_q[\seqj{a}{1}{k-2}0a_k]$ and $\pi_q[\seqj{a}{1}{k-2}1a_k]$ do not intersect.
    Recall from Subsection \ref{subsection: construction of A_q} that if $k-1$ is a free zero, then $k$ must be a fixed index, so $\seqj{a}{1}{k-2}a_k\overline{a_k} \notin A_q^k$ because the $k$th entry is not $a_k$.
    It can be easily checked that $\seqj{a}{1}{k-2}0a_k \prec \seqj{a}{1}{k-2}a_k\overline{a_k}$ are lexicographically consecutive and that $\seqj{a}{1}{k-2}a_k\overline{a_k} \prec \seqj{a}{1}{k-2}1a_k$ are lexicographically consecutive.
    So, by Lemma \ref{lemma: projected intervals in pi_q(A_q)}, $\pi_q[\seqj{a}{1}{k-2}a_k\overline{a_k}]$ is a gap contained in $ \mathrm{conv}(\pi_q(A_q)) \setminus \pi_q(A_q)$.    

    Again by Lemma \ref{lemma: projected intervals in pi_q(A_q)}, $\pi_q[\seqj{a}{1}{k-2}a_k \overline{a_k}]$ intersects both $\pi_q[\seqj{a}{1}{k-2}0a_k]$ and $\pi_q[\seqj{a}{1}{k-2}1 a_k]$ and must therefore contain the interval between them, which completes the proof.
\end{proof}

Note that by the definition of the fixed expansion of $1$, the first free zero cannot occur before index $M+1$, hence we guarantee there are no level $k$ gaps for $k \leq M+1$.
To prove a level $k$ gap cannot be a level $l$ gap for any $l \neq k$, we prove that the gap 
which contains the interval between $\pi_q[\seqj{a}{1}{k-2}0a_k]$ and $\pi_q[\seqj{a}{1}{k-2}1 a_k]$ is determined uniquely by the sequence $\seqj{a}{1}{k-2} \in A_q^{k-2}$. This is the content of the following lemma.

\begin{lemma}
\label{lemma: level j gaps are well defined}
    Let $q \in (q_9,2)$.
    Let $G$ be the gap which contains the interval between $\pi_q[\seqj{a}{1}{k-2}0a_k]$ and $\pi_q[\seqj{a}{1}{k-2}1a_k]$ for some $k \in \mathbb{N}$ and some $\seqj{a}{1}{k-2}0 a_k, \seqj{a}{1}{k-2}1a_k \in A_q^k$.
    Let $H$ be the gap which contains the interval between $\pi_q[\seqj{b}{1}{l-2}0b_l]$ and $\pi_q[\seqj{b}{1}{l-2}1b_l]$ for some $l \in \mathbb{N}$ and $\seqj{b}{1}{l-2}0 b_l, \seqj{b}{1}{l-2}1b_l \in A_q^l$.
    Then if $\seqj{a}{1}{k-2} \neq \seqj{b}{1}{l-2}$ then $G \cap H = \emptyset$.
\end{lemma}

\begin{proof}
    Assume the hypotheses of the lemma and that $k,l \in \mathbb{N}$ with $k \leq l$.
    Suppose that $\seqj{a}{1}{k-2} \prec \seqj{b}{1}{k-2}$. It is clear that $G \subset \pi_q[\seqj{a}{1}{k-2}]$ and $H \subset \pi_q[\seqj{b}{1}{l-2}] \subset \pi_q[\seqj{b}{1}{k-2}]$. Since $k-1$ is a free zero, $k-2$ is a fixed index so $a_{k-2} = b_{k-2}$ and therefore $\seqj{a}{1}{k-2}$ and $\seqj{b}{1}{k-2}$ are not lexicographically consecutive.
    By Lemma \ref{lemma: cylinders intersect iff lex cons} we know that $\pi_q[\seqj{a}{1}{k-2}] \cap \pi_q[\seqj{b}{1}{k-2}] = \emptyset$ so $G \cap H = \emptyset$.

    Suppose instead that $\seqj{a}{1}{k-2} = \seqj{b}{1}{k-2}$. In order for the premise $\seqj{a}{1}{k-2} \neq \seqj{b}{1}{l-2}$ to hold, we require that $k<l$. Notice that since $l-1$ is a free zero and $k$ is a fixed index, we require $l \geq k+2$.
    Without loss of generality\footnote{If instead we assume that $\seqj{b}{1}{k} = \seqj{a}{1}{k-2}0a_k$ then \eqref{equation: cylinder containment} becomes $\pi_q[\seqj{b}{1}{l-2}1b_l] \subset \pi_q[\seqj{b}{1}{l-2}] \subset \pi_q[\seqj{a}{1}{k-2}0a_k]$ and the rest of the argument is similar with $H$ disjoint and to the left of $G$ instead of disjoint and to the right of $G$.}, let $\seqj{b}{1}{k} = \seqj{a}{1}{k-2}1a_k$. Therefore 
    \begin{equation}
        \label{equation: cylinder containment}
        \pi_q[\seqj{b}{1}{l-2}0b_l] \subset \pi_q[\seqj{b}{1}{l-2}] \subset \pi_q[\seqj{a}{1}{k-2}1a_k].
    \end{equation}
    With Lemma \ref{lemma: projected intervals in pi_q(A_q)}, $\seqj{b}{1}{l-2}0b_l \in A_q^l$ so there exists some $x \in \pi_q[\seqj{b}{1}{l-2}0b_l] \cap \pi_q(A_q)$.
    Using \eqref{equation: cylinder containment}, we know that $x \in \pi_q[\seqj{a}{1}{k-2}1 a_k] \cap \pi_q(A_q)$. 
    Since $G$ contains the interval between $\pi_q[\seqj{a}{1}{k-2}0 a_k]$ and $\pi_q[\seqj{a}{1}{k-2}1 a_k]$ and $G$ is a connected component of $\mathrm{conv}(\pi_q(A_q)) \setminus \pi_q(A_q)$, we know that $y < x$ for all $y \in G$.
    
    Similarly, since $H$ contains the interval between $\pi_q[\seqj{b}{1}{l-2}0b_l]$ and $\pi_q[\seqj{b}{1}{l-2}1b_l]$ we know that $x < z$ for all $z \in H$ (see Figure \ref{figure: level k gaps are well defined}).
    Since $x \in \pi_q(A_q)$ we conclude that $G \cap H = \emptyset$.
\end{proof}

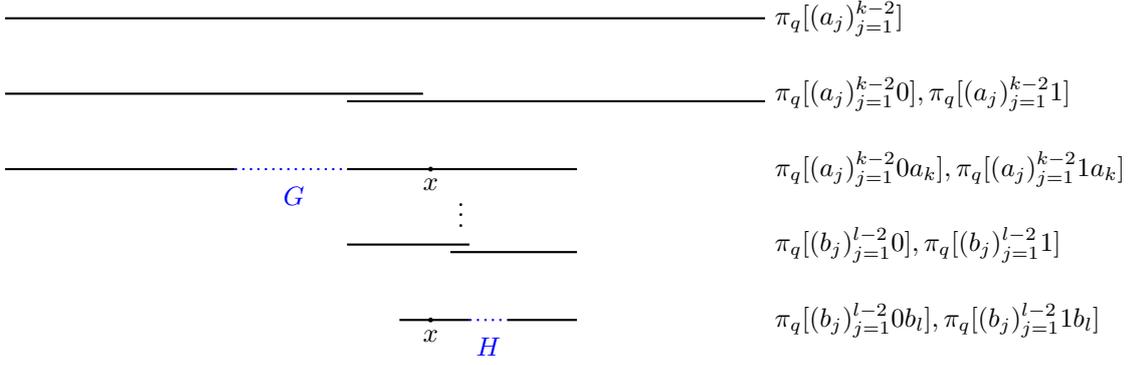
\begin{figure}[t]
    \centering
    \begin{tikzpicture}
        \begin{scope}[thick]
            \draw (0,0) -- (10,0);

            \draw (0,-1) -- (5.5,-1);
            \draw (4.5,-1.1) -- (10,-1.1);

            \draw (0,-2) -- (3.025,-2);
            \draw (4.5,-2) -- (7.525,-2);

            \draw (4.5,-3) -- (6.11,-3);
            \draw (5.86,-3.1) -- (7.525,-3.1);

            \draw (5.19, -4) -- (6.11,-4);
            \draw (6.61,-4) -- (7.525,-4);

            \draw[blue,dotted] (3.025,-2) -- (4.5,-2);
            \draw[blue,dotted] (6.11,-4) -- (6.61,-4);

            \filldraw[black] (5.6,-4) circle (0.5pt);
            \node[below] at (5.6,-4){$x$};

            \filldraw[black] (5.6,-2) circle (0.5pt);
            \node[below] at (5.6,-2){$x$};

            \node[blue,below] at (3.8,-2.1){$G$};
            \node[blue,below] at (6.36,-4.1){$H$};

            \node at (6,-2.5){$\vdots$};

            \node[right] at (10,0){$\pi_q[\seqj{a}{1}{k-2}]$};
            \node[right] at (10,-1){$\pi_q[\seqj{a}{1}{k-2}0], \pi_q[\seqj{a}{1}{k-2}1]$};
            \node[right] at (10,-2){$\pi_q[\seqj{a}{1}{k-2}0a_k], \pi_q[\seqj{a}{1}{k-2}1a_k]$};
            \node[right] at (10,-3){$\pi_q[\seqj{b}{1}{l-2}0], \pi_q[\seqj{b}{1}{l-2}1]$};
            \node[right] at (10,-4){$\pi_q[\seqj{b}{1}{l-2}0b_l], \pi_q[\seqj{b}{1}{l-2}1b_l]$};
        \end{scope}
    \end{tikzpicture}
    \caption{The projections of cylinders used to show $G \cap H = \emptyset$ in the proof of Lemma \ref{lemma: level j gaps are well defined}. The sets $G$ and $H$ contain the intervals represented by the dotted lines above them. The point $x \in \pi_q(A_q)$ is an element of neither $G$ nor $H$ but necessarily lies between them.}
    \label{figure: level k gaps are well defined}
\end{figure}

Next we show that every gap is a level $k$ gap for some $k \in \mathbb{N}$.

\begin{lemma}
\label{lemma: all gaps are level k gaps}
    Let $q \in (q_9,2)$. If $G$ is a gap then $G$ is a level $k$ gap for some $k \in \mathbb{N}$.
\end{lemma}

\begin{proof}
    Let $q \in (q_9,2)$ and let $1^M \seqj{c}{M+1}{\infty}$ be the fixed expansion of $1$. 
    Let $G$ be a gap and let ${k_G} \in \mathbb{N}$ be the largest number such that there exists an element $\seqj{a}{1}{{k_G}-2} \in A_q^{{k_G}-2}$ with the property that $G \subset \interior{\pi_q[\seqj{a}{1}{{k_G}-2}]}$. 
    We show that ${k_G}$ exists and claim that $G$ is a level ${k_G}$ gap.
    
    Recall that $j_0$ is the smallest free zero of $1^M \seqj{c}{M+1}{\infty}$, so let $\seqj{a}{1}{j_0-1}$ be the unique element of $A_q^{j_0-1}$, then $\seqj{a}{1}{j_0-1}$ is a prefix of every element of $A_q$.
    So $\pi_q(A_q) \subset \pi_q[\seqj{a}{1}{j_0-1}]$.
    The endpoints of $\pi_q[\seqj{a}{1}{j_0-1}]$ are not in $\pi_q(A_q)$ so $\mathrm{conv}(\pi_q(A_q)) \subset \interior{\pi_q[\seqj{a}{1}{j_0-1}]}$.
    Hence all gaps are contained in $\interior{\pi_q[\seqj{a}{1}{j_0-1}]}$ and so for any gap $G$, ${k_G}$ is at least $j_0+1$.
    A calculation shows that $|\pi_q[\seqj{a}{1}{{k_G}-2}]| = q^{-{k_G}+2}(\frac{1}{q-1})$ which approaches $0$ as ${k_G} \rightarrow \infty$.
    Since every gap $G$ has positive length, ${k_G}$ is bounded above for every gap $G$. 
    This shows that for any gap $G$, ${k_G}$ exists, is finite and is unique since it is maximal within a finite set.
    
    Let $G$ be a gap and ${k_G} \in \mathbb{N}$ as above. We show that ${k_G}-1$ is a free zero by contradiction.
    Suppose ${k_G}-1$ is a fixed index and let $a_{{k_G}-1} \in \{0,1\}$ be such that $\seqj{a}{1}{{k_G}-2}a_{{k_G}-1} \in A_q^{{k_G}-1}$ so $\seqj{a}{1}{{k_G}-2}\overline{a_{{k_G}-1}} \notin A_q^{{k_G}-1}$.
    We know that $G \subset \interior{\pi_q[\seqj{a}{1}{{k_G}-2}]}$ and $\interior{\pi_q[\seqj{a}{1}{k_G-2}]} = \interior{\pi_q[\seqj{a}{1}{{k_G}-2}a_{k_G -1}]} \cup \interior{\pi_q[\seqj{a}{1}{{k_G}-2}\overline{a_{k_G -1}}]}$.
    Therefore, since $G \not\subset \interior{\pi_q[\seqj{a}{1}{{k_G}-2}a_{k_G -1}]}$ we know that $G \cap \pi_q[\seqj{a}{1}{{k_G}-2}\overline{a_{{k_G}-1}}] \neq \emptyset$.
    Because $\seqj{a}{1}{{k_G}-2}\overline{a_{{k_G}-1}} \notin A_q^{{k_G}-1}$, Lemma \ref{lemma: projected intervals in pi_q(A_q)} tells us that $\pi_q[\seqj{a}{1}{{k_G}-2}\overline{a_{{k_G}-1}}] \cap \pi_q(A_q) = \emptyset$ so we know that $\pi_q[\seqj{a}{1}{{k_G}-2}\overline{a_{{k_G}-1}}]\subset G$. 
    Observe that $ \pi_q(\seqj{a}{1}{{k_G}-2}(\overline{a_{{k_G}-1}})^\infty) \in \pi_q[\seqj{a}{1}{{k_G}-2}\overline{a_{{k_G}-1}}]$ is an extremal point of $\pi_q[\seqj{a}{1}{{k_G}-2}]$ so $\pi_q(\seqj{a}{1}{{k_G}-2}(\overline{a_{{k_G}-1}})^\infty) \in \pi_q[\seqj{a}{1}{{k_G}-2}]\setminus \interior{\pi_q[\seqj{a}{1}{{k_G}-2}]}$ which contradicts $G \subset \interior{\pi_q[\seqj{a}{1}{{k_G}-2}]}$. 
    Hence ${k_G}-1$ is a free zero.

    We next show that $\pi_q[\seqj{a}{1}{{k_G}-2}\overline{a_{k_G}}\overline{a_{k_G}}] \cap G = \emptyset$.
    We show that this implies that $G$ contains the interval between $\pi_q[\seqj{a}{1}{{k_G}-2}0a_{k_G}]$ and $\pi_q[\seqj{a}{1}{{k_G}-2}1a_{k_G}]$, proving $G$ is a level $k_G$ gap.
    Since ${k_G}-1$ is a free zero ${k_G}-2$ and ${k_G}$ are fixed indices so let $a_{k_G} \in \{0,1\}$ be such that $\seqj{a}{1}{{k_G}-2}0a_{k_G}, \seqj{a}{1}{{k_G}-2}1a_{k_G} \in A_q^{k_G}$.
    Notice that by Lemma \ref{lemma: projected intervals in pi_q(A_q)}, $\pi_q[\seqj{a}{1}{{k_G}-2}\overline{a_{k_G}}\overline{a_{k_G}}] \cap \pi_q(A_q) = \emptyset$ because $\seqj{a}{1}{{k_G}-2}\overline{a_{k_G}}\overline{a_{k_G}} \notin A_q^{k_G}$.
    Suppose $\pi_q[\seqj{a}{1}{{k_G}-2}\overline{a_{k_G}}\overline{a_{k_G}}] \cap G \neq \emptyset$, then since $G$ is a gap and $\pi_q[\seqj{a}{1}{{k_G}-2}\overline{a_{k_G}}\overline{a_{k_G}}] \cap \pi_q(A_q) = \emptyset$, this is equivalent to $\pi_q[\seqj{a}{1}{{k_G}-2}\overline{a_{k_G}}\overline{a_{k_G}}] \subset G$.
    In this case, $\pi_q(\seqj{a}{1}{{k_G}-2}(\overline{a_{k_G}})^\infty) \in G$, but $\pi_q(\seqj{a}{1}{{k_G}-2}(\overline{a_{k_G}})^\infty) \in \pi_q[\seqj{a}{1}{{k_G}-2}]\setminus \interior{\pi_q[\seqj{a}{1}{{k_G}-2}]}$ which contradicts $G \subset \interior{\pi_q[\seqj{a}{1}{{k_G}-2}]}$. 
    Hence, $\pi_q[\seqj{a}{1}{{k_G}-2}\overline{a_{k_G}}\overline{a_{k_G}}] \cap G = \emptyset$.

    Consider the union
    $$\pi_q[\seqj{a}{1}{{k_G}-2}] = \bigcup_{\delta \in \{a_{k_G} a_{k_G}, a_{k_G} \overline{a_{k_G}}, \overline{a_{k_G}}a_{k_G}, \overline{a_{k_G}}\overline{a_{k_G}}\}} \pi_q[\seqj{a}{1}{{k_G}-2}\delta].$$
    Since $\pi_q[\seqj{a}{1}{{k_G}-2}\overline{a_{k_G}}\overline{a_{k_G}}] \cap G = \emptyset$, we know that 
    $$G \subset \bigcup_{\delta \in \{a_{k_G} a_{k_G}, a_{k_G} \overline{a_{k_G}}, \overline{a_{k_G}}a_{k_G}\}} \pi_q[\seqj{a}{1}{{k_G}-2}\delta].$$
    We know that $\{a_{k_G} a_{k_G} , \overline{a_{k_G}} a_{k_G}\} = \{0 a_{k_G} , 1 a_{k_G}\}$ and $\seqj{a}{1}{{k_G}-2}a_{k_G}\overline{a_{k_G}} \notin A_q^{k_G}$.
    Since $G$ is an interval which is not contained in either $\interior{\pi_q[\seqj{a}{1}{{k_G}-2}0a_{k_G}]}$ or $\interior{\pi_q[\seqj{a}{1}{{k_G}-2}1a_{k_G}]}$, we know $ \pi_q[\seqj{a}{1}{{k_G}-2}a_{k_G}\overline{a_{k_G}}] \cap G \neq \emptyset$.
    Again, by Lemma \ref{lemma: projected intervals in pi_q(A_q)} we know that $\pi_q[\seqj{a}{1}{{k_G}-2}a_{k_G}\overline{a_{k_G}}] \cap \pi_q(A_q) = \emptyset$.
    Since $G$ is a gap, the fact that $\pi_q[\seqj{a}{1}{{k_G}-2}a_{k_G}\overline{a_{k_G}}] \cap G \neq \emptyset$ tells us that $\pi_q[\seqj{a}{1}{{k_G}-2}a_{k_G}\overline{a_{k_G}}] \subset G$.

    By inspection, either $a_{k_G} a_{k_G} \prec a_{k_G} \overline{a_{k_G}} \prec \overline{a_{k_G}}a_{k_G}$ or $\overline{a_{k_G}} a_{k_G} \prec a_{k_G} \overline{a_{k_G}} \prec a_{k_G} a_{k_G}$ and in 
    each case both pairs $\{a_{k_G} a_{k_G} , a_{k_G} \overline{a_{k_G}}\}$ and $\{a_{k_G}\overline{a_{k_G}}, \overline{a_{k_G}}a_{k_G}\}$ are lexicographically consecutive.
    Therefore by Lemma \ref{lemma: projected intervals in pi_q(A_q)},
    $$\pi_q[\seqj{a}{1}{k_G-2} a_{k_G} \overline{a_{k_G}}] \cap \pi_q[\seqj{a}{1}{k_G-2} a_{k_G} a_{k_G}] \neq \emptyset,$$
    and
    $$\pi_q[\seqj{a}{1}{k_G-2} a_{k_G} \overline{a_{k_G}}] \cap \pi_q[\seqj{a}{1}{k_G-2} \overline{a_{k_G}} a_{k_G}] \neq \emptyset.$$
    Since $\pi_q[\seqj{a}{1}{{k_G}-2}a_{k_G}\overline{a_{k_G}}] \subset G$ and $G$ is an interval, we know $G$ contains the interval between $\pi_q[\seqj{a}{1}{{k_G}-2}0a_{k_G}]$ and $\pi_q[\seqj{a}{1}{{k_G}-2}1a_{k_G}]$. Hence $G$ is a level $k_G$ gap.
\end{proof}

\begin{lemma}
\label{lemma: level k gaps decrease}
   Let $q  \in (q_9,2)$. If $G$ is a level $k$ gap then $q^{-k} < |G| < q^{-k+1}$.
   Moreover, if $H$ is a level $l$ gap then $|G|> |H|$ whenever $k<l$.
\end{lemma}

\begin{proof}
    Let $q \in (q_9,2)$ and let $G$ be a level $k$ gap. 
    Suppose $G$ is the level $k$ gap which contains the interval between $\pi_q[\seqj{a}{1}{k-2}0a_k]$ and $\pi_q[\seqj{a}{1}{k-2}1 a_k]$ for some $\seqj{a}{1}{k-2}0 a_k, \seqj{a}{1}{k-2}1a_k \in A_q^k$.
    We find lower and upper bounds on $|G|$ by inspecting further restrictions on the endpoints of $G$.

    Let $J_1, J_2, J_3$ be successive free zeros such that $J_1 = k-1$.
    Let $\alpha_1$ and $\alpha_2$ be the binary strings between the pairs of free zeros $J_1, J_2$ and $J_2, J_3$ respectively. 
    That is $\alpha_1 = \seqj{a}{J_1+1}{J_2-1}$ and $\alpha_2 = \seqj{a}{J_2+1}{J_3-1}$.
    Since $J_1,J_2,J_3$ are successive free zeros, there are no free zeros in the ranges $J_1+1, \ldots , J_2-1$ and $J_2+1 , \ldots , J_3-1$, so $\alpha_1$ and $\alpha_2$ are fixed since they are indexed by fixed indices.
    There is at least one and at most four fixed indices between free zeros of any sequence in $A_q$, so $1 \leq |\alpha_1|,|\alpha_2| \leq 4$. 
    Therefore, for any $a_{J_1}, a_{J_2} \in \{0,1\}$ we know that
    $$\seqj{a}{1}{k-2} a_{J_1} \alpha_1 a_{J_2} \alpha_2 \in A_q^{k + |\alpha_1| + |\alpha_2|}.$$ 
    By the above construction, we have the equality:
    \begin{multline}
    \label{equation: bounding gaps setwise equivalence}
    \{\seqj{a}{1}{k-2} a_{J_1} \alpha_1 a_{J_2} \alpha_2 \in A_q^{k + |\alpha_1| + |\alpha_2|} : a_{J_1}, a_{J_2} \in \{0,1\} \} \\
    = \{ \seqj{a'}{1}{k + |\alpha_1| + |\alpha_2|} \in A_q^{k + |\alpha_1| + |\alpha_2|} : \seqj{a'}{1}{k-2} = \seqj{a}{1}{k-2}\}.
    \end{multline}

    In words, \eqref{equation: bounding gaps setwise equivalence} is saying that the expression $\seqj{a}{1}{k-2} a_{J_1} \alpha_1 a_{J_2} \alpha_2$ describes all sequences in $A_q^{k + |\alpha_1| + |\alpha_2|}$ prefixed by $\seqj{a}{1}{k-2}$.

    We know $G$ contains the interval between $\pi_q[\seqj{a}{1}{k-2}0 a_k]$ and $\pi_q[\seqj{a}{1}{k-2}1 a_k]$, therefore, $G$ contains the interval between the rightmost element of $\pi_q[A_q^{k+|\alpha_1| + |\alpha_2|}]$ contained in $\pi_q[\seqj{a}{1}{k-2}0 a_k]$ and the leftmost element of $\pi_q[A_q^{k + |\alpha_1| + |\alpha_2|}]$ contained in $\pi_q[\seqj{a}{1}{k-2}1 a_k]$.
    By \eqref{equation: bounding gaps setwise equivalence}, the rightmost element of $\pi_q[A_q^{k+|\alpha_1| + |\alpha_2|}]$ contained in $\pi_q[\seqj{a}{1}{k-2}0 a_k]$ is given by $\pi_q[\seqj{a}{1}{k-2}0 \alpha_1 1 \alpha_2]$ and
    the leftmost element of $\pi_q[A_q^{k+|\alpha_1| + |\alpha_2|}]$ contained in $\pi_q[\seqj{a}{1}{k-2}1 a_k]$ is given by $\pi_q[\seqj{a}{1}{k-2}1 \alpha_1 0 \alpha_2]$.
    Moreover, we know $G$ is contained in the convex hull of the union of these intervals, that is,
    $$G \subset \mathrm{conv}( \pi_q[\seqj{a}{1}{k-2}0 \alpha_1 1 \alpha_2] \cup \pi_q[\seqj{a}{1}{k-2}1 \alpha_1 0 \alpha_2] ),$$
    because both $\pi_q[\seqj{a}{1}{k-2}0 \alpha_1 1 \alpha_2]$ and $\pi_q[\seqj{a}{1}{k-2}1 \alpha_1 0 \alpha_2 ]$ are elements of $\pi_q[A_q^{k+|\alpha_1|+|\alpha_2|}]$.
    
    The above argument finds an interval which is contained in $G$ and an interval which contains $G$.
    This provides both upper and lower bounds on $|G|$ (see Figure \ref{figure: upper and lower bounds for level k gap}).
    The containments described above give the following inequality,
    $$
    \pi_q(\seqj{a}{1}{k-2}1 \alpha_1 0 \alpha_2 0^\infty) - \pi_q(\seqj{a}{1}{k-2} 0 \alpha_1 1 \alpha_2 1^\infty) \leq |G| \leq \pi_q(\seqj{a}{1}{k-2}1 \alpha_1 0 \alpha_2 1^\infty) - \pi_q(\seqj{a}{1}{k-2} 0 \alpha_1 1 \alpha_2 0^\infty),
    $$
    which algebraically gives,
    \begin{multline}
    \label{equation: upper and lower bounds on |G|}
    q^{-(k-2)}\left[ q^{-1} - q^{-2-|\alpha_1|} - q^{-2-|\alpha_1| - |\alpha_2|}\left(\frac{1}{q-1}\right)\right] \leq |G| \\
    \leq q^{-(k-2)}\left[q^{-1} - q^{-2-|\alpha_1|} + q^{-2 - |\alpha_1| - |\alpha_2|}\left(\frac{1}{q-1}\right)\right].
    \end{multline}

    For the lower bound, we observe that the left hand side of \eqref{equation: upper and lower bounds on |G|} is smallest when the negative terms are largest, i.e. when $|\alpha_1| = |\alpha_2| = 1$. 
    Then
    \begin{align*}
    |G| &\geq q^{-k+2}\left[q^{-1} - q^{-3} - q^{-4}\left(\frac{1}{q-1}\right)\right] \\
     &> q^{-k+1}\left[1 - \frac{25}{8}q^{-3}\right]
    \end{align*}
    where the last inequality comes from $q<2 \implies q^{-3} < 2q^{-4}$ and $q > q_9 \implies \frac{1}{q-1} < \frac{9}{8}$. Now using $q > q_9 \implies q^{-3} < \frac{1}{7}$ and $q^{-1} < \frac{31}{56}$ we have
    $$|G| > q^{-k}.$$

    For the upper bound, the right hand inequality of \eqref{equation: upper and lower bounds on |G|} is equivalent to
    $$|G| \leq  q^{-k+2}\left[q^{-1} + q^{-2-|\alpha_1|}\left(\frac{q^{-|\alpha_2|}}{q-1} - 1 \right)\right].$$
    The product $q^{-2-|\alpha_1|}\left(\frac{q^{-|\alpha_2|}}{q-1} - 1\right)$ is negative so $|G| < q^{-k+1}$.
    Therefore, $q^{-k} < |G| <  q^{-k+1}$.
    If $H$ is a level $l$ gap for some $l > k$ then since $l-1$ is a free zero, $l \geq k+2$ so $|H| < q^{-l+1} \leq q^{-k-1} < q^{-k} < |G|$ and the lemma is proved.
    \end{proof}
    
\begin{figure}
    \centering
    \begin{tikzpicture}
        \begin{scope}[thick]
            \draw (0,0) -- (10,0);
            \node[right] at (10,0){$\pi_q[\seqj{a}{1}{k-2}]$};

            \draw (0,-1) -- (6,-1);
            \draw (4,-1.1) -- (10,-1.1);
            \node[right] at (10,-1){$\pi_q[\seqj{a}{1}{k-2}0], \pi_q[\seqj{a}{1}{k-2}1]$};

            \node at (2.5,-1.5){$\vdots$};
            \node at (7.5,-1.5){$\vdots$};

            \draw (1,-2) -- (4,-2);
            \draw (6,-2) -- (9,-2);
            \node[right] at (10,-2){$\pi_q[\seqj{a}{1}{k-2}0\alpha_1], \pi_q[\seqj{a}{1}{k-2}1\alpha_1]$};

            \draw[dashed] (1,-3) -- (3,-3);
            \draw (2,-3.1) -- (4,-3.1);
            \draw (6,-3.1) -- (8,-3.1);
            \draw[dashed] (7,-3) -- (9,-3);
            \node[right] at (10,-3){$\pi_q[\seqj{a}{1}{k-2}0\alpha_1 1], \pi_q[\seqj{a}{1}{k-2}1\alpha_1 0 ]$};

            \node at (3,-3.5){$\vdots$};
            \node at (7,-3.5){$\vdots$};

            \draw (2.5,-4) -- (3.5,-4);
            \draw (6.5,-4) -- (7.5,-4);
            \draw[blue] (3.2,-4.1) -- (6.8,-4.1);
            % \draw[blue,dotted] (3.5,-3.9) -- (6.5,-3.9);
            \node[blue,below] at (5,-4.1){$G$};
            \node[right] at (10,-4){$\pi_q[\seqj{a}{1}{k-2}0\alpha_1 1 \alpha_2], \pi_q[\seqj{a}{1}{k-2}1\alpha_1 0 \alpha_2]$};

        \end{scope}
    \end{tikzpicture}
    \caption{Upper and lower bounds for $|G|$ in the context of Lemma \ref{lemma: level k gaps decrease}.}
    \label{figure: upper and lower bounds for level k gap}
\end{figure}
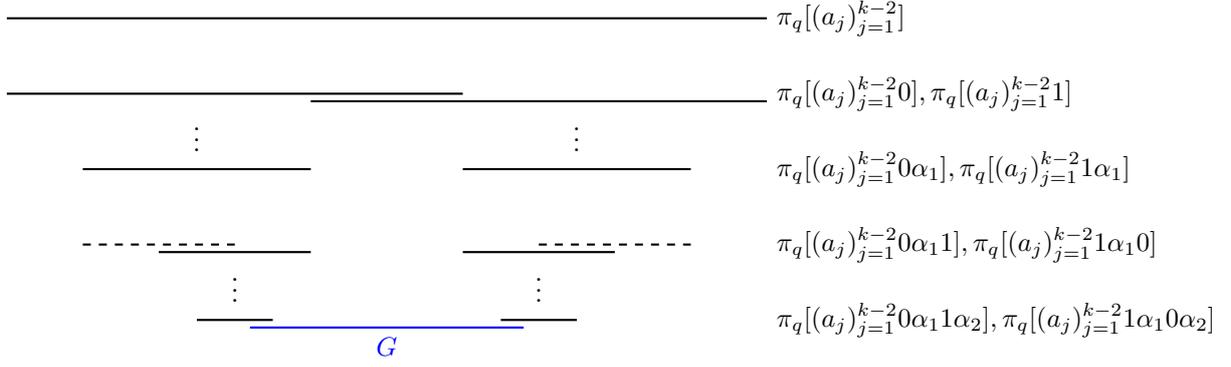

Lemma \ref{lemma: level k gaps decrease} allows us to estimate the length of the bridge either side of some gap $G$.
If $G$ is a level $k$ gap of $\pi_q(A_q)$, we know by Lemma \ref{lemma: level k gaps decrease} that if $l > k$ then all level $l$ gaps are smaller than $G$. 
Let $H$ be the gap at least as large and immediately to the left of $G$.
Then $H$ is a level $m$ gap for some $m \leq k$.
Therefore to bound $L_G$ below it suffices to find a lower bound on the distance between $G$ and $H$.

\begin{lemma}
\label{lemma: lower bound on bridge}
    Let $q \in (q_9,2)$.
    If $G$ is a level $k$ gap of $\pi_q(A_q)$ then the bridges associated with $G$ on the left and the right, $L_G$ and $R_G$, have diameter $|L_G|,|R_G| > q^{-k-4}$.
\end{lemma}

\begin{proof}
    Let $k-1$ be a free zero and let $\seqj{a}{1}{k-2}0 a_k, \seqj{a}{1}{k-2} 1 a_k \in A_q^k$.
    By Lemma \ref{lemma: level j gaps are well defined}, we know there is a unique level $k$ gap $G$ which contains the interval between $\pi_q[\seqj{a}{1}{k-2}0 a_k]$ and $\pi_q[\seqj{a}{1}{k-2} 1 a_k]$.
    Given $G$, it is sufficient to bound the diameter of the left bridge $|L_G|$ from below because the same bound on the diameter of the right bridge $|R_G|$ follows by symmetry.
    To bound $|L_G|$ from below we seek an interval immediately to the left of $G$ which does not intersect any gap larger than $G$, whose length we bound below.

    We start by proving that $\pi_q[\seqj{a}{1}{k-2}0 a_k]$ does not contain a level $l$ gap for any $l \leq k$. 
    Let $H$ be a level $l$ gap for some $l \leq k$, then there are elements $\seqj{b}{1}{l-2} 0 b_l , \seqj{b}{1}{l-2} 1 b_l \in A_q^l$ such that $H$ contains the interval between $\pi_q[\seqj{b}{1}{l-2} 0 b_l]$ and $\pi_q[\seqj{b}{1}{l-2} 1 b_l]$, and this uniquely determines $H$.
    By the proof of Lemma \ref{lemma: all gaps are level k gaps}, $H$ contains the interval $\pi_q[\seqj{b}{1}{l-2} b_l \overline{b_l}]$ and since $l \leq k$, we have the inequality
    $$|H| \geq |\pi_q[\seqj{b}{1}{l-2} b_l \overline{b_l}]| \geq |\pi_q[\seqj{a}{1}{k-2} 0 a_k]|.$$
    Hence $H$ cannot be contained in $\pi_q[\seqj{a}{1}{k-2} 0 a_k]$.
    
    By Lemmas \ref{lemma: all gaps are level k gaps} and \ref{lemma: level k gaps decrease}, this implies that any gap contained in $\pi_q[\seqj{a}{1}{k-2}0 a_k]$ is smaller than $G$.
    We emphasise that it is possible for a gap with diameter greater than or equal to $|G|$ to \textit{intersect} $\pi_q[\seqj{a}{1}{k-2}0 a_k]$ but no such gap is \textit{contained} in $\pi_q[\seqj{a}{1}{k-2}0 a_k]$.
    Removing the elements of $\pi_q[\seqj{a}{1}{k-2}0 a_k]$ which are not elements of $\pi_q(A_q)$ and taking the convex hull removes the possibility of intersecting with a gap of diameter greater than or equal to $|G|$ (see Figure \ref{figure: subset which avoids larger gaps}).
    That is, the set $\mathrm{conv}(\pi_q[\seqj{a}{1}{k-2}0 a_k] \cap \pi_q(A_q))$ does not intersect any gap with diameter greater than or equal to $|G|$.
    Therefore a lower bound on $|\pi_q[\seqj{a}{1}{k-2}0 a_k] \cap \pi_q(A_q)|$ is a lower bound on $|L_G|$.
    It now suffices to show that $|\pi_q[\seqj{a}{1}{k-2}0 a_k] \cap \pi_q(A_q)| > q^{-k-4}$.
    
    \begin{figure}
        \centering
        \begin{tikzpicture}
        \begin{scope}[thick]
            \draw (1,0) -- (4,0);
            \draw (6,0) -- (9,0);

            \node[above] at (2.5,0){$\pi_q[\seqj{a}{1}{k-2}0 a_k]$};
            \node[above] at (7.5,0){$ \pi_q[\seqj{a}{1}{k-2}1 a_k ]$};

            \draw [dotted](3.5,-0.5) -- (6.5,-0.5);

            \node[below] at (5,-0.5){$G$};

            \draw [dotted] (-1.5,-0.5) -- (1.5,-0.5);

            \node[below] at (0,-0.5){$J$};

            \draw (1.5,-1) -- (3.5,-1);

            \node[below] at (2.5,-1){$\mathrm{conv}(\pi_q[\seqj{a}{1}{k-2}0 a_k] \cap \pi_q(A_q))$};
        \end{scope}
        \end{tikzpicture}
        \caption{Constructing the subset $ \mathrm{conv}(\pi_q[\seqj{a}{1}{k-2} 0 a_k]\cap \pi_q(A_q)) \subset \pi_q[\seqj{a}{1}{k-2} 0 a_k]$ which does not intersect any gap with diameter greater than or equal to $|G|$. 
        $J$ is defined to be the gap which contains the point $\pi_q(\seqj{a}{1}{k-2} 0 a_k 0^\infty)$ which may satisfy $|J| \geq |G|$.}            \label{figure: subset which avoids larger gaps}
    \end{figure}
    
    As in the proof of Lemma \ref{lemma: level k gaps decrease}, let $J_1, J_2, J_3 , J_4$ be successive free zeros such that $J_1 = k-1$ and let $\alpha_1, \alpha_2, \alpha_3$ be the fixed binary strings between the pairs of free zeros $J_1, J_2$ and $J_2, J_3$ and $J_3, J_4$ respectively.
    That is $\alpha_1 = \seqj{a}{J_1+1}{J_2-1}$, $\alpha_2 = \seqj{a}{J_2+1}{J_3-1}$ and $\alpha_3 = \seqj{a}{J_3+1}{J_4-1}$ are fixed because there are no free zeros in the ranges $J_1+1, \ldots , J_2-1$ and $J_2+1, \ldots , J_3-1$ and $J_3+1, \ldots , J_4-1$.
    There is at least one and at most four fixed indices between free zeros of any sequence in $A_q$, so $1 \leq |\alpha_1|,|\alpha_2|, |\alpha_3| \leq 4$.
    As before, the implication of this is that $\alpha_1, \alpha_2,\alpha_3$ are such that for any $a_{J_1}, a_{J_2}, a_{J_3} \in \{0,1\}$, 
    $\seqj{a}{1}{k-2} a_{J_1} \alpha_1 a_{J_2} \alpha_2 a_{J_3} \alpha_3 \in A_q^{k+1 + |\alpha_1| + |\alpha_2| + |\alpha_3|}$.
    Using this, we know that $\pi_q[\seqj{a}{1}{k-2}0 \alpha_1 0 \alpha_2 0 \alpha_3]$ and $\pi_q[\seqj{a}{1}{k-2} 0 \alpha_1 1 \alpha_2 1 \alpha_3]$ are subintervals of $\pi_q[\seqj{a}{1}{k-2} 0 a_k]$ and intersect $\pi_q(A_q)$, so the distance between them is a lower bound for $|L_G|$. 
    This distance is given by
    \begin{align*}
    |L_G| &\geq \pi_q(\seqj{a}{1}{k-2}0 \alpha_1 1 \alpha_2 1 \alpha_3 0^\infty) - \pi_q(\seqj{a}{1}{k-2} 0 \alpha_1 0 \alpha_2 0 \alpha_3 1^\infty) \\
    &=q^{-(k-1+|\alpha_1|)}\left[ q^{-1} + q^{-2-|\alpha_2|} - q^{-2 - |\alpha_2| - |\alpha_3|}\left(\frac{1}{q-1}\right)\right] \\
    &= q^{-k-|\alpha_1|}\left[ 1 + q^{-1-|\alpha_2|}\left(1 - \frac{q^{-|\alpha_3|}}{q-1}\right)\right].
    \end{align*}
    We aim to minimize this expression under the constraints $1 \leq |\alpha_1|, |\alpha_2|, |\alpha_3| \leq 4$. Since $|\alpha_3| \geq 1$, the term in round brackets is positive which implies the term in square brackets is positive. This gives $|\alpha_1| = 4$ and $|\alpha_2| = 4$. $|\alpha_3| = 1$ minimizes the term in round brackets and hence the whole expression. This gives,
    $$|L_G| \geq q^{-k-4}\left[1 + q^{-5}\left(1-\frac{q^{-1}}{q-1}\right)\right] > q^{-k-4},$$
    which completes the proof.
\end{proof}

Recall from Lemma \ref{lemma: level k gaps decrease} that if $G$ is a level $k$ gap then $|G| < q^{-k+1}$.
We combine this with Lemma \ref{lemma: lower bound on bridge} to find a lower bound on the thickness of $\pi_q(A_q)$.

% \begin{proposition}    
% \label{proposition: thickness bound}
%     If $q \in (q_9,2)$ then $\tau(\pi_q(A_q)) > q^{-5}$.
% \end{proposition}

\begin{proof}[Proof of Proposition \ref{proposition: thickness bound}]
    By Lemma \ref{lemma: all gaps are level k gaps}, we can rewrite the definition of the thickness of $\pi_q(A_q)$ in terms of level $k$ gaps.
    That is, we seek a lower bound on
    $$\tau(\pi_q(A_q)) =  \inf \left\{ \min \left\{ \frac{|L_G|}{|G|} , \frac{|R_G|}{|G|} \right\} : k \in \mathbb{N}, G\  \mathrm{a\ level\ } k \mathrm{\ gap} \right\}.$$

    If $G$ is a level $k$ gap then $|L_G|,|R_G| > q^{-k-4}$ by Lemma \ref{lemma: lower bound on bridge} and $|G| < q^{-k+1}$ by Lemma \ref{lemma: level k gaps decrease}.
    Therefore
    $$\min \left\{ \frac{|L_G|}{|G|}, \frac{|R_G|}{|G|} \right\} > \frac{q^{-k-4}}{q^{-k+1}} = q^{-5}.$$
    Since this is independent of $k \in \mathbb{N}$, we conclude that $\tau(\pi_q(A_q)) > q^{-5}$.
\end{proof}

\subsection{Interleaving}
\label{subsection: interleaving}

Let $C_1$ and $C_2$ be two compact subsets of $\mathbb{R}$ then $C_1$ and $C_2$ are \textit{interleaved} if conv$(C_1) \not\subset \mathbb{R}\setminus C_2 $ and conv$(C_2) \not\subset \mathbb{R} \setminus C_1$.
This is often written more directly as neither set is contained in a gap of the other.
For $C_1$ and $C_2$ to be interleaved, it is sufficient to show that $C_2 \subset \mathrm{conv}(C_1)$ and $|C_2| \geq |G|  $ where $G$ is the largest interval in $\mathrm{conv}(C_1) \setminus C_1$.

In this subsection we prove the following proposition, which will be the consequence of the two lemmas which follow it.

\begin{proposition}
    \label{proposition: interleaving}
    The sets $(2-q)\pi_q(\mathcal{S}^9)+1$ and $\pi_q(A_q)$ are interleaved.
\end{proposition}

\begin{figure}
    \centering
    \begin{tikzpicture}
        \draw[thick] (1,0) -- (4,0);
        \draw[thick] (0,-1) -- (2,-1);
        \draw[thick] (3,-1) -- (5,-1);
        \node[right] at (6,0) {$\mathrm{conv}(\pi_q(A_q))$};
        \node[right] at (6,-1) {$\mathrm{conv}((2-q)\pi_q(\mathcal{S}^9) + 1) \setminus G$};
        \node[below] at (0,-1) {$1$};
        \node[below] at (5,-1) {$\frac{1}{q-1}$};
    \end{tikzpicture}
    \caption{$\pi_q(A_q)$ and $(2-q)\pi_q(\mathcal{S}^9) + 1$ are interleaved. Here, $G$ denotes the largest gap of $((2-q)\pi_q(\mathcal{S}^9) + 1)$.}
    \label{figure: interleaved}
\end{figure}
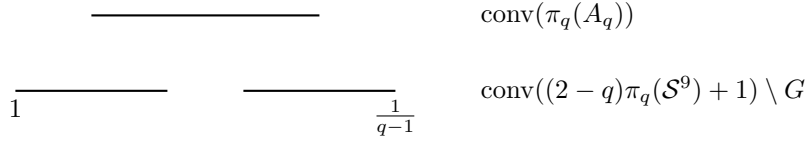

As shown in \cite{sidorov2009expansions} and Subsection \ref{subsection: structure of pi_q(A_q)} respectively, the sets $\pi_q(\mathcal{S}^k)$ and $\pi_q(A_q)$ are each the complement of an interval with a collection of infinitely many gaps of decreasing length. 
We show that $\mathrm{conv}((2-q)\pi_q(\mathcal{S}^9) + 1)$ contains $\pi_q(A_q)$ and that the largest gap of $(2-q)(\pi_q(\mathcal{S}^9) + 1)$ is smaller than $|\pi_q(A_q)|$, which is sufficient for the sets to be interleaved. 
Figure \ref{figure: interleaved} shows $\mathrm{conv}(\pi_q(A_q))$ and $\mathrm{conv}((2-q)\pi_q(\mathcal{S}^9) + 1)$ with the largest gap removed.
Notice that $\pi_q(\mathcal{S}^9)$ contains $0$ and $\frac{1}{q-1}$ because $0 = \pi_q(0^\infty) \in \pi_q(\mathcal{S}^9)$ and $\frac{1}{q-1} = \pi_q(1^\infty) \in \pi_q(\mathcal{S}^9)$. 
Since $0$ and $\frac{1}{q-1}$ are extremal points of $\pi_q(\mathcal{S}^9)$, it is easy to check that $\mathrm{conv}((2-q)\pi_q(\mathcal{S}^9) + 1) = [1,\frac{1}{q-1}]$.

\begin{lemma}
\label{lem:convex hulls contained}
    If $ q \in (q_9,2)$ then $\pi_q(A_q) \subset [1,\frac{1}{q-1}]$.
\end{lemma}

\begin{proof}
Let $q \in (q_9,2)$ and let $\seqj{c}{1}{\infty} \in \{-1,0,1\}^\mathbb{N}$ be the fixed expansion of $1$ (see Subsection \ref{subsection: construction of A_q}). 
Let $\seqj{a}{1}{\infty} \in A_q$ and let $\seqj{b}{1}{\infty}$ be the sequence with $c_j = a_j - b_j$ for all $j \in \mathbb{N}$. 
Using that $\pi_q(\seqj{c}{1}{\infty}) =1$, $\pi_q(\seqj{b}{1}{\infty}) > 0$ and $\pi_q(\seqj{a}{1}{\infty}), \pi_q(\seqj{b}{1}{\infty}) \in \pi_q(\mathcal{S}^9)$, we know that $\pi_q(\seqj{a}{1}{\infty}) > 1$.
Lexicographically, $\seqj{a}{1}{\infty} \in A$ satisfies $\seqj{a}{1}{\infty} \prec 1^\infty$, which implies $\pi_q(\seqj{a}{1}{\infty}) < \frac{1}{q-1}$. 
Therefore $\pi_q(A_q) \subset [1,\frac{1}{q-1}]$.
\end{proof}

\begin{lemma}
\label{lem:gap is small enough}
    If $q \in (q_9,2)$ the diameter of the largest gap of $((2-q)\pi_q(\mathcal{S}^9)+1)$ is smaller than $|\pi_q(A_q)|$.
\end{lemma}

\begin{proof}
    The largest gap of $((2-q)\pi_q(\mathcal{S}^9)+1)$ has the same diameter as the largest gap of $(2-q)\pi_q(\mathcal{S}^9)$.
    We prove the lemma by bounding this gap above and showing that this upper bound is less than a lower bound for $|\pi_q(A_q)|$.
    We start by finding a lower bound for $|\pi_q(A_q)|$.

    Let $q \in (q_9,2)$. 
    Recall from Subsection \ref{subsection: extension of the base q expansion} that $H_q = \left[ -\frac{q}{q^2-1} , \frac{q}{q^2-1} \right]$.
    If $q_k < q \leq q_{k+1}$ then by Lemma \ref{lemma: 2-q inequality}, $0 < f_1^k(1) = q^k - q^{k-1} - \cdots - q -1 \leq \frac{1}{q} < \frac{q}{q^2-1}$, so $f_1^k(1) \in H_q$.
    Let $s$ be the smallest index such that $f_1^s(1) \in H_q$ (so $s \leq k$). 
    Let $\seqj{c}{1}{\infty}$ be the fixed expansion of $1$.
    By inspection of the words $w \in W_2$, if $c_{s+1} \neq 0$ then $c_{s+2} = 0$, so if $j_0$ is the first zero (which is a free zero) of $\seqj{c}{1}{\infty}$ then $j_0 \leq s+2$.
    Since $s\leq k$, $j_0 \leq k+2$.

    Let $(j_{2k})_{k=0}^\infty = \Jfree$ be the sequence of free zeros of $\seqj{c}{1}{\infty}$ and for $l \in \mathbb{N}$ let $\alpha_l = \seqj{a}{j_{2(l-1)}+1}{j_{2l}-1}$ be the values of every sequence in $A_q$ at the fixed indices between $j_{2(l-1)}$ and $j_{2l}$.
    Since there is at least one free zero in every set of five consecutive indices after $j_0$, and free zeros never occupy consecutive indices, we know that $1 \leq |\alpha_l| \leq 4$ for every $ l \in \mathbb{N}$.
    By definition, 
    $$A_q = \{\seqj{a}{1}{j_0-1}a_{j_0}\alpha_1 a_{j_2} \alpha_2 a_{j_4} \ldots : a_{j_{2k}} \in \{0,1\} \ \forall k \in \mathbb{N}_{\geq0}\},$$
    so $\pi_q(A_q)$ is bounded below and above by $L = \pi_q(\seqj{a}{1}{j_0-1}0\alpha_1 0 \alpha_2 0 \ldots )$
    and
    $U = \pi_q(\seqj{a}{1}{j_0-1}1 \alpha_1 1 \alpha_2 1 \ldots )$
    respectively, and $L,U \in \pi_q(A_q)$.
    Hence 
    \begin{align*}
        |\pi_q(A_q)| &= U - L = q^{-j_0+1}\left[q^{-1} + q^{-2 - |\alpha_1|} + q^{-3- |\alpha_1| -|\alpha_2|} + \cdots \right] \\
        &\geq q^{-k-1}\left[q^{-1} + q^{-6} + q^{-11} + \cdots \right] = q^{-k-2}\left[\frac{q^5}{q^5-1}\right] \\
        &> q^{-k-2}.
    \end{align*}
    We compare this lower bound for $|\pi_q(A_q)|$ with an upper bound for the largest gap of $(2-q)\pi_q(\mathcal{S}^9)$.
    Let $q \in (q_9,2)$ and let $G$ be the largest gap of $\pi_q(\mathcal{S}^9)$. 
    A consequence of the arguments in \cite[Proof of Lemma 4.1]{sidorov2009expansions} is that $|G| < q^{-8}$.
    If $q_k < q \leq q_{k+1}$ then $q^{-k-1} \leq ( 2-q) < q^{-k}$ by Lemma \ref{lemma: 2-q inequality}. 
    Therefore the diameter of the largest gap of $ (2-q)\pi_q(\mathcal{S}^9)+1$ is bounded above by $(2-q)|G| < q^{-k-8}$.

    Hence if $q \in (q_9,2)$ and $ q \in (q_k, q_{k+1}]$ then $|\pi_q(A_q)| > q^{-k-2} > q^{-k-8} > (2-q)|G|$.
\end{proof}

We now observe that Proposition \ref{proposition: interleaving} follows from Lemmas \ref{lem:convex hulls contained} and \ref{lem:gap is small enough} combined with the fact that $\mathrm{conv}(\pi_q(\mathcal{S}^9)) = [1,\frac{1}{q-1}]$.

\subsection{Proof of Theorem \ref{theorem: interval in C_3}}

\begin{proof}[Proof of Theorem \ref{theorem: interval in C_3}]
Let $q \in (q_9,2)$. 
By Newhouse's theorem (\cite[Lemma 4]{Newhouse1979}), in order for two compact subsets of $\mathbb{R}$ to have nonempty intersection, it is sufficient to show that the product of the thicknesses of the sets is greater than 1 and that the sets are interleaved, that is, neither set lies entirely within a gap of the other.
Proposition \ref{proposition: interleaving} shows that $(2-q)\pi_q(\mathcal{S}^9) + 1$ and $\pi_q(A_q)$ are interleaved and Proposition \ref{proposition: thickness bound} shows that $\tau(\pi_q(A_q)) > q^{-5}$. 
In \cite{sidorov2009expansions} it is shown that $\tau(\pi_q(\mathcal{S}^k)) > q^{k-3}$ whenever $q_k < q < 2$ for all $k \geq 3$. 
Hence, for $k=9$ we can say that $\tau(\pi_q(\mathcal{S}^9)) > q^6$ whenever $q \in (q_9,2)$. 
Because thickness is preserved under affine transformations, we also know that $\tau((2-q)\pi_q(\mathcal{S}^9) + 1) > q^6$.
Hence we can conclude that $\tau((2-q)\pi_q(\mathcal{S}^9) + 1) \times \tau(\pi_q(A_q)) > q^6 \times q^{-5} > 1$ and by Newhouse's theorem, $((2-q)\pi_q(\mathcal{S}^9)+1) \cap \pi_q(A_q) \neq \emptyset$.

By Lemma \ref{lemma: projection of S^k is contained in U_q}, since $q \in (q_9,2)$, we know
$$((2-q)\pi_q(\mathcal{S}^9)+1) \cap \pi_q(A_q) \neq \emptyset,$$
implies that 
$$((2-q)\mathcal{U}_q + 1) \cap \pi_q(A_q) \neq \emptyset.$$
Hence it is sufficient to show that $((2-q)\mathcal{U}_q + 1) \cap \pi_q(A_q) \neq \emptyset$ implies that $q \in \mathcal{C}_3$.

If $((2-q)\mathcal{U}_q + 1) \cap \pi_q(A_q) \neq \emptyset$ then let $qy \in ((2-q)\mathcal{U}_q + 1) \cap \pi_q(A_q)$ for some $y \in \left[ \frac{1}{q} , \frac{1}{q(q-1)}\right]$. 
Then $\frac{1}{q-1} - \frac{qy-1}{2-q} \in \mathcal{U}_q$\footnote{Recall that $x \in \mathcal{U}_q$ if and only if $\frac{1}{q-1} - x \in \mathcal{U}_q$ and we know that $\frac{qy-1}{2-q} \in \mathcal{U}_q$.}, $qy \in \pi_q(A_q) \subset \mathcal{U}_q $ and $qy-1 \in \pi_q(A_q)-1 \subset \mathcal{U}_q$ by construction. 
Hence $f_0(y),f_1(y), f_2(y) \in \mathcal{U}_q$, so $y$ has three orbits. i.e. $y \in \mathcal{V}_q^{(3)}$ and so $q \in \mathcal{C}_3$.
\end{proof}

\section{Proof of Theorem \ref{theorem: positive Haudsorff dimension}}
\label{section: proof of positive dimension}

In this section we prove Theorem \ref{theorem: positive Haudsorff dimension} and Corollary \ref{corollary: positive Hausdorff dimension}. 
We start by proving that Theorem \ref{theorem: positive Haudsorff dimension} is equivalent to Theorem \ref{theorem: positive dimension variation 1} and that Corollary \ref{corollary: positive Hausdorff dimension} follows from Theorem \ref{theorem: positive dimension variation 1}.
We then show that Theorem \ref{theorem: positive dimension variation 1} is equivalent to Theorem \ref{theorem: positive dimension variation 2} before proving Theorem \ref{theorem: positive dimension variation 2}.

Let $a = \seqj{a}{1}{k} \in \{0,1,2\}^*$ be a finite ternary sequence, then $\mapstring{f}{a}{j}{1}{k}$ is the associated finite sequence of maps.
We write $f_a$ to denote the finite composition of maps $\mapseqn{f}{a}{j}{1}{k}$.
When we do not state it explicitly, $|a|$ denotes the length of the sequence $a$, so in this case $|a| = k$.
Given $q \in (1,2)$, $x \in I_q$, we define
$$\orbsk{E_q}{k}{x} = \{\mapstring{f}{i}{j}{1}{k} \in \{f_0, f_1, f_2 \}^k: \mapseqn{f}{i}{j}{1}{l}(x) \in I_q \mathrm{\ for \ all \ } 1 \leq l \leq k\},$$
$$\orbsk{E_q}{*}{x} = \bigcup_{k \in \mathbb{N}_{\geq 0}} \orbsk{E_q}{k}{x},$$
and the set
$$\mathcal{D} = \{q \in (1,2) : |\orbs{E_q}{x}| \in \{1, 2^{\aleph_0}\} \ \forall x \in I_q\}.$$
If $q \in \mathcal{D}$ then by Lemma \ref{lemma: slice orbit bijection}, $|\slice{q}{y}| \in \{1,2^{\aleph_0}\}$ for all $ y \in [0,1]$.
Therefore Theorem \ref{theorem: positive Haudsorff dimension} is equivalent to Theorem \ref{theorem: positive dimension variation 1}.

\begin{theorem}
\label{theorem: positive dimension variation 1}
    Let $q \in \mathcal{D}$ and let $y \in [0,1]$. 
    If $\slice{q}{y}$ is uncountable then there is some $s > 0$ depending only on $q$ such that $\hdim(\slice{q}{y}) \geq s$.
\end{theorem}

The implications of Corollary \ref{corollary: implications of the bijection lemma}, \cite[Theorem 2.1]{sidorov2009expansions} and \cite[Corollary 1.3]{BAKER2015515} are that $(1,q_{\aleph_0}) \setminus \{G\} \subset \mathcal{D}$ and $\mathcal{T} \cap (q_{\aleph_0}, q_\mathrm{KL}) \subset \mathcal{D}$.
Hence, Corollary \ref{corollary: positive Hausdorff dimension} follows from Lemma \ref{lemma: slice orbit bijection} and Theorem \ref{theorem: positive dimension variation 1}.
Some work is required to prove that Theorem \ref{theorem: positive dimension variation 1} is equivalent to Theorem \ref{theorem: positive dimension variation 2}.

\begin{theorem}
\label{theorem: positive dimension variation 2}
Let $q \in \mathcal{D}$. 
If $x \in J_q$ then there is some $s> 0$ depending only on $q$ such that $\hdim(\slice{q}{x(q-1)}) \geq s$.
\end{theorem}

\begin{proof}[Proof of Theorem \ref{theorem: positive dimension variation 1} $\iff$ Theorem \ref{theorem: positive dimension variation 2}.]
The forwards implication is more straightforward so we prove this first.
Observe that $x \in I_q \iff x(q-1) \in [0,1]$.
Let $x \in J_q$ and let $q \in \mathcal{D}$, then we have the following sequence of implications.
\begin{equation}
\label{equation: dimension variation implications}
x \in J_q \overset{\text{Lemma \ref{lemma: unique orbit if avoids switch}}}{\implies} |\orbs{E_q}{x}| \geq 2 \overset{q \in \mathcal{D}}{\implies} \orbs{E_q}{x} \mathrm{\ uncountable \ } \overset{\text{Lemma \ref{lemma: slice orbit bijection}}}{\iff} \slice{q}{x(q-1)} \mathrm{\ uncountable }.
\end{equation}
The above chain of implications shows that the hypotheses of Theorem \ref{theorem: positive dimension variation 1} hold if the hypotheses of Theorem \ref{theorem: positive dimension variation 2} hold, and since the conclusions of the theorems are the same, we have proved the forwards implication.

For the reverse implication, let $q \in \mathcal{D}$ and let $x(q-1) \in [0,1]$ satisfy $\slice{q}{x(q-1)}$ is uncountable.
By Lemma \ref{lemma: slice orbit bijection} this is equivalent to $\orbs{E_q}{x}$ is uncountable.
By Lemma \ref{lemma: unique orbit if avoids switch} this implies that there is a finite sequence of maps $\mapstring{f}{a}{j}{1}{k} \in \orbsk{E_q}{*}{x}$ which satisfies $\mapseqn{f}{a}{j}{1}{k}(x) \in J_q$.
Applying Theorem \ref{theorem: positive dimension variation 2}, we conclude that there is some $s > 0$ depending only on $q$ such that $\hdim (\slice{q}{\mapseqn{f}{a}{j}{1}{k}(x)(q-1)}) \geq s$.
Hence it suffices to show, for any finite sequence of maps $\mapstring{f}{a}{j}{1}{k} \in \orbsk{E_q}{*}{x}$, that $\hdim(\slice{q}{x(q-1)}) \geq \hdim(\slice{q}{\mapseqn{f}{a}{j}{1}{k}(x)(q-1)})$.

Let $k \in \mathbb{N}_{\geq 0}$, $x \in I_q$ and let $\mapstring{f}{a}{j}{1}{k} \in \orbsk{E_q}{*}{x}$ be a finite sequence of maps.
By the definition of $\orbs{E_q}{x}$,
$$\mapstring{f}{a}{j}{1}{k} \fseq \in \orbs{E_q}{x} \iff \fseq \in \orbs{E_q}{\mapseqn{f}{a}{j}{1}{k}(x)}.$$
Hence by Lemma \ref{lemma: slice orbit bijection},
$$\pi_3(\seqj{i}{1}{\infty}) \in \slice{q}{\mapseqn{f}{a}{j}{1}{k}(x) (q-1)} \iff \pi_3(\seqgen{a}{j}{1}{k} \seqj{i}{1}{\infty} ) \in \slice{q}{x(q-1)}.$$
That is, for some affine map $F : [0,1] \rightarrow [0,1]$ which maps $\pi_3(\seqj{i}{1}{\infty})$ to $\pi_3(\seqj{a}{1}{k} \seqj{i}{1}{\infty})$,
$$F(\slice{q}{\mapseqn{f}{a}{j}{1}{k}(x) (q-1)}) \subset \slice{q}{x(q-1)},$$
so $\hdim(\slice{q}{x(q-1)}) \geq \hdim(\slice{q}{\mapseqn{f}{a}{j}{1}{k}(x)(q-1)})$.
\end{proof}

We outline the logic of the proof of Theorem \ref{theorem: positive dimension variation 2}.

\begin{enumerate}
    \item 
    Let $q \in \mathcal{D}$ and let $x \in J_q$.
    If there is some set $\mathcal{R} \subset \{0,1,2\}^\mathbb{N}$ such that $\pi_3(\mathcal{R}) \subset \slice{q}{x(q-1)}$ then it suffices to prove that $\hdim(\pi_3(\mathcal{R}))\geq s$ for some $s > 0$.
    The set $\mathcal{R}$ will depend on both $q$ and $x$ but $s$ will depend only on $q$.
    
    \item For all $k \in \mathbb{N}_{\geq 0}$ we construct the sets $\mathcal{R}^k \subset \{0,1,2\}^\mathbb{N}$ in terms of a function $
    A: \{0,1\}^* \rightarrow \{0,1,2\}^*$ written $A(\epsilon) = b^\epsilon = \seqj{b^\epsilon}{1}{N_\epsilon}$ which depends on $x$ and $q$.
    Precisely, for all $k \in \mathbb{N}_{\geq 0}$, $\mathcal{R}^k$ is the union of cylinders $\cup_{\epsilon \in \{0,1\}^k} [b^\epsilon]$.
    So elements of $\mathcal{R}^k$ are sequences in $\{0,1,2\}^\mathbb{N}$ which are prefixed by $b^\epsilon$ for some $\epsilon \in \{0,1\}^k$.
        
    \item Define $\mathcal{R} = \cap_{k \in \mathbb{N}_{\geq 0}} \mathcal{R}^k$. Therefore, the set $\pi_3(\mathcal{R})$ is the countable intersection over $k \in \mathbb{N}_{\geq 0}$ of sets $\pi_3(\mathcal{R}^k)$. 
    That is,
    $$\pi_3(\mathcal{R}) = \bigcap_{k \in \mathbb{N}_{\geq 0}} \bigcup_{\epsilon \in \{0,1\}^k} \pi_3[b^\epsilon].$$
    It will follow from the properties of the function $A$ that the $\mathcal{R}^k$ sets are nested and nonempty so the above intersection is nonempty.
    
    \item It is a consequence of the construction of the sets $\mathcal{R}^k$ that for each $k \in \mathbb{N}_{\geq 0}$, if $[b^\epsilon] \subset \mathcal{R}^k$ then $ \mapstringS{f}{b}{\epsilon}{j}{1}{N_\epsilon} \in \orbsk{E_q}{*}{x}$. 
    So if $\seqj{i}{1}{\infty} \in \mathcal{R}$, we know that $\fseq \in \orbs{E_q}{x}$ which provides the property that $\pi_3(\mathcal{R}) \subset \slice{q}{x(q-1)}$ by Lemma \ref{lemma: slice orbit bijection}.
    \item A measure $\mu$ is defined on $\pi_3(\mathcal{R})$ by declaring the value of $\mu(\pi_3[b^\epsilon])$ for all $\epsilon \in \{0,1\}^*$. 
    We will show that this measure $\mu$ satisfies the following property:
    \begin{equation}
    \label{equation: property of measure on R}
        \tag{P}
        \parbox{\dimexpr\linewidth-4em}{%
    \strut
    For every set $U \subset \mathbb{R}$, $\mu(U) \leq 4|U|^s$ where $s = \frac{\log 2}{M \log 3},$
    \strut
  }
    \end{equation}
    where $M \in \mathbb{N}$ is a constant depending only on $q$ which arises from the construction of the $\mathcal{R}^k$ sets.
    (We emphasise that the constant $M$ in this section is different from the $M$ which appears in the definition of the fixed expansion of $1$ in the previous section.)
    Via the mass distribution principle \cite{falconer2004fractal}, the existence of such a measure proves that $\mathcal{H}^s(\pi_3(\mathcal{R})) \geq \frac{1}{4}$ and that $\hdim(\pi_3(\mathcal{R})) \geq s$.
    \item 
    \label{item: R^k has 2^k elements}
    The existence of a measure $\mu$ with property \eqref{equation: property of measure on R} relies on the fact that for any $l \in \mathbb{N}_{\geq 0}$, $\pi_3(\mathcal{R}^l)$ is the union of $2^l$ intervals with disjoint interiors and of diameter at least $3^{-lM}$.
    Using this fact, it is possible to prove that there exists a measure $\mu$ such that whenever $lM \leq k < (l+1)M$ and $\epsilon \in \{0,1\}^k$, $\mu(\pi_3[b^\epsilon]) \leq 2(2^{-\frac{k}{M}})$.
    \item The existence of the $\mathcal{R}^k$ sets with the required property described in Item \ref{item: R^k has 2^k elements} makes use of branching tree constructions and the compactness of the interval $J_q$.
\end{enumerate}

The majority of the work of the proof is in the construction of the $\mathcal{R}^k$ sets.

\subsection{Branching trees}
In this subsection we prove the following proposition.
To do this we follow \cite{BAKER2015515} by constructing the \textit{branching tree} and the \textit{infinite branching tree} for any $x \in I_q$.

\begin{proposition}
\label{proposition: existence of doubly infinite branching sequences}
    Let $q \in \mathcal{D}$. For every $x \in J_q$ there are two distinct finite sequences of maps in $\orbsk{E_q}{*}{x}$ which map $x$ into $\interior{J_q}$.
\end{proposition}

The branching tree constructions below provide a useful visual aid when thinking about the orbit spaces $\orbs{E_q}{x}$ for $x \in I_q$.
This is because for any $x \in I_q$, the set of infinite paths in the branching tree of $x$ is in bijection with $\orbs{E_q}{x}$.
We show how the structure of the infinite branching tree for $x \in I_q$ allows us to distinguish between the cases when $\orbs{E_q}{x}$ is at most countably infinite and when it is uncountable.
Recall from \eqref{equation: dimension variation implications} that the hypotheses of Theorem 4.2 were shown to imply that $\orbs{E_q}{x}$ is uncountable.
We will use the uncountability of $\orbs{E_q}{x}$ to conclude that the infinite branching tree of $x$ has a certain structure.
This forms the first step towards constructing the set $\mathcal{R} \subset \{0,1,2\}^\mathbb{N}$ alluded to in the outline above.
Before defining the branching trees we will need the following definitions.

Let $x \in I_q$ and let $ b = \seqj{b}{1}{N} \in \{0,1,2\}^*$ be a finite ternary sequence.
If $b$ is such that $\mapstring{f}{b}{j}{1}{N} \in \orbsk{E_q}{*}{x}$ and $f_b(x) \in J_q$ then $\mapstring{f}{b}{j}{1}{N}$ is a \textit{branching sequence} for $x$, $f_b(x)$ is a \textit{branching point} of $x$, and $N$ is the \textit{branching length} of $\mapstring{f}{b}{j}{1}{N}$.
In particular, if $x \in J_q$ then if $e \in \{0,1,2\}^*$ is the empty word, $f_e$ is the identity map so $f_e(x) = x \in J_q$ is a branching point and the empty word is a branching sequence for $x$.
A \textit{minimal branching sequence} for $x$ is a branching sequence for $x$ such that there does not exist a branching sequence for $x$ with smaller branching length.
If $ |\orbs{E_q}{x}| > 1 $ then it is obvious that there is a branching sequence and hence also a minimal branching sequence for $x$.

If $x \in I_q$ and $b = \seqj{b}{1}{N} \in \{0, 1, 2\}^*$ is such that $\mapstring{f}{b}{j}{1}{N}$ is a branching sequence for $x$ such that there are at least two indices $i \in \{0,1,2\}$ for which $\orbs{E_q}{f_i(f_b(x))}$ is infinite then $\mapstring{f}{b}{j}{1}{N}$ is said to be a \textit{doubly infinite branching sequence} for $x$, $f_b(x)$ is a \textit{doubly infinite branching point} of $x$ and as before $N$ is the branching length of $\mapstring{f}{b}{j}{1}{N}$.
A \textit{minimal doubly infinite branching sequence} for $x$ is a doubly infinite branching sequence for $x$ such that there does not exist a doubly infinite branching sequence for $x$ with smaller branching length.

\subsubsection{Branching tree $\mathcal{T}(x)$}
We define the \textit{branching tree}, $\mathcal{T}(x)$, for any $x \in I_q$.

Let $x \in I_q$. Suppose $x$ satisfies $|\orbs{E_q}{x}| = 1$ then the branching tree for $x$ is defined to be an infinite horizontal line, representing the unique orbit of $x$ in $E_q$.
If $|\orbs{E_q}{x}| > 1$ then there is a minimal branching sequence $\mapstring{f}{b}{j}{1}{N}$ for $x$, which gives the branching point $f_b(x)$. 
To construct the branching tree $\mathcal{T}(x)$, the transformation $f_b$ is represented by a finite horizontal line which then \textit{bifurcates} or \textit{trifurcates} into two or three branches respectively as follows.
If $f_b(x) \in \interior{J_q}$ then we have a trifurcation and if $f_b(x) \in \partial J_q$ then we have a bifurcation.
In the first case, the three branches of the trifurcation correspond to the images $\{f_0(f_b(x)), f_1(f_b(x)), f_2(f_b(x))\}$ which we define to be the \textit{roots} of the branches.
In the second case, if $f_b(x) = \frac{1}{q}$ then the two branches of the bifurcation correspond to the images $\{f_0(f_b(x)), f_2(f_b(x))\}$ and if $f_b(x) = \frac{1}{q(q-1)}$ then they correspond to $\{f_1(f_b(x)) , f_2(f_b(x))\}$.
This is a formal way of saying that the branching points have two or three branches which correspond to those maps in $\{f_0, f_1, f_2\}$ which satisfy the property that $f_b(x)$ is in the domain of $f_i$, and these points $f_i(f_b(x))$ are the roots of the branches.
Recall that $f_1$ is not defined at $\frac{1}{q}$ and $f_0$ is not defined at $\frac{1}{q(q-1)}$, hence there are only two branches of $\mathcal{T}(x)$ at these points.

If for some $i \in \{0,1,2\}$, $f_i(f_b(x))$ satisfies $|\orbs{E_q}{f_i(f_b(x))}| = 1$ then the branch with root $f_i(f_b(x))$ is extended by an infinite horizontal line. 
If $|\orbs{E_q}{f_i(f_b(x))}| > 1$ then $f_i(f_b(x))$ has a unique minimal branching sequence $\mapstring{f}{c}{j}{1}{L}$ so we extend the branch with root $f_i(f_b(x))$ by a finite horizontal line segment which then bifurcates or trifurcates according to the value of $f_{c}(f_i(f_b(x)))$ as before.
These rules are repeatedly applied to successive branches of the construction which results in a tree we call the \textit{branching tree} of $x$ and denote by $\mathcal{T}(x)$.

For any branching point $f_b(x)$ of $\mathcal{T}(x)$, the branch with root $f_i(f_b(x))$ is the branching tree $\mathcal{T}(f_i(f_b(x)))$ and is a subtree (in the graph theoretical sense) of $\mathcal{T}(x)$.
The branches of $\mathcal{T}(x)$ are in one-to-one correspondence with points of the form $f_i(f_b(x))$ where $f_b(x)$ is a branching point of $x$, $i \in \{0,1,2\}$ and $f_i(f_b(x)) \in I_q$. 
We remark that, from the definition of the branching tree, there is a bijection between the space of infinite paths in $\mathcal{T}(x)$ and the orbit space $\orbs{E_q}{x}$.
We say that a tree or a branch is finite if it contains finitely many branching points and infinite otherwise.

Observe that $|\orbs{E_q}{f_2(\frac{1}{q})}| = |\orbs{E_q}{f_1(\frac{1}{q(q-1)})}| = 1$ and $\frac{1}{q}$, $\frac{1}{q(q-1)}$ are not in the domains of $f_1$ and $f_0$ respectively.
This means that if $f_b(x) \in \partial J_q$, there is at most one index $i \in \{0,1,2\}$ such that $\orbs{E_q}{f_i(f_b(x))}$ is infinite.
With Lemma \ref{lemma: unique orbit if avoids switch} this implies that if $\mapstring{f}{b}{j}{1}{N}$ is a doubly infinite branching sequence for $x$ then $f_b(x) \in \interior{J_q}$.
Heuristically, we are aiming to show that the number of branches grows sufficiently quickly, so this kind of branching can be ignored.
In fact, any finite branches can be ignored.
This motivates the definition of the \textit{infinite branching tree} below.

\begin{figure}
\label{fig:branching tree}

\centering
    % \begin{subfigure}[b]{0.3\textwidth}
        % \centering
            \begin{tikzpicture} [scale = 1]
            \begin{scope}[thick]
        % \begin{scope}[very thick]
            %the four box lines + dotted diagonal
            
            %first branch + bifurcation
            \draw (0,0) -- (3,0);
            \draw (3,-3) -- (3,3);
            
            %second order branches
            \draw (3,3) -- (7,3);
            \draw (3,0) -- (7,0);
            \draw (3,-3) -- (7,-3);
            
            %second order bifurcations
            \draw (7,2) -- (7,4);
            \draw (7,-1) -- (7,1);
            \draw (7,-4) -- (7,-2);
            
            %third order branches
            \draw (7,4) -- (9,4);
            % \draw (7,3) -- (9,3);
            \draw (7,2) -- (9,2);
            \draw (7,1) -- (9,1);
            \draw (7,0) -- (9,0);
            \draw (7,-1) -- (9,-1);
            \draw (7,-2) -- (9,-2);
            \draw (7,-3) -- (9,-3);
            \draw (7,-4) -- (9,-4);
            
            \node [above] at (0,0) {$x$};
            \node [above] at (2.5,0) {$f(x)$};
            \node [above] at (4.5,3) {$f_0 (f(x))$};
            \node [above] at (4.5,0) {$f_1(f(x))$};
            \node [above] at (4.5,-3) {$f_2(f(x))$};
        \end{scope}
        \end{tikzpicture}
    \caption{The start of the branching tree $\mathcal{T}(x)$ of some $x \in I_q$.}
    
\end{figure}
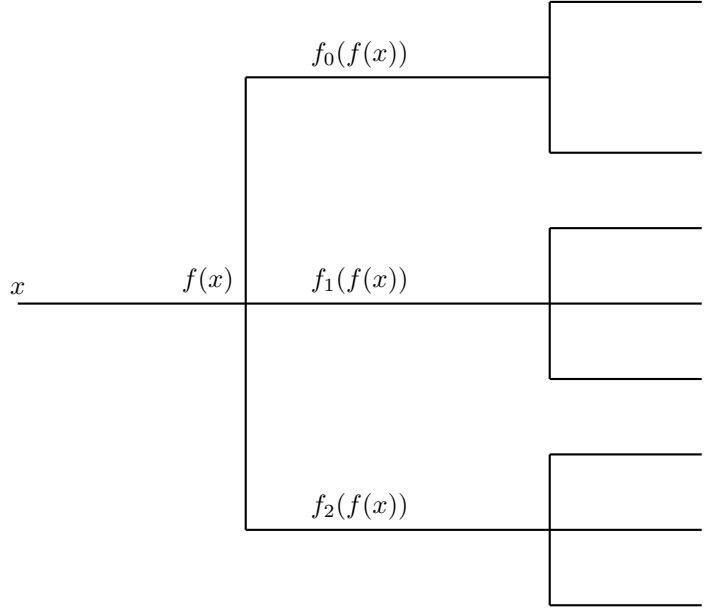

\subsubsection{Infinite branching tree $\mathcal{T}_{\infty}(x)$}
The infinite branching tree for $x \in I_q$ ignores all finite branching behaviour in the sense that only doubly infinite branching points generate branching points in this tree.
We now define the \textit{infinite branching tree}, $\mathcal{T}_\infty(x)$, for any $x \in I_q$.

Suppose $x \in I_q$ has the property that either $\orbs{E_q}{x}$ is finite or for each branching point $f_b(x)$ of $x$ we have at least two indices $i \in \{0,1,2\}$ such that $\orbs{E_q}{f_i(f_b(x))}$ is finite, then the \textit{infinite branching tree for $x$}, $\mathcal{T}_\infty(x)$, is defined to be an infinite horizontal line. 
This is equivalent to $x$ admitting no doubly infinite branching sequence and so there does not exist a doubly infinite branching point of $x$.

Suppose this is not the case, then there is a minimal doubly infinite branching sequence $\mapstring{f}{b}{j}{1}{N}$ for $x$ and $f_b(x)$ is a doubly infinite branching point.
Let $\mapstring{f}{b}{j}{1}{N}$ be a minimal doubly infinite branching sequence, then the infinite branching tree for $x$ consists of a finite horizontal line segment corresponding to the transformation $f_b$ which then bifurcates or trifurcates depending on whether there are two or three indices $i \in \{0,1,2\}$ such that $\orbs{E_q}{f_i(f_b(x))}$ is infinite.
In each case the branches of the bifurcation or trifurcation correspond to the images $f_i(f_b(x))$ that satisfy $\orbs{E_q}{f_i(f_b(x))}$ is infinite, and as before, the points $f_i(f_b(x))$ are defined to be the roots of the branches.
Each branch is then extended by the same rules ad infinitum.

We have the analogous observation that for any branching point $f_b(x)$ of $\mathcal{T}_\infty(x)$, the branch with root $f_i(f_b(x))$ is the infinite branching tree $\mathcal{T}_\infty(f_i(f_b(x)))$ and is a subtree of $\mathcal{T}_\infty(x)$.
Similarly, the branches of $\mathcal{T}_\infty(x)$ are in one-to-one correspondence with points of the form $f_i(f_b(x))$ where $f_b(x)$ is a doubly infinite branching point of $x$ and $\orbs{E_q}{f_i(f_b(x))}$ is infinite.

\subsubsection{Null infinite points}
\label{subsubsection: null infinite points}

It is possible that $x \in I_q$ satisfies $\orbs{E_q}{x}$ is infinite but that for each branching point $f_b(x)$ of $x$, there is exactly one index $i \in \{0,1,2\}$ such that $\orbs{E_q}{f_i(f_b(x))}$ is infinite. In this case $\mathcal{T}(x)$ is infinite, $\mathcal{T}_\infty(x)$ is an infinite horizontal line and $x$ is said to be a \textit{null infinite point}.
If $x$ is a null infinite point then $|\orbs{E_q}{x}| = \aleph_0$.

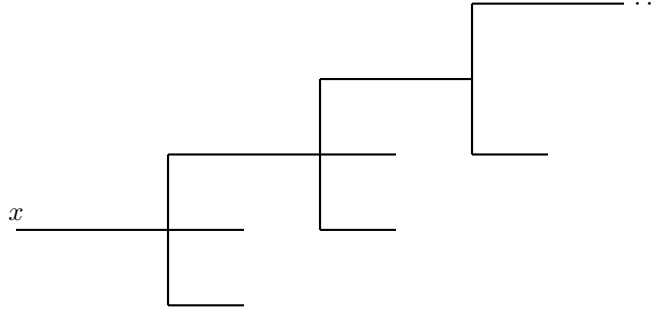
\begin{figure}[h]
\label{fig:infinite branching tree}

\centering
    % \begin{subfigure}[b]{0.3\textwidth}
        % \centering
            \begin{tikzpicture} [scale = 1]
            \begin{scope}[thick]
        % \begin{scope}[very thick]
            %the four box lines + dotted diagonal
            
            %first branch + bifurcation
            \draw (0,0) -- (2,0);
            \draw (2,-1) -- (2,1);
            \draw (2,-1) -- (3,-1);
            \draw (2,0) -- (3,0);
            \draw (2,1) -- (4,1);
            
            % %second order branches
            \draw (4,0) -- (4,2);
            \draw (4,0) -- (5,0);
            \draw (4,1) -- (5,1);
            \draw (4,2) -- (6,2);
            
            % %third order branches
            \draw (6,1) -- (6,3);
            \draw (6,1) -- (7,1);
            % \draw (6,2) -- (7,2);
            \draw (6,3) -- (8,3);
            
            \node [right] at (8,3) {$\cdots$};
            \node [above] at (0,0) {$x$};

        \end{scope}
        \end{tikzpicture}
    \caption{The branching tree $\mathcal{T}(x)$ of a null infinite point $x$. All branching points generate exactly one infinite branch.}
\end{figure}

\subsubsection{Proof of Proposition \ref{proposition: existence of doubly infinite branching sequences}}
For the following proof, we require the above definitions and remarks to deal with the subtle difference between the cases when $\orbs{E_q}{x}$ is countably infinite and when it is uncountably infinite.
Precisely, we require the conclusion that $x$ admits a doubly infinite branching sequence from the hypothesis that $\orbs{E_q}{x}$ is uncountable.

\begin{proof}[Proof of Proposition \ref{proposition: existence of doubly infinite branching sequences}]
    Let $q \in \mathcal{D}$ and let $x \in J_q$, so $|\orbs{E_q}{x}| > 1$.
    Since $q \in \mathcal{D}$, $|\orbs{E_q}{x}| = 2^{\aleph_0}$ so $x$ is not a null infinite point.
    Therefore $x$ admits a minimal doubly infinite branching sequence, $\mapstring{f}{b}{j}{1}{N} \in \orbsk{E_q}{*}{x}$ such that $f_b(x) \in \interior{J_q}$ and there are at least two indices $i_0, i_1 \in \{0,1,2\}$ such that $\orbs{E_q}{f_{i_0}(f_b(x))}$ and $\orbs{E_q}{f_{i_1}(f_b(x))}$ are infinite.
    Again using $q \in \mathcal{D}$, this implies that $|\orbs{E_q}{f_{i_0}(f_b(x))}| = |\orbs{E_q}{f_{i_1}(f_b(x))}| = 2^{\aleph_0}$ so neither $f_{i_0}(f_b(x))$ nor $f_{i_1}(f_b(x))$ are null infinite points.
    Therefore both $f_{i_0}(f_b(x))$ and $f_{i_1}(f_b(x))$ admit minimal doubly infinite branching sequences.
    
    Let $h^0 = \seqj{h^0}{1}{K_0}$ and $h^1 = \seqj{h^1}{1}{K_1}$ be elements of $\{0,1,2\}^*$ such that the associated finite sequences of maps
    $\mapstringS{f}{h}{0}{j}{1}{K_0} \in \orbsk{E_q}{*}{f_{i_0}(f_b(x))}$ and $ \mapstringS{f}{h}{1}{j}{1}{K_1}\in \orbsk{E_q}{*}{f_{i_1}(f_b(x))}$ are minimal doubly infinite branching sequences for $f_{i_0}(f_b(x))$ and $f_{i_1}(f_b(x))$ respectively.
    Therefore $f_{h^0} \circ f_{i_0} \circ f_b(x) \in \interior{J_q}$ and $f_{h^1} \circ f_{i_1} \circ f_b(x) \in \interior{J_q}$.
    Let $b^0 = \seqj{b^0}{1}{N_0}$ and $ b^1 = \seqj{b^1}{1}{N_1}$ be elements of $\{0, 1, 2\}^*$ with $b^0 = \seqj{b}{1}{N} i_0 \seqj{h^0}{1}{K_0}$ and $b^1 = \seqj{b}{1}{N} i_1 \seqj{h^1}{1}{K_1}$.
    Therefore $f_{b^0} = f_{h^0} \circ f_{i_0} \circ f_b$ and $f_{b^1} = f_{h^1} \circ f_{i_1} \circ f_b$.
    We know that $f_{b^0}(x), f_{b^1}(x) \in \interior{J_q}$ and $\mapstringS{f}{b}{0}{j}{1}{N_0}, \mapstringS{f}{b}{1}{j}{1}{N_1} \in \orbsk{E_q}{*}{x}$ by construction and the proposition holds.
\end{proof}

\subsection{Positive Hausdorff dimension}

In this subsection we construct $\mathcal{R} \subset \{0,1,2\}^\mathbb{N}$ and a measure $\mu$ on $\pi_3(\mathcal{R})$ which satisfies the hypotheses of the mass distribution principle and we use this to prove Theorem \ref{theorem: positive dimension variation 2}.
We emphasise that $\mathcal{R}$ depends upon the fixed $q \in \mathcal{D}$ and $x \in J_q$ which are chosen arbitrarily.
We construct the constant $M \in \mathbb{N}$ and the map $A: \{0,1\}^* \rightarrow \{0,1,2\}^*$ before proving their required properties in Proposition \ref{proposition: the constant and the map}.
Recall the definitions of \textit{prefix} and \textit{strict prefix} from Subsection \ref{subsection: background theory}

Let $q \in \mathcal{D}$ and let $x \in J_q$.
By Proposition \ref{proposition: existence of doubly infinite branching sequences} there are two doubly infinite branching sequences $\mapstringS{f}{b}{0}{j}{1}{N_0}, \mapstringS{f}{b}{1}{j}{1}{N_1}  \in \orbsk{E_q}{*}{x}$ such that $f_{b^0}(x), f_{b^1}(x) \in \interior{J_q}$.
The maps $f_0, f_1, f_2$ are continuous, so any finite composition of these maps is also continuous. 
This allows us to fix an open interval $U_x$ containing $x$ such that $f_{b^0}(U_x), f_{b^1}(U_x) \subset \interior{J_q}$.
Since $x \in J_q$ was arbitrary, this generates an open cover $\{U_x\}_{x \in J_q}$ of $J_q$.
By compactness of the closed interval $J_q$, this open cover must admit a finite subcover. That is, there is some finite subset $F \subset J_q$ such that $\{U_{x'}\}_{x' \in F}$ is a finite open cover of $J_q$.
For each $x' \in F$ we can choose a pair of doubly infinite branching sequences $\mapstringS{f}{b}{0}{j}{1}{N_0}, \mapstringS{f}{b}{1}{j}{1}{N_1}$ and define $M_{x'} = \max \{N_0, N_1\}$. 
We subsequently define $M = \max_{x' \in F}\{M_{x'}\}$.

Let $x \in J_q$ be arbitrary. 
Then $x \in U_{x'}$ for some $x' \in F$ and there are doubly infinite branching sequences $\mapstringS{f}{b}{0}{j}{1}{N_0}, \mapstringS{f}{b}{1}{j}{1}{N_1} \in \orbsk{E_q}{*}{x}$ for $x$ such that $\max\{N_0, N_1\} \leq M$.
We call such a pair of sequences of maps the \textit{branching pair} for $x$.
Although the branching pair of some $x \in J_q$ may not be unique, by fixing a branching pair for each $x \in J_q$, the definition is valid.
We emphasise that since $x \in J_q$ was arbitrary, $M$ is independent of $x$ and depends only on $q$.
Since $f_{b^0}(x), f_{b^1}(x) \in \interior{J_q}$, by Proposition \ref{proposition: existence of doubly infinite branching sequences} there are branching pairs for $f_{b^0}(x)$ and $f_{b^1}(x)$. 
That is, there are finite sequences of maps $\mapstringS{f}{c}{0}{j}{1}{K_0} , \mapstringS{f}{c}{1}{j}{1}{K_1} \in \orbsk{E_q}{*}{f_{b_0}(x)}$ and $\mapstringS{f}{d}{0}{j}{1}{L_0} , \mapstringS{f}{d}{1}{j}{1}{L_1} \in \orbsk{E_q}{*}{f_{b_1}(x)}$ with $K_0, K_1, L_0, L_1 \leq M$ such that $f_{c^0} \circ f_{b^0}(x), f_{c^1} \circ f_{b^0}(x), f_{d^0} \circ f_{b^1}(x), f_{d^1} \circ f_{b^1}(x) \in \interior{J_q}$. 
Let $\seqj{b^{00}}{1}{N_{00}}, \seqj{b^{01}}{1}{N_{01}}, \seqj{b^{10}}{1}{N_{10}}, \seqj{b^{11}}{1}{N_{11}}  \in \{0, 1, 2\}^*$ be the finite sequences such that
\begin{align*}
    f_{b^{00}} &= f_{c^0} \circ f_{b^0}, \\
    f_{b^{01}} &= f_{c^1} \circ f_{b^0}, \\
    f_{b^{10}} &= f_{d^0} \circ f_{b^1}, \\
    f_{b^{11}} &= f_{d^1} \circ f_{b^1}.
\end{align*}
Then, by construction, $\mapstringS{f}{b}{{00}}{j}{1}{N_{00}}, \mapstringS{f}{b}{{01}}{j}{1}{N_{01}}, \mapstringS{f}{b}{{10}}{j}{1}{N_{10}}, \mapstringS{f}{b}{{11}}{j}{1}{N_{11}} \in \orbsk{E_q}{*}{x}$.
Again using Proposition \ref{proposition: existence of doubly infinite branching sequences}, we can argue that since, $f_{b^\epsilon}(x) \in \interior{J_q}$ for all $\epsilon \in \{0,1\}^2$, each of these points has a branching pair. 
In general, for $\epsilon \in \{0,1\}^k$ and $B \in \{0,1\}$ we inductively define $\mapstringS{f}{b}{{\epsilon B}}{j}{1}{N_{\epsilon B}} \in \{f_0, f_1, f_2\}^*$ to be a finite sequence of maps which satisfies $f_{b^{\epsilon B}} = f_{g^B} \circ f_{b^{\epsilon}}$ where $f_{g^0}, f_{g^1} \in \orbsk{E_q}{*}{f_{b^\epsilon}(x)}$ is the branching pair for $f_{b^{\epsilon}}(x)$.
This defines $\mapstringS{f}{b}{\epsilon}{j}{1}{N_\epsilon}$, and by construction $\mapstringS{f}{b}{\epsilon}{j}{1}{N_\epsilon} \in \orbsk{E_q}{*}{x}$ for all $\epsilon \in \{0,1\}^*$.

In summary for every $\epsilon \in \{0,1\}^*$ we have an associated finite ternary sequence $b^\epsilon = \seqj{b^\epsilon}{1}{N_\epsilon} \in \{0,1,2\}^*$ with the property that $\mapstringS{f}{b}{\epsilon}{j}{1}{N_\epsilon} \in \orbsk{E_q}{*}{x}$.
We write this mapping as $A : \{0,1\}^* \rightarrow \{0,1,2\}^*$ where $A(\epsilon) = b^\epsilon = \seqj{b^\epsilon}{1}{N_\epsilon}$.
Note that if $\epsilon$ is the empty word then $b^\epsilon$ is the empty word.
Moreover, the map $A$ depends on $x \in J_q$ and $q \in \mathcal{D}$ because the branching pairs in the construction of $b^\epsilon$ depend on $x$ and $q$.

\begin{proposition}
\label{proposition: the constant and the map}
    For any $q \in \mathcal{D}$, the constant $M \in \mathbb{N}$ and the map $A: \{0,1\}^* \rightarrow \{0,1,2\}^*$ satisfy the following properties for any $x \in J_q$.
    \begin{enumerate}
        \item 
        If $\epsilon \in \{0,1\}^*$ then $ \mapstringS{f}{b}{\epsilon}{j}{1}{N_\epsilon} \in \orbsk{E_q}{*}{x}$. 
        \item 
        If $\epsilon, \epsilon' \in \{0,1\}^*$, then $b^\epsilon$ is a prefix of $b^{\epsilon'}$ if and only if $\epsilon$ is a prefix of $\epsilon'$.
        \item 
        If $\epsilon$ and $\epsilon'$ are distinct elements of $\{0,1\}^k$ for some $k \in \mathbb{N}_{\geq 0}$ then $\pi_q[b^\epsilon]$ and $\pi_q[b^{\epsilon'}]$ intersect in at most one point.
        \item 
        For all $k \in \mathbb{N}_{\geq 0}$, if $\epsilon \in \{0,1\}^k$ then $N_\epsilon \leq kM$.
    \end{enumerate}

\end{proposition}

\begin{proof}
    Let $q \in \mathcal{D}$, $x \in J_q$ for every part of this proof.
    \begin{enumerate}
        \item This follows easily from the construction of $\mapstringS{f}{b}{\epsilon}{j}{1}{N_\epsilon}$ for each $\epsilon \in \{0,1\}^*$.
        
        \item 
        Let $\epsilon, \epsilon' \in \{0,1\}^*$.
        If $\epsilon = \epsilon'$ the result is obvious.
        Assume that $\epsilon$ is a \textit{strict} prefix of $\epsilon'$.
        Suppose that $B \in \{0,1\}$ is such that $\epsilon B$ is a prefix of $\epsilon'$ and recall by the inductive definition of $b^\epsilon$ that $f_{b^{\epsilon B}} = f_{g^B} \circ f_{b^\epsilon}$ where $f_{g^0}, f_{g^1}$ is the branching pair of $f_{b^\epsilon}(x)$.
        This equation can be rewritten as
        $$ \mapseqnS{f}{b}{\epsilon B }{1}{N_{\epsilon B}} = f_{g^B} \circ \mapseqnS{f}{b}{\epsilon}{1}{N_\epsilon}.$$
        Hence $\seqj{b^{\epsilon B}}{1}{N_\epsilon} = \seqj{b^\epsilon}{1}{N_\epsilon}$ so $b^\epsilon$ is a prefix of $b^{\epsilon B}$.
        If $\epsilon' = \epsilon B$ then we are done.
        If not then $\epsilon B$ is a strict prefix of $\epsilon'$ and we repeat the above argument.
        Performing this argument finitely many times proves the result in the general case that $\epsilon$ is a strict prefix of $\epsilon'$.

        Suppose $\epsilon$ is not a prefix of $\epsilon'$ and vice versa, then there is some $l \in \mathbb{N}$ such that the length $l$ prefixes of $\epsilon$ and $\epsilon'$ are distinct but the length $l-1$ prefixes coincide.
        Without loss of generality, let $\delta 0$ and $\delta 1$ be the length $l$ prefixes of $\epsilon$ and $\epsilon'$ respectively.
        Then $b^{\delta 0}$ and $b^{\delta 1}$ satisfy $f_{b^{\delta 0}} = f_{g^0} \circ f_{b^\delta}$ and $f_{b^{\delta 1}} = f_{g^1} \circ f_{b^{\delta}}$ where $f_{g^0}, f_{g^1}$ is the branching pair for $f_{b^{\delta}}(x)$.
        Since $f_{g^0}, f_{g^1}$ is a branching pair, we know that $g^0$ is not a prefix of $g^1$ and vice versa.
        Therefore $b^{\delta 0}$ is not a prefix of $b^{\delta 1}$ and vice versa.
        Since $b^{\delta 0}$ and $b^{\delta 1}$ are prefixes of $b^{\epsilon}$ and $b^{\epsilon'}$ respectively, we can conclude that $b^{\epsilon}$ is not a prefix of $b^{\epsilon'}$ and vice versa.

        \item 
        It can be easily checked that intervals of the form $\pi_3[b]$ and $\pi_3[b']$ for $b,b' \in \{0,1,2\}^*$ are either nested or are disjoint apart from possibly at a single point.
        They can only be nested if $b$ is a prefix of $b'$ or vice versa.
        If $\epsilon, \epsilon' \in \{0,1\}^k$ are distinct then $\epsilon$ is not a prefix of $\epsilon'$ and vice versa so by Item 2, $b^\epsilon$ is not a prefix of $b^{\epsilon'}$ and vice versa.
        Therefore $\pi_3[b^\epsilon]$ and $\pi_3[b^{\epsilon'}]$ are not nested and intersect in at most a single point.
        
        \item 
        We prove the claim by induction.
        Let $\mapstringS{f}{b}{0}{j}{1}{N_0}, \mapstringS{f}{b}{1}{j}{1}{N_1} \in \orbsk{E_q}{*}{x}$ be the branching pair for $x$, so $N_0, N_1 \leq M$.
        Therefore the claim holds for $\epsilon \in \{0,1\}$.
        Given $\epsilon \in \{0,1\}^k$ and $b^\epsilon = \seqj{b^\epsilon}{1}{N_\epsilon}$, assume that $N_\epsilon \leq k M$.
        By definition, for any $B \in \{0,1\}$, $f_{b^{\epsilon B}} = f_{g^B} \circ f_{b^\epsilon}$ where $f_{g^0}, f_{g^1}$ is the branching pair for $f_{b^\epsilon}(x)$.
        We know $|g^0| , |g^1| \leq M$ and $N_{\epsilon B} = |g^B| + N_\epsilon$ so $N_{\epsilon B} \leq (k+1)M$.
        Since every $\delta \in \{0,1\}^{k+1}$ is of the form $\epsilon B$ for some $\epsilon \in \{0,1\}^k$ and some $B \in \{0,1\}$, the proof is complete.
    \end{enumerate}
\end{proof}

Fix some $q \in \mathcal{D}$ and $x \in J_q$, on which the construction of the set $\mathcal{R}$ will implicitly depend due to the dependence of the map $A$ on $q$ and $x$.
For each $k \in \mathbb{N}_{\geq 0}$ define the set $\mathcal{R}^k$ by
$$\mathcal{R}^k = \bigcup_{\epsilon \in \{0,1\}^k} [b^\epsilon].$$
Note that the sequences $b^\epsilon$ which define the cylinders of $\mathcal{R}^k$ need not have length $k$ but rather are the images under $A$ of the binary strings of length $k$.
Recall that $A(\epsilon) = b^\epsilon = \seqj{b^\epsilon}{1}{N_\epsilon}$ so the cylinder $[b^\epsilon] \subset \mathcal{R}^k$ has length $N_\epsilon$.
Define also
$$\mathcal{R} = \bigcap_{k \in \mathbb{N}_{\geq 0}} \mathcal{R}^k.$$
That is, $\mathcal{R}$ is the set of infinite sequences $\seqj{i}{1}{\infty} \in \{0,1,2\}^\mathbb{N}$ with the property that for infinitely many $l \in \mathbb{N}$, $\seqj{i}{1}{l} \in \mathcal{R}^k$ for some $k \in \mathbb{N}$.
Equivalently, $\mathcal{R}$ is the set of infinite sequences $\seqj{i}{1}{\infty}$ such that 
for all $k \in \mathbb{N}_{\geq 0}$ there is some $\epsilon \in \{0,1\}^*$ with the property that
$b^\epsilon$ is a prefix of $\seqj{i}{1}{k}$ and for some $B \in \{0,1\}$, $\seqj{i}{1}{k}$ is a strict prefix of $b^{\epsilon B}$.
We emphasise that $\mathcal{R}$ is implicitly dependent upon the fixed $q \in \mathcal{D}$ and $x \in J_q$ due to the dependence of the map $A$ on these values.
By the second part of Proposition \ref{proposition: the constant and the map}, the sets $\mathcal{R}^k$ are nested, and since $\mathcal{R}^k$ is nonempty for all $k \in \mathbb{N}_{\geq 0}$, we know that $\mathcal{R}$ is nonempty.
Using the fact that $ \mapstringS{f}{b}{\epsilon}{j}{1}{N_\epsilon} \in \orbsk{E_q}{*}{x}$, a simple consequence of the definition of the orbit space is that if $\seqj{i}{1}{k}$ is a prefix of $b^\epsilon$ then $\mapstring{f}{i}{j}{1}{k} \in \orbsk{E_q}{*}{x}$.

By the above remarks it follows that $\fseq \in \orbs{E_q}{x}$ for all $\seqj{i}{1}{\infty} \in \mathcal{R}$. 
Hence, by Lemma \ref{lemma: slice orbit bijection}, $\pi_3(\mathcal{R}) \subset \slice{q}{x(q-1)}$.
Moreover we know that 
$$\pi_3(\mathcal{R}) = \bigcap_{k \in \mathbb{N}_{\geq 0}} \bigcup_{\epsilon \in \{0,1\}^k} \pi_3[b^\epsilon],$$
so we define a measure $\mu$ on $\pi_3(\mathcal{R})$ by defining the values of $\mu(\pi_3[b^\epsilon])$ for all $\epsilon \in \{0,1\}^*$.
Given $\epsilon \in \{0,1\}^k$, define
$$\mu(\pi_3[b^\epsilon]) = 2^{-k}.$$

By Proposition \ref{proposition: the constant and the map} we know that if $\epsilon, \epsilon' \in \{0,1,2\}^k$ are distinct then $\pi_3[b^\epsilon]$ and $\pi_3[b^{\epsilon'}]$ intersect in at most one point, and therefore $\mu(\pi_3[b^\epsilon] \cup \pi_3[b^{\epsilon'}]) = \mu(\pi_3[b^\epsilon]) + \mu(\pi_3[b^{\epsilon'}])$.
Then
$$\mu(\pi_3(\mathcal{R}^k)) = \sum_{\epsilon \in \{0,1\}^k} \mu(\pi_3[b^\epsilon]) = 1,$$
since $|\{0,1\}^k| = 2^k$.
It follows that $\mu(\pi_3(\mathcal{R})) = 1$.
It is a consequence of Lemma \ref{lemma: subsequences are interval containments} that if $b^\epsilon$ is a prefix of $\seqj{i}{1}{l} \in \{0,1,2\}^*$ then $\pi_3[\seqj{i}{1}{l}] \subset \pi_3[b^\epsilon]$ so $\mu(\pi_3[\seqj{i}{1}{l}]) \leq  \mu(\pi_3[b^\epsilon])$.

\begin{proposition}
\label{proposition: mu on intervals is bounded above}
    The measure $\mu$ has the property that for all $U \subset \mathbb{R}$, $\mu(U) \leq 4|U|^s$ where $s = \frac{\log 2}{M \log 3}$.
\end{proposition}

\begin{proof}
    We start by bounding the measure of intervals of the form $\pi_3[\seqj{i}{1}{l}]$ where $ \seqj{i}{1}{l} \in \{0,1,2\}^*$ before generalising to arbitrary intervals in $\mathbb{R}$.
    
    Let $\seqj{i}{1}{l} \in \{0,1,2\}^*$ be such that $\seqj{i}{1}{l}$ is not a prefix of $b^\epsilon$ for any $\epsilon \in \{0,1\}^*$.
    Then $\seqj{i}{1}{l}$ is not a prefix of any sequence in $\mathcal{R}$ and hence $\pi_3[\seqj{i}{1}{l}] \cap \pi_3(\mathcal{R})$ is either empty or consists of a single point, so $\mu(\pi_3[\seqj{i}{1}{l}]) = 0$.
    Suppose instead that there is some $k \in \mathbb{N}_{\geq 0}$, $\epsilon \in \{0,1\}^k$ and some $B \in \{0,1\}$ with the property that $\seqj{i}{1}{l}$ is a strict prefix of $b^{\epsilon B}$.
    We can choose $k$ to be minimal so that $b^\epsilon$ is a prefix of $\seqj{i}{1}{l}$.
    Therefore $\mu(\pi_3[\seqj{i}{1}{l}]) \leq \mu(\pi_3[b^\epsilon]) \leq 2^{-k}$.
    Since $\seqj{i}{1}{l}$ is a strict prefix of $b^{\epsilon B}$ we know, using the third part of Proposition \ref{proposition: the constant and the map}, that $l < N_{\epsilon B} \leq (k+1)M$, so $2^{-k} = 2(2^{-k-1}) \leq 2(2^{-l/M})$.
    Therefore, for any $\seqj{i}{1}{l} \in \{0,1,2\}^*$, $\mu(\pi_3[\seqj{i}{1}{l}]) \leq 2(2^{-l/M})$.

    Let $U \subset \mathbb{R}$ satisfy $3^{-l-1} \leq |U| < 3^{-l}$ for some $l \in \mathbb{N}$. 
    Then since $|\pi_3[\seqj{i}{1}{l}]| = 3^{-l}$ for any $\seqj{i}{1}{l} \in \{0,1,2\}^l$, $U$ intersects at most two intervals of the form $\pi_3[\seqj{i}{1}{l}]$, so $\mu(U) \leq 4(2^{-l/M})$.
    Solving $2^{-l/M} = (3^{-l-1})^{s_l}$ gives $s_l = \frac{l}{l+1}\cdot \frac{\log 2}{M \log 3} < s = \frac{\log 2}{M \log 3}$.
    Therefore $\mu(U) \leq 4(3^{-l-1})^s \leq 4|U|^s$.
\end{proof}

Note that since $s$ is a constant depending only on $M$, and $M$ is dependent only on $q$, we know that $s$ is independent of $x \in J_q$ and depends only on $q \in \mathcal{D}$.

\begin{proof}[Proof of Theorem \ref{theorem: positive dimension variation 2}]
    By the mass distribution principle \cite{falconer2004fractal}, Proposition \ref{proposition: mu on intervals is bounded above}, implies that $\mathcal{H}^s(\pi_3(\mathcal{R})) \geq \frac{\mu(\pi_3(\mathcal{R}))}{4} = \frac{1}{4}$ and $\hdim(\pi_3(\mathcal{R})) \geq s$.
    Given $q \in \mathcal{D}$ and $x \in J_q$, we know that $\pi_3(\mathcal{R}) \subset \slice{q}{x(q-1)}$ so we conclude that $\hdim(\slice{q}{x(q-1)} \geq s$ which proves Theorem \ref{theorem: positive dimension variation 2}.
\end{proof}
 
\section{The special case of the $k$-Bonacci numbers}
\label{section: k-Bonacci results}

In this section we prove Theorem \ref{theorem: k-Bonacci numbers admit any odd slice} as well as an additional result (Theorem \ref{theorem: U = C_2}) in the special case where $q$ is a $k$-Bonacci number for some $k \in \mathbb{N}_{\geq 3}$.
We recall the notation from Subsection \ref{subsection: base q expansions}.
By Lemma \ref{lemma: slice orbit bijection}, Theorem \ref{theorem: k-Bonacci numbers admit any odd slice} is equivalent to the following theorem.

\begin{theorem}
\label{theorem: k-Bonacci numbers admit any odd slice variation}
    Let $ \{q_i\}_{i=3}^\infty$ be the set of $k$-Bonacci numbers excluding $G$.
    \begin{enumerate}
        \item If $q \in \{q_i\}_{i=3}^\infty$ then for all $m \in \mathbb{N}$, $\hdim(\mathcal{V}_q^{(2m+1)}) > \epsilon > 0$ where $\epsilon$ depends only on $q$.
        \item $ \{q_i\}_{i=3}^\infty \subset \mathcal{C}_{\aleph_0}$.
    \end{enumerate}
\end{theorem}
Note that $\{q_i\}_{i=3}^\infty \subset \mathcal{C}_{2m+1}$ is an immediate implication of part 1 of the theorem.
In order to prove the theorem we will need the following lemma on the cardinalities of orbit spaces.

\begin{lemma}
\label{lemma: uniquely maps to}
    Let $q \in (G,2) $ and let $\alpha \in \{0,1\}^\mathbb{N}$ be arbitrary.
    If $x_n = \pi_q(1^n 0 \alpha)$ for some $n \in \mathbb{N}$ then $|\orbs{E_q}{x_n}| = |\orbs{E_q}{x_{n-1}}| = \cdots = |\orbs{E_q}{x_1}|$.
\end{lemma}

\begin{proof}
    Let $\alpha \in \{0,1\}^\mathbb{N}$ be fixed and $x_n = \pi_q(1^n 0 \alpha)$ for all $n \in \mathbb{N}$.
    It is obvious that $q^{-1} + q^{-2} \leq x_2 < x_3 < \ldots$ and since $q \in (G,2)$ implies that $q^{-1} + q^{-2} > \frac{1}{q(q-1)}$, we know that 
    \begin{equation}
    \label{equation: maps uniquely to}
    x_n \in \left(\frac{1}{q(q-1)}, \frac{1}{q-1}\right],
    \end{equation}
    for all $n \geq 2$.    
    \eqref{equation: maps uniquely to} implies that $f_2$ is the only map in $\{f_0,f_1, f_2\}$ that satisfies $f_i(x_n) \in I_q$. By observing that $f_2(x_n) = x_{n-1}$, it is clear that since \eqref{equation: maps uniquely to} holds for all $n \geq 2$, $|\orbs{E_q}{x_n}| = |\orbs{E_q}{x_{n-1}}| = \cdots = |\orbs{E_q}{x_1}|$.
\end{proof}

Recall the reflection notation $\overline{x}$ from Subsection \ref{subsection: structure of pi_q(A_q)}.
We observe that since $\frac{1}{q-1} = \pi_q(1^\infty)$, 
\begin{equation}
\label{equation: reflections of sequences}
\pi_q(\overline{x}) = \frac{1}{q-1} - \pi_q(x).
\end{equation}

\begin{lemma}
\label{lemma: result on expansion of f_1 map}
    Let $k \geq 3$, $q = q_k$, and $x_m = \pi_q(1(0^k)^m \alpha)$ where $m \in \mathbb{N}$ and $\alpha \in \{0,1\}^\mathbb{N}$ is arbitrary. 
    Then $f_1(x_m) = \pi_q((1^k)^{m-1}\overline{\alpha})$.
\end{lemma}

\begin{proof}
    First observe that for any $x$ in the domain of $f_1$, that $f_1(x) = \frac{1}{q-1} - \frac{1}{2-q}f_2(x)$.
    If $q = q_k$ for some $k \geq 3$ then by the proof of Lemma \ref{lemma: 2-q inequality}, $\frac{1}{2-q} = q^k$.
    Since $f_2(x_m) = \pi_q((0^k)^m \alpha)$ we know that $\frac{1}{2-q}f_2(x_m) = \pi_q((0^k)^{m-1} \alpha)$.
    Using \eqref{equation: reflections of sequences}, $f_1(x_m) = \frac{1}{q-1} - \pi_q((0^k)^{m-1}\alpha) = \pi_q((1^k)^{m-1}\overline{\alpha})$.
\end{proof}

\begin{figure}[t]
\centering
    \begin{subfigure}[t]{0.45\textwidth}
    \centering
    \vskip 0pt
        \begin{tikzpicture} [scale = 5.85]
            \begin{scope}[thick]
            \draw (0,0) -- (1.1915,0);
            \draw (0,0) -- (0,1.1915);
            \draw (1.1915,0) -- (1.1915,1.1915);
            \draw (0,1.1915) -- (1.1915,1.1915);
            
            \draw [dotted] (0,0) -- (1.1915,1.1915);
            
            \draw (0,0) -- (0.6478,1.1915);
            \draw (0.5437,0) -- (1.1915,1.1915);
            \draw (0.6478,0) -- (0.5437,1.1915);
            
            \draw [dotted] (0.6478,1.1915) -- (0.6478,0);
            \draw [dotted] (0.5437,0) -- (0.5437,1.1915);
            
        % \end{scope}
        
        \draw [red] (0.5804,0) -- (0.5804,1.0674);
        
        \draw [red] (0.5804,1.0674) -- (1.0674,1.0674);
        
        \draw [blue] (0.5804,0.4196) -- (0.4196,0.4196);
        
        \draw [red] (0.5804,0.0674) -- (0.0674,0.0674);
        
        \draw [blue] (1.0674,1.0674) -- (1.0674,0.9633);
        \draw [blue] (1.0674,0.9633) -- (0.9633,0.9633);
        \draw [blue] (0.9633,0.9633) -- (0.9633,0.7718);
        \draw [blue] (0.9633,0.7718) -- (0.7718,0.7718);
        \draw [blue] (0.7718,0.7718) -- (0.7718,0.4196);
        \draw [blue] (0.7718,0.4196) -- (0.5804,0.4196);
        
        \draw [blue] (0.4196,0.4196) -- (0.4196,0.7718);
        \draw[blue] (0.4196,0.7718) -- (0.5804,0.7718);
        \draw [red] (0.5804,0.7718) -- (0.7718,0.7718);
        
        \draw [blue] (0.0674,0.0674) -- (0.0674,0.1240);
        \draw [blue] (0.0674, 0.1240) -- (0.1240,0.1240);
        \draw [blue] (0.1240,0.1240) -- (0.1240,0.2281);
        \draw [blue] (0.1240,0.2281) -- (0.2281,0.2281);
        \draw [blue] (0.2281,0.2281) -- (0.2281,0.4196);
        \draw [blue] (0.2281,0.4196) -- (0.4196,0.4196);

        \node[above] at (0.3,1.195){\small $f_0$};
    \node[above] at (0.6,1.195){\small $f_1$};
    \node[above] at (0.9,1.195){\small $f_2$};

    \node[below] at (0,0){\small $0$};
    \node[below] at (0.5437,0){\small $\frac{1}{q}$};
    \node[below] at (0.6478,0){\small $\frac{1}{q(q-1)}$};
    \node[below] at (1.195,0){\small $\frac{1}{q-1}$};

    \filldraw[white] (0.5437,1.195) circle (0.2pt);
    \filldraw[white] (0.6478,1.195) circle (0.2pt);
    \draw (0.5437,1.190) circle (0.2pt);
    \draw (0.6478,1.190) circle (0.2pt);
    \filldraw[black] (0.5437,0) circle (0.2pt);
    \filldraw[black] (0.6478,0) circle (0.2pt);
    \filldraw[black] (0,0) circle (0.2pt);
    \filldraw[black] (1.19,1.19) circle (0.2pt);

        \end{scope}
    \end{tikzpicture}
    \caption{The three orbits of the point $x_1$.}
    \label{figure: point with orbit three}
    \end{subfigure}
    \hfill
    \begin{subfigure}[t]{0.45\textwidth}
    \centering
    \vskip 5mm
        \begin{tikzpicture}[scale = 0.7, very thin]
    \okfun{8}{0.5436890};
    \draw[line width = 0.4mm, white] (0,0) -- (4.99,4.99);
    \draw[line width = 0.4mm, white](5.01,5.01) -- (10,10);
    \draw [line width = 0.001mm, red] (0,4.870856) -- (10,4.870856);
    \draw[thick] (0,0) -- (0,10) -- (10,10) -- (10,0) -- (0,0);
    \end{tikzpicture}
    \caption{A horizontal slice with three elements.}
    \end{subfigure}
    \caption{$E_q$ and the corresponding fractal $K_q$ for $q = q_3$.}
\end{figure}
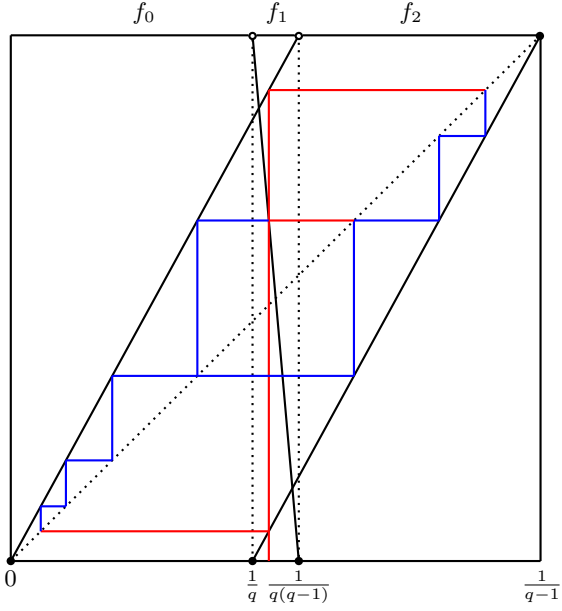
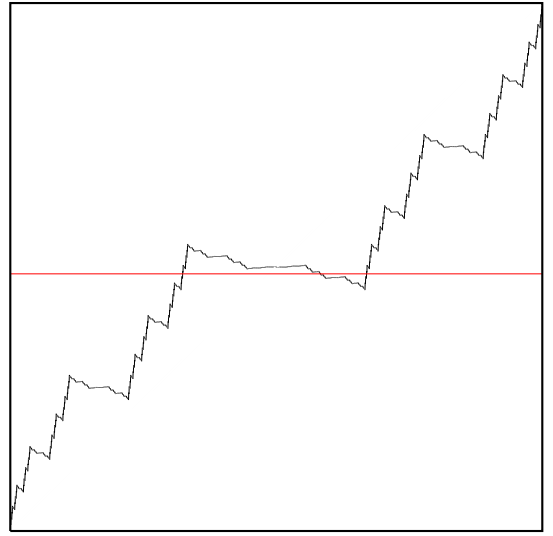

Before proving Theorem \ref{theorem: k-Bonacci numbers admit any odd slice variation}, we require the following definition and lemma.
For all $k \in \mathbb{N}$, define
$$\hat{\mathcal{S}}^k = \{\seqj{i}{1}{\infty} \in \{0,1\}^\mathbb{N} : \seqj{i}{1}{\infty} \mathrm{\ avoids \ } (01^k) \mathrm{\ and \ }(10^k) \mathrm{\ and \ does \ not \ end \ with \ } (01^{k-1})^\infty \mathrm{\ or \ } (10^{k-1})^\infty \},$$
and
$$\tilde{\mathcal{S}}^k = \{\seqj{i}{1}{\infty} \in \{0,1\}^\mathbb{N} : \seqj{i}{1}{\infty} \mathrm{\ avoids \ } (0^k) \mathrm{\ and \ }(1^k) \mathrm{\ and \ does \ not \ end \ with \ } (01^{k-1})^\infty \mathrm{\ or \ } (10^{k-1})^\infty \},$$
and note that $\tSk{k} \subset \hat{\mathcal{S}}^k \subset \mathcal{S}^k$ so $\pi_q(\tSk{k}) \subset \pi_q(\hat{\mathcal{S}}^k) \subset \pi_q(\mathcal{S}^k)$ for each $k \in \mathbb{N}$.
Note also that $\delta \in \mathcal{S}^k$ if and only if $\overline{\delta} \in \mathcal{S}^k$ and the same holds for $\hat{\mathcal{S}}^k$ and $\tilde{\mathcal{S}}^k$. 

By \cite[Lemma 4]{glendinning2001unique}, if $q = q_k$, then $\pi_q(\hat{\mathcal{S}}^k) = \mathcal{U}_q$ and so if $q \in (q_k,2)$ then $\pi_q(\hat{\mathcal{S}}^k) \subset \mathcal{U}_q$. 
This is because $\mathcal{U}_q$ is increasing with $q$ in the sense that if $\pi_q(\seqj{i}{1}{\infty}) \in \mathcal{U}_q$ for some $q \in (1,2)$ then $\pi_{q'}(\seqj{i}{1}{\infty}) \in \mathcal{U}_{q'}$ for all $q' \in (q,2)$ \cite[Lemma 4]{glendinning2001unique}.
Observe that if $\delta \in \tSk{k}$ then $0^n \delta \in \hat{\mathcal{S}}^k$ for all $n \in \mathbb{N}$.
This proves the following lemma.

\begin{lemma}
\label{lemma: uniquely maps to using delta}
    Let $q \in (q_k,2)$ and $\delta \in \tSk{k}$.
    Then $|\orbs{E_q}{\pi_q(0^n \delta)}| = |\orbs{E_q}{\pi_q(\delta)}| = 1$ for all $n \in \mathbb{N}$.
\end{lemma}
 
Since $\pi_q(\mathcal{S}^k) \setminus \pi_q(\hat{\mathcal{S}}^k)$ and $\pi_q(\hat{\mathcal{S}}^k) \setminus \pi_q(\tSk{k})$ have countably many points, we know that for any $q \in (1,2)$,
$$\hdim(\pi_q(\mathcal{S}^k)) = \hdim(\pi_q(\hat{\mathcal{S}}^k)) = \hdim(\pi_q(\tSk{k})).$$
It was shown in \cite[Theorem 2]{glendinning2001unique} that $\hdim(\pi_q(\mathcal{S}^k))>0$ for all $q \in (q_\mathrm{KL}, 2)$, $k \in \mathbb{N}$ and hence $\hdim(\pi_q(\tSk{k})) > 0$.
We note that $q_3 \approx 1.83929 > q_{\mathrm{KL}} \approx 1.78723$, so the range in which this result holds contains $\{q_i\}_{i=3}^\infty$.

\begin{proof}[Proof of Theorem \ref{theorem: k-Bonacci numbers admit any odd slice variation}]
    Proof of part 1:
    
    Let $k \geq 3$ and let $q = q_k$. 
    Define
    $$X_m = \{(1(0^k)^m \delta) \in \{0,1\}^\mathbb{N} : \delta \in \tSk{k}\}.$$
    Since an affine image of $\pi_q(\tSk{k})$ is contained in $\pi_q(X_m)$ for any $m \in \mathbb{N}$, we know that $\hdim(\pi_q(X_m)) \geq \hdim(\pi_q(\tSk{k}))> \epsilon > 0 $ for some $\epsilon$ which depends only on $q$.
    Hence, it suffices to show that if $x \in \pi_q(X_m)$ then $|\orbs{E_q}{x}| = 2m+1$.

    Since $q = q_k$ for some $k \geq 3$,
    $$q^k = q^{k-1} + \cdots + q + 1,$$
    and equivalently
    $$q^{-1} = q^{-2} + \cdots + q^{-k-1},$$
    which implies that
    \begin{equation}
    \label{equation: k-Bonacci base q expansion result}
    \pi_q(1 0^k \alpha ) = \pi_q(0 1^k \alpha),
    \end{equation}
    where $\alpha \in \{0,1\}^\mathbb{N}$ is an arbitrary infinite binary sequence.
    Let $\delta \in \tSk{k}$, and let $x_m = \pi_q(1(0^k)^m \delta)$ for all $m \in \mathbb{N}$. 
    \eqref{equation: k-Bonacci base q expansion result} implies that $\pi_q(1(0^k)^m \delta) = \pi_q(0(1^k)(0^k)^{m-1}\delta)$ and with Lemma \ref{lemma: result on expansion of f_1 map} we evaluate $f_0(x_m), f_1(x_m)$ and $f_2(x_m)$,
    \begin{align*}
    f_0(x_m) &= \pi_q(1^k (0^k)^{m-1} \delta),\\
    f_1(x_m) &= \pi_q((1^k)^{m-1}\overline{\delta}),\\
    f_2(x_m) &= \pi_q((0^k)^m \delta),
    \end{align*}
     and proceed by induction.
    By Lemma \ref{lemma: uniquely maps to using delta}, we can see that $|\orbs{E_q}{f_2(x_m)}| = |\orbs{E_q}{f_1(x_m)}| = 1$ and by Lemma \ref{lemma: uniquely maps to}, $|\orbs{E_q}{f_0(x_m)}| = |\orbs{E_q}{\pi_q(1(0^k)^{m-1}\delta)}| = |\orbs{E_q}{x_{m-1}}| $.
    Since $|\orbs{E_q}{x_m}| = |\orbs{E_q}{f_0(x_m)}| + |\orbs{E_q}{f_1(x_m)}| + |\orbs{E_q}{f_2(x_m)}|$, this means that $|\orbs{E_q}{x_m}| = 2 + |\orbs{E_q}{x_{m-1}}|$.
    For $x_1$, Lemma \ref{lemma: uniquely maps to} shows that $|\orbs{E_q}{f_0(x_1)}| = |\orbs{E_q}{f_1(x_1)}| = |\orbs{E_q}{f_2(x_1)}| = 1$.
    Hence, $|\orbs{E_q}{x_m}| = 2m+1$ for all $m \in \mathbb{N}$, or equivalently, $\pi_q(X_m) \subset \mathcal{V}_q^{(2m+1)}$.
    Therefore, if $q =q_k$ for some $k \geq 3$ and $m \in \mathbb{N}$, since $\pi_q(X_m) \subset \mathcal{V}_q^{(2m+1)}$ and $\hdim(\pi_q(X_m)) > \epsilon > 0$, we have proved that $\hdim(\mathcal{V}_q^{(2m+1)}) > \epsilon > 0$.

    Proof of part 2:

    We show that $x_{\aleph_0} = \pi_q(10^\infty) \in \mathcal{V}_q^{(\aleph_0)}$ for all $q \in \{q_i\}_{i=3}^\infty$.
    Let $q = q_k$ for some $k \geq 3$. 
    Using \eqref{equation: k-Bonacci base q expansion result} we know that 
    \begin{equation}
    \label{equation: k-Bonacci implication countable orbit}
    \pi_q(10^\infty) = \pi_q(01^k 0^\infty).
    \end{equation}
    Since $x_{\aleph_0} = 1/q$, we know that $f_0(x_{\aleph_0}), f_2(x_{\aleph_0}) \in I_q$ but $f_1(x_{\aleph_0})$ is not defined.
    We count the number of orbits of $f_0(x_{\aleph_0})$ and $f_2(x_{\aleph_0})$.
    $f_2(x_{\aleph_0}) = \pi_q(0^\infty)$ so $|\orbs{E_q}{f_2(x_{\aleph_0}})| = 1$. 
    \eqref{equation: k-Bonacci implication countable orbit} implies that $f_0(x_{\aleph_0}) = \pi_q(1^k 0^\infty)$. 
    Hence, by Lemma \ref{lemma: uniquely maps to}, $|\orbs{E_q}{\pi_q(1^k 0^\infty)}| = |\orbs{E_q}{\pi_q(1 0^\infty)}| = |\orbs{E_q}{x_{\aleph_0}}|$. 
    Therefore $|\orbs{E_q}{x_{\aleph_0}}| = |\orbs{E_q}{x_{\aleph_0}} \cup \orbs{E_q}{f_2(x_{\aleph_0}})|$ so $\orbs{E_q}{x_{\aleph_0}}$ is infinite.
    At each branching point of $x_{\aleph_0}$ the maps $f_i \in \{f_0, f_1, f_2\}$ which satisfy $f_i(x_{\aleph_0}) \in I_q$ are $f_0$ and $f_2$.
    These maps have the property that $|\orbs{E_q}{f_0(x_{\aleph_0})}|$ is infinite and $|\orbs{E_q}{f_2(x_{\aleph_0}}| = 1$.
    Recalling the definition from Subsection \ref{subsubsection: null infinite points}, we have that $x_{\aleph_0}$ is a null infinite point so $|\orbs{E_q}{x_{\aleph_0}}|$ is countably infinite, that is $x_{\aleph_0} \in \mathcal{V}_q^{(\aleph_0)}$.
\end{proof}

We note that in general, the sets $\mathcal{C}_{2m}$ for $m \in \mathbb{N}$ are harder to understand because the only points $x \in I_q$ for which it is possible that $|\orbs{E_q}{x}| = 2$ are $x \in \{\frac{1}{q}, \frac{1}{q(q-1)}\}$.
Recalling that $\mathcal{U}$ is the set of $q \in (1,2)$ for which the base $q$ expansion of $1$ is unique, we have the following theorem.

\begin{theorem}
\label{theorem: U = C_2}
    $\mathcal{C}_2 = \mathcal{U}$.
\end{theorem}

\begin{proof}
    Let $q \in \mathcal{U}$. 
    Then $f_0(\frac{1}{q}) = 1$, $f_1(\frac{1}{q})$ is not defined and $f_2(\frac{1}{q}) = 0$, so $|\orbs{E_q}{\frac{1}{q}}| = 2$ and $q \in \mathcal{C}_2$.

    Let $q \in \mathcal{C}_2$. 
    If $x \in \mathcal{V}_q^{(2)}$ then by Lemma \ref{lemma: unique orbit if avoids switch}
    there is some unique finite sequence $\seqj{i}{1}{k} \in \{0,1,2\}^*$ such that $\mapstring{f}{i}{j}{1}{k}(x) \in J_q$.
    Moreover, we can observe that if $\mapstring{f}{i}{j}{1}{k}(x) \in \interior{J_q}$ then $|\orbs{E_q}{x}| \geq 3$.
    Hence $\mapstring{f}{i}{j}{1}{k}(x) \in \partial J_q = \{\frac{1}{q} , \frac{1}{q(q-1)} \}$.
    Without loss of generality, let $\mapstring{f}{i}{j}{1}{k}(x) = \frac{1}{q}$, then $\hat{f}_0(\frac{1}{q}) = f_0(\frac{1}{q}) = 1$ and $f_2(\frac{1}{q}) = 0$.
    Since $|\orbs{E_q}{\frac{1}{q}}| = 2$, we know $|\orbs{E_q}{1}| = 1$ so $|\orbs{\hat{E}_q}{1}| = 1$ and $q \in \mathcal{U}$.
\end{proof}

\bibliography{bib.bib}
\bibliographystyle{plain}

\end{document}